\documentclass[a4,12pt]{amsart}

\usepackage[latin1]{inputenc}
\usepackage{dsfont}
\usepackage{amsxtra}
\usepackage{amsmath}
\usepackage{amscd}
\usepackage{amssymb}
\usepackage{amsfonts}
\usepackage[all]{xy}
\usepackage{mathrsfs}
\usepackage{amsthm} 
\usepackage{color}
\usepackage{upgreek }

\theoremstyle{plain}

\newtheorem{theorem}{Theorem}[section]
\newtheorem{lemma}[theorem]{Lemma}
\newtheorem{claim}[theorem]{Claim}
\newtheorem{proposition}[theorem]{Proposition}
\newtheorem{conj}[theorem]{Conjecture}
\newtheorem{corollary}[theorem]{Corollary}

\theoremstyle{definition}
\newtheorem{definition}[theorem]{Definition}

\theoremstyle{remark}
\newtheorem{remark}[theorem]{Remark}
\newtheorem{example}[theorem]{Example}

\xyoption{all}
\def\bbA{\mathbb{A}}

\def\bbC{\mathbb{C}}

\def\bbI{\mathbb{I}}

\def\bbN{\mathbb{N}}

\def\bbP{\mathbb{P}}
\def\bbQ{\mathbb{Q}}

\def\bbX{\mathbb{X}}

\def\bbZ{\mathbb{Z}}

\def\scrP{\mathscr{P}}

\def\scrW{\mathscr{W}}

\def\frakg{\mathfrak{g}}
\def\frake{\mathfrak{e}}
\def\fraks{\mathfrak{s}}
\def\frakl{\mathfrak{l}}

\def\frakB{\mathfrak{B}}

\def\frakL{\mathfrak{L}}
\def\frakM{\mathfrak{M}}

\def\frakS{\mathfrak{S}}

\def\calA{\mathcal{A}}
\def\calB{\mathcal{B}}
\def\calC{\mathcal{C}}

\def\calE{\mathcal{E}}
\def\calF{\mathcal{F}}

\def\calI{\mathcal{I}}

\def\calL{\mathcal{L}}
\def\calM{\mathcal{M}}

\def\calO{\mathcal{O}}

\def\calV{\mathcal{V}}
\def\calW{\mathcal{W}}

\def\frakL{\mathfrak{L}}

\def\frakg{\mathfrak{g}}
\def\frakh{\mathfrak{h}}
\def\frakl{\mathfrak{l}}

\def\frakn{\mathfrak{n}}

\def\fraksl{\mathfrak{sl}}

\def\bfh{\mathbf{h}}

\def\bfk{\mathbf{k}}
\def\bfh{\mathbf{l}}

\def\bfx{\mathbf{x}}

\def\bfh{\mathbf{h}}

\def\bfY{\mathbf{Y}}
\def\bfZ{\mathbf{Z}}

\def\vep{\varepsilon}
\def\RL{R^\Lambda}
\def\simto{\overset{\sim}\to}

\def\d{{\operatorname{d}\nolimits}}

\def\ch{{\operatorname{ch}\nolimits}}
\def\H{\operatorname{H}\nolimits}
\def\GL{\operatorname{GL}\nolimits}
\def\diag{\operatorname{diag}\nolimits}
\def\Hu{\operatorname{\underline\H}\nolimits}

\def\Ru{{\underline R}}

\def\bfhu{{{\underline\bfh}}}
\def\bfkg{{{\underline\bfk}^\Lambda}}

\def\bfku{{{\underline\bfk}}}

\def\k{{\operatorname{k}\nolimits}}
\def\Mat{{\operatorname{Mat}\nolimits}}
\def\CF{\operatorname{CF}\nolimits}
\def\Z{\operatorname{Z}\nolimits}
\def\W{\operatorname{W}\nolimits}
\def\Wu{\operatorname{\underline\W}\nolimits}

\def\top{\operatorname{top}\nolimits}
\def\soc{\operatorname{soc}\nolimits}

\def\grdim{{\operatorname{grdim}\nolimits}}
\def\min{{\operatorname{min}\nolimits}}
\def\Ker{{\operatorname{Ker}\nolimits}}
\def\add{\operatorname{add}\nolimits}
\def\Hom{\operatorname{Hom}\nolimits}
\def\Tr{\operatorname{Tr}\nolimits}
\def\End{\operatorname{End}\nolimits}

\def\Res{\operatorname{Res}\nolimits}
\def\Ind{\operatorname{Ind}\nolimits}
\def\Ext{\operatorname{Ext}\nolimits}

\def\Ob{\operatorname{Ob}\nolimits}
\def\id{\operatorname{id}\nolimits}

\def\hw{{\operatorname{hw}\nolimits}}
\def\cyc{\operatorname{cyc}\nolimits}
\def\Tr{\operatorname{tr}\nolimits}

\def\Mod{\operatorname{-Mod}\nolimits}
\def\mod{\operatorname{-mod}\nolimits}
\def\grmod{\operatorname{-grmod}\nolimits}
\def\proj{\operatorname{-proj}\nolimits}
\def\grproj{\operatorname{-grproj}\nolimits}

\def\height{\operatorname{ht}\nolimits}
\def\ad{\operatorname{ad}\nolimits}

\author{P. Shan, M. Varagnolo, E. Vasserot}
\email{peng.shan@math.u-psud.fr}
\email{michela.varagnolo@u-cergy.fr}
\email{vasserot@math.jussieu.fr}

\title
[]
{On the center of quiver-Hecke algebras}

\begin{document}
\begin{abstract}
We compute the equivariant cohomology ring of the moduli space of framed instantons over the affine plane.
It is a Rees algebra associated with the center of cyclotomic degenerate affine Hecke algebras of type $A$.
We also give some related results on the center of quiver Hecke algebras and cohomology of quiver varieties.
\end{abstract}
\thanks{This research was partially supported by the ANR grant number ANR-10-BLAN-0110 and
ANR-13-BS01-0001-01.}

\maketitle

\setcounter{tocdepth}{3}

\tableofcontents

\section{Introduction}

A few years ago, the cohomology ring of the Hilbert scheme of points on $\bbC^2$ was computed in \cite{LS}, \cite{V01},
motivated by some conjectures of Chen and Ruan on orbifold cohomology rings, 
which were later proved in \cite{FG}.
One of the main motivation of the present work is to compute a larger class of cohomology ring of quiver varieties.
More precisely, for each pair of positive integers $r,n$, one can consider the moduli space $\frakM(r,n)$ of framed 
instantons with second Chern class $n$ on $\bbP^2$.
This is a smooth quasi-projective variety (over $\bbC$) which can be viewed as a quiver variety attached to the Jordan 
quiver. An $(r+2)$-dimension torus acts naturally on $\frakM(r,n)$ and it contains an $(r+1)$-dimensional subtorus whose 
action preserves the symplectic form.
One of our goal is to compute the equivariant cohomology ring of $\frakM(r,n)$ with respect to this subtorus. 
Since $\frakM(r,n)$ is equivariantly formal, one can easily recover the usual cohomology ring from the equivariant one. 
For $r=1$, i.e., the case of the Hilbert scheme, it can be easily deduced from \cite{V01}
that the equivariant cohomology ring we are interested in is the Rees algebra associated with the center of the group 
algebra of the symmetric group with respect to the age filtration.
Here we obtain a similar description for arbitrary $r$ with the group algebra of the symmetric group replaced by a 
level $r$ cyclotomic quotient of the degenerate affine Hecke algebra of type $A$.

This result has been conjectured soon after \cite{V01} was written.
However, its proof requires two new ingredients which were introduced only very recently.
An important tool used in \cite{V01} is Nakajima's action of an Heisenberg algebra on the cohomology spaces of the 
Hilbert scheme.
A similar action on the cohomology of $\frakM(r)=\bigsqcup_{n\geqslant 0}\frakM(r,n)$ was introduced by Baranovsky a 
few years ago, but it is insufficient to compute the cohomology ring.
What we need in fact is the action of (a degenerate version of) a new algebra $\scrW$ which is much bigger than the 
Heisenberg algebra. This action was
introduced recently in \cite{SV} to give a proof
of the AGT conjecture for pure $N=2$ gauge theories for the group $U(n)$. 
Here, we define a similar action of $\scrW$ on the center (or the concenter) of 
cyclotomic quotients of degenerate affine Hecke algebras. 
Then we compare it with the representation of $\scrW$ on the cohomology of $\frakM(r)$ to obtain the desired 
isomorphism. To do this, we use categorical representation theory.

Categorical representations of Kac-Moody algebras have captured a lot of interest recently since the work of Khovanov-Lauda  \cite{KL09}, \cite{KL2}, \cite{KL3} and Rouquier \cite{R08}.
It has been observed recently by Beliakova-Habiro-Guliyev-Lauda-Zivkovic \cite{BHLZ14}, \cite{BGHL14}, that a categorical representation gives rise to some interesting structures on the center (or cocenter) of the underlying categories and not only on their Grothendieck group.
More precisely, Khovanov-Lauda define
a 2-Kac-Moody algebra in \cite{KL3} which is a 2-category satisfying certain axioms. Their idea is that the trace of this 2-category has a natural structure of an associative algebra $L\frakg$ which should be some kind of loop algebra over the underlying Kac-Moody algebra $\frakg$. The $\fraks\frakl_2$-case has been worked out in \cite{BHLZ14},
and the $\fraks\frakl_n$-one in \cite{BGHL14}. Naturally, the center (or cocenter) of $2$-representations of the $2$-Kac-Moody algebra
gives rise to representations of $L\frakg$.

In this paper we first use a similar idea to investigate the center and cocenter of cyclotomic quiver-Hecke algebras 
associated with a Kac-Moody algebra $\frakg$ of arbitrary type. The category of projective modules over these algebras 
provide minimal categorical representations of $\frakg$, in the sense of Rouquier \cite{R08}.
We compute the representation of $L\frakg$ on the center (or cocenter) of these minimal categorical representations. 
When $\frakg$ is symmetric of finite type, we identify these $L\frakg$-modules with the 
\emph{local and global Weyl modules}, 
which can be realized in the (equivariant) Borel-Moore homology spaces of quiver varieties by \cite{N02}.

Then, in order to compute the cohomology ring of $\frakM(r,n),$
we consider another situation where $\frakg$ is replaced by an Heisenberg algebra. 
On the categorical level this corresponds to the \emph{Heisenberg categorifications}
which have also been studied recently. But once again, instead of considering a $2$-Heisenberg algebra, we focus on 
the particular categorical representation given by the 
module category of degenerate affine Hecke algebras of type $A$.
The analog of $L\frakg$ in this case is the algebra $\scrW$ mentioned above.
Probably one can generalize both situations (the Kac-Moody one and the Heisenberg one) 
using \cite{KKP}, where
categorifications of some generalized (Borcherds-)Kac-Moody algebras are considered. 
Here, we do not go further in this direction.

Another motivation for this work is to compare the 2-Kac-Moody algebras in \cite{KL3} and \cite{R08}.
The definition of Rouquier contains less axioms than the definition of Khovanov-Lauda.
By \cite{K} the module category of cyclotomic quiver-Hecke algebras admits a representation 
of $\frakg$ in the sense of Rouquier.
By \cite{W1} it should also admits a representation of $\frakg$ in Khovanov-Lauda's sense, which is stronger.
According to Cautis-Lauda \cite{CL}, to prove this it is enough to check that the center of the category
is positively graded with a one-dimensional degree zero component. 
These two conditions are difficult to check for minimal categorical representations.
Using the representation of $L\frakg$ we prove that the first condition holds. The second one is more subtle.
It is equivalent to the indecomposability of the weight subcategories of the minimal categorical representations.
In other words, each of these categories should have a single block.
This is well-known in type A and in affine type A by the work of Brundan and Lyle-Mathas in \cite{B08}, \cite{LM}.
We can prove it in some new cases, using the fact that quiver varieties are connected 
(proved by Crawley-Boevey).
But the general case is still unknown.
Note that a third way to prove that minimal categorical representations are indeed representations
of Khovanov-Lauda's 2-Kac-Moody algebra is by a  direct computation, see Remark \ref{rk:KL action} in the appendix.

Now, let us describe more precisely the structure and the main results of the paper.
Fix a symmetrizable Kac-Moody algebra $\frakg$ and a dominant integral weight $\Lambda$ of $\frakg$.

In Section 2 we give some generalities on the centers and cocenters of linear categories.
In Section 3 we introduce the cyclotomic quiver Hecke algebra
of type $\frakg$ and level $\Lambda$ over a field $\k$ (of any characteristic). 
It is a symmetric algebra which decomposes as a direct sum 
$R^\Lambda=\oplus_{\alpha\in Q_+}R^\Lambda(\alpha)$, where $\alpha$ runs
over the positive part of the root lattice of $\frakg$. 
To $\frakg$ we can attach another Lie algebra, $L\frakg$, given by generators and relations. 
It coincides with the loop algebra of $\frakg$ in finite types $ADE$. The first result is the following.

\medskip

{\bf Theorem 1.} {\sl Assume that $\frakg$ is symmetric and that the condition \eqref{Q} is satisfied. Then,

$(a)$  there is a $\bbZ$-graded representation of $L\frakg$ on $\Tr(R^\Lambda)$,

$(b)$   if $\frakg$ is of finite type then $\Tr(R^\Lambda)$ is isomorphic, as a $\bbZ$-graded $L\frakg$-module, to the Weyl module
with highest weight $\Lambda$.
}

\medskip

Note that there are two different notions of Weyl modules for loop Lie algebras used in the literature 
(the local and the global ones).
Both versions can indeed be recovered, see Theorem \label{cor:weyl} below for more details.
Note also that the proof of part $(b)$ involves the geometrical incarnation, given by Nakajima,  
of Weyl modules of $L\frakg$ via the equivariant cohomology of a quiver  variety $\frakM(\Lambda)$ attached to 
$\Lambda.$

\medskip

{\bf Theorem 2.} {\sl For any $\alpha\in Q_+$ the following hold

$(a)$ the trace and the center of $R^\Lambda(\alpha)$ are positively graded,

$(b)$  if $\frakg$ is symmetric of finite type, the dimension of the degree zero subspace of $\Z(R^\Lambda(\alpha))$
is one dimensional.
}

\medskip

The proof uses a reduction to $\fraks\frakl_2$.
Part $(b)$ relies on the geometrical interpretation of Weyl modules. It also uses the non-degenerate pairing between the trace $\Tr(R^\Lambda(\alpha))$
and the center $\Z(R^\Lambda(\alpha))$ of $R^\Lambda(\alpha)$ given by the symmetrizing form.

\smallskip

Finally, in Section 4 we focus on  the Jordan quiver. In this case, instead of the cyclotomic quiver Hecke algebra, we 
consider a level $r$ cyclotomic quotient $R^r(n)$ of the degenerate affine Hecke algebra of $\frakS_n$ defined over 
$\k[\hbar]$. Let $R^r(n)_1$ be its specialization at $\hbar=1$. The center of $R^r(n)_1$ has a natural filtration defined in 
terms of Jucy-Murphy elements. Let $\operatorname{Rees}(\Z(R^r(n)_1))$ be the corresponding Rees algebra.  
Set $R^r=\bigoplus_{n\in\bbN}R^r(n)$ and  let $\Tr(R^r)'$ be a localization of $\Tr(R^r).$ Consider the equivariant 
cohomology   $H_G^*(\frakM(r,n),\k)$   
of the quiver variety $\frakM(r,n)$ relatively to an $(r+1)$-dimensional torus $G$ with coefficient in $\k$. 
Assume that the field $\k$ is of characteristic zero.

\medskip

{\bf Theorem 3.} {\sl  The following hold 

(a) there is a level $r$ representation of $\scrW$ in $\Tr(R^r)'$, 

(b) there is a $\bbZ$-graded algebra isomorphism $\operatorname{Rees}(\Z(R^r(n)_1))\simeq H_G^*(\frakM(r,n),\k).$}

\medskip

The proof of this theorem uses  the representation of $\scrW$ on a localization 
$H_G^*(\frakM(r),\k)'$ introduced in \cite{SV}.

After our paper appeared on the arXiv, A. Lauda informed us that there is some overlap 
between our results and his ongoing projects with collaborators.

\medskip

\section{Generalities}

Let $\k$ be a commutative noetherian ring.

\subsection{The center and the trace of a category}

\subsubsection{Categories}\label{ss:categories}
All categories are assumed to be small.
A \emph{$\k$-linear} category is a
category enriched over the tensor category of $\k$-modules,
a \emph{$\k$-category} is an additive $\k$-linear category.
For any $\k$-linear category $\calC$ and any $\k$-algebra $\k'$, let $\calC':=\k'\otimes_\k\calC$ be the
$\k'$-linear category whose objects are the same as those of $\calC$, but whose morphism spaces are given by
$$\Hom_{\calC'}(a,b)=\k'\otimes_\k\Hom_\calC(a,b)\qquad\forall a,b\in\calC.$$
We denote the identity of an object $a$ by $1_a$ or by $1$ if no confusion is possible.
All the functors $F$ on $\calC$ are assumed to be additive and/or $\k$-linear.
An additive and $\k$-linear functor is called a $\k$-functor.
Let $\End(F)$ be the endomorphism ring of $F$.
We may denote the identity element in $\End(F)$ by $F,1_F$ or $1$, and
the identity functor of $\calC$ by $1_\calC$ or $1$.
The center of $\calC$ is defined as $\Z(\calC)=\End(1_\calC)$.
A composition of functors $E$ and $F$ is written as $EF$
while a composition of morphisms of
functors $y$ and $x$ is written as $y\circ x$.

An additive category $\calC$ will be always equipped with its \emph{trivial} exact structure, i.e., 
the admissible exact sequences
are the split short exact sequences. Therefore, a Serre subcategory $\calI\subset\calC$ is a full additive 
subcategory  which
is stable under taking direct summands, and the quotient additive category 
$\calB=\calC/\calI$ is such that
$$\Hom_{\calB}(a,b)=\Hom_{\calC}(a,b)\,\Big/\sum_{c\in\calI}\Hom_{\calC}(c,b)\circ\Hom_{\calC}(a,c),
\qquad\forall a,b\in\calC.$$
A \emph{short exact sequence} of additive categories is a sequence of functors which is equivalent to
a sequence $0\to\calI\to\calC\to\calB\to 0$ as above.

Fix an integer $\ell$.
By an \emph{$(\ell\bbZ)$-graded} $\k$-category we'll mean a $\k$-category $\calC$ equipped with a strict
$\k$-automorphism $[\ell]$ which we call \emph{shift of the grading}. 
Unless specified otherwise, a functor $F$ of $(\ell\bbZ)$-graded $\k$-categories is 
always assumed to be graded, i.e., it is a $\k$-functor $F$ with an
isomorphism $F\circ[\ell]\simeq [\ell]\circ F$. For each integer $k\in\bbN\cap(\ell\bbZ)$
we'll abbreviate $[k]=[\ell]\circ[\ell]\circ\dots\circ[\ell]$ ($|k/\ell|$ times) and $[-k]=[k]^{-1}$.

Let $\calC/\ell\bbZ$ be the category enriched over the tensor category of $(\ell\bbZ)$-graded
$\k$-modules whose objects are the same as those of $\calC$, 
but whose morphism spaces are given by
$$\Hom_{\calC/\ell\bbZ}(a,b)=\bigoplus_{k\in\ell\bbZ}\Hom_\calC(a,b[k]).$$
Note that the center $\Z(\calC/\ell\bbZ)$ is a graded ring whose degree $k$-component is equal to
$\Hom(1,[k])$.

Given a $\bbZ$-graded $\k$-module $M$ let $M^d=\{x\in M\,;\,\deg(x)=d\}$ for each $d\in\bbZ$.
For any integer $\ell$, the \emph{$\ell$-twist} of $M$ is the $(\ell\bbZ)$-graded
$\k$-module $M^{[\ell]}$ such that $(M^{[\ell]})^{d}=M^{d/\ell}$ if $\ell|d$ and 0 else.
Then, for each $\bbZ$-graded $\k$-category $\calC$ there is a canonical $(\ell\bbZ)$-graded $\k$-category
$\calC^{[\ell]}$ called the \emph{$\ell$-twist} of $\calC$ such that $\calC^{[\ell]}=\calC$ as a $\k$-category
and the shift of the grading $[\ell]$ in $\calC^{[\ell]}$ is the same as the shift of the grading $[1]$ in $\calC$.
We have
$$\Hom_{\calC^{[\ell]}/\ell\bbZ}(a,b)=\Hom_{\calC/\bbZ}(a,b)^{[\ell]},\qquad\forall a,b.$$


Finally, for any category $\calC$ we denote by $\calC^c$ the idempotent completion.

\medskip

\subsubsection{Trace and center}
Let $\calC$ be a  $\k$-linear category and
$HH_*(\calC)$ be the Hochschild homology of $\calC$, see \cite[sec.~3.1]{Ke98a}.
It is a $\bbZ$-graded $\k$-module.
We set $\Tr(\calC)=HH_0(\calC)$ and $\CF(\calC)=\Hom_\k(\Tr(\calC),\k)$. 
We call $\Tr(\calC)$ the \emph{cocenter} or the \emph{trace} of $\calC$
and $\CF(\calC)$ the set of \emph{central forms on $\calC$}.
Recall that 
$$\Tr(\calC)=\Big(\bigoplus_{a\in\Ob(\calC)}\End_\calC(a)\Big)\Big/\sum_{f,g}\k\,[f,g]\quad \text{for\ any} \ f:a\to b,\ g:b\to a.$$
For any morphism $f$ in $\calC$, let $\Tr(f)$ denote its image in $\Tr(\calC)$.

\smallskip

Now, let $A$ be any $\k$-algebra. Unless specified otherwise, all algebras are assumed to be unital.
Let $\Z(A)$ be the center of $A$ and $HH_*(A)$ be its Hochschild homology.
Define $\Tr(A)$ and $\CF(A)$ as above, i.e., $\Tr(A)=A/[A,A]$ where $[A,A]\subset A$ is the 
$\k$-submodule spanned by the commutators of elements of $A$.
For any element $a\in A$ let $\Tr(a)$ denote its image $a+[A,A]$ in $\Tr(A)$.
Let $A$-mod and $A$-proj be the categories of finitely generated modules 
and finitely generated projective modules. 
For any commutative $\k$-algebra $R$ and any $\k$-module $M$
we abbreviate $RM=R\otimes_\k M$.
The following is well-known.
\smallskip

\begin{proposition}\label{prop:trace} 
Let $A,B$ be $\k$-algebras and $\calB,\calC$ be $\k$-linear categories.


(a) If $\calB\subset\calC$ is full and  any object of $\calC$ is isomorphic to a direct summand of a direct sum of objects
of $\calB$, then the embedding $\calB\subset\calC$
yields an isomorphism $\Tr(\calB)\to\Tr(\calC)$.

(b) If $\calC=A\mod$ or $A\proj$ then $\Z(A)=\Z(\calC)$.
If $\calC=A\proj$ then $\Tr(A)=\Tr(\calC).$ 

(c) For any commutative $\k$-algebra $R$ we have $\Tr(RA)=R\Tr(A)$.

(d) We have
$\Tr(A\otimes_\k B)=\Tr(A)\otimes_\k\Tr(B)$ and $\Z(A\otimes_\k B)=\Z(A)\otimes_\k\Z(B)$.


(e) $\Z(\calC)$ acts on $\Tr(\calC)$ via the map
$\Z(\calC)\to\End_\k(\Tr(\calC)),$ $a\mapsto(\Tr(a')\mapsto \Tr(aa'))$.

(f) A short exact sequence of $\k$-categories $0\to\calI\to\calC\to\calB\to 0$
yields an exact sequence of $\k$-linear maps
$\Tr(\calI)\to\Tr(\calC)\to\Tr(\calB)\to0.$

\qed
\end{proposition}


\smallskip

For a future use, let us give some details on part $(f)$. 
Assume that $\calC=\calC^c$.
For any object $X$ let $\add(X)\subset\calC$ 
be the smallest $\k$-subcategory  
containing $X$ which is closed under taking direct summands.
Then, the functor $\Hom_\calC(X,\bullet)$ yields an equivalence 
$\add(X)\to\End_\calC(X)^\text{op}\proj.$
In particular, if $\calC$ has a finite number of indecomposable objects
$X_1,X_2,\dots X_n$ (up to isomorphisms) and $X=\bigoplus_{i=0}^dX_i$, then we have an equivalence
$\calC\simeq\End_\calC(X)^\text{op}\proj$.

Now, assume that $\calC=A\proj,$ where $A$ is a finitely generated $\k$-algebra.
Given a Serre $\k$-subcategory $\calI\subset\calC$, there is an idempotent $e\in A$ such that
$\calI=eAe\proj$ and the functor $\calI\to\calC$ is given by $M\mapsto Ae\otimes_{eAe} M$.
Set $\calB=\calC/\calI$. Then, we have $\calB^c=B\proj$ where $B=A/AeA$ and the composed functor
$\calC\to\calB\to\calB^c$ is given by $M\mapsto B\otimes_A M.$
We must prove that taking the trace we get an exact sequence of $\k$-modules
$\Tr(\calI)\to\Tr(\calC)\to\Tr(\calB)\to 0.$
Equivalently, we must check that the following complex is exact
$$\xymatrix{eAe/[eAe,eAe]\ar[r]^-{i}&A/[A,A]\ar[r]^-{j}&B/[B,B]\ar[r]&0.}$$
Note that $\ker j=(AeA+[A,A])/[A,A]$ and
$\text{im}\, i=(eAe+[A,A])/[A,A].$  Since $aeb=ebae+[ae,eb]$ for all $a,b\in A$, we deduce that
$\ker j=\text{im}\, i$, proving the claim.

\medskip

\subsubsection{Operators on the trace}

\begin{definition}\label{df:trace}
Given a functor $F:\calC\to\calC'$ between two $\k$-categories and a morphism of functors $x\in\End(F)$,
the \emph{trace} of $F$ on $x$ is the linear map
$$\Tr_F(x): \Tr(\calC)\to\Tr(\calC'),\quad \Tr(f)\mapsto \Tr(x(a)\circ F(f))$$
where $f\in\End(a)$ and $x(a)\circ F(f)\in\End(F(a)).$
\end{definition}

\smallskip
Note that $x(a)\circ F(f)=F(f)\circ x(a)$ by functoriality. Below are some basic properties of the trace map, whose proofs are standard and are left to the reader.

\smallskip

\begin{lemma}\label{lem:traceop}
(a) For each $F_1$, $F_2:\calC\to\calC'$, $x\in \End(F_1\oplus F_2)$, we have 
$\Tr_{F_1\oplus F_2}(x)=\Tr_{F_1}(x_{11})+\Tr_{F_2}(x_{22})$
where $x_{11}\in\End(F_1)$, $x_{22}\in\End(F_2)$ are the diagonal coordinates of $x$.

(b) For two morphisms $\rho: F_1\to F_2$, $\psi: F_2\to F_1$, we have $\Tr_{F_1}(\psi\circ\rho)=\Tr_{F_2}(\rho\circ\psi)$. In particular,
if $\rho:F_1\to F_2$ is an isomorphism of functors, then for any $x\in\End(F_1)$ we have
$\Tr_{F_2}(\rho\circ x\circ\rho^{-1})=\Tr_{F_1}(x)$.

(c) For each $F:\calC\to\calC'$, $G:\calC'\to\calC''$, $x\in\End(F)$ and $y\in\End(G),$ we have
$\Tr_{GF}(yx)=\Tr_G(y)\circ\Tr_F(x).$
\qed

\end{lemma}

\medskip

\subsubsection{Adjunction}\label{sss:adjunction}

Given two $\k$-categories $\calC_1$, $\calC_2$, a
\emph{pair of adjoint functors} $(E, F)$ from $\calC_1$ to $\calC_2$ is the datum 
$(E,F,\eta_E,\vep_E)$ of
functors $E:\calC_1\to\calC_2$, $F\!:\calC_2\to\calC_1$ and morphisms of functors 
$\eta_E: 1_{\calC_1}\to FE$ and $\varepsilon_E: EF\to 1_{\calC_2},$ 
called \emph{unit} and \emph{counit}, such that $(\varepsilon_E E)\circ(E\eta_E)=E$ and 
$(F\varepsilon_E)\circ(\eta_EF)=F$. 

A \emph{pair of biadjoint functors} $\calC_1\to\calC_2$ is the datum 
$(E,F,\eta_E,\vep_E,\eta_F,\vep_F)$ of functors 
$E:\calC_1\to\calC_2$, $F:\calC_2\to\calC_1,$ 
morphisms of functors $\eta_E:1_{\calC_1}\to FE$, 
$\varepsilon_E: EF\to 1_{\calC_2}$ such that $(E,F,\eta_E,\vep_E)$
and $(F,E,\eta_F,\vep_F)$ are adjoint pairs. 

\begin{example} Given two pairs of adjoint functors $(E,F)$, $(E',F')$
from $\calC_1$ to $\calC_2$,
the direct sum $(E\oplus E', F\oplus F')$ is an adjoint pair such that
\begin{align*}
&\eta_{E\oplus E'}=(\eta_{E},0,0,\eta_{E'}): 1_{\calC_1}\to F E\oplus F E'\oplus F' E\oplus F'E',\\
&\varepsilon_{E\oplus E'}=\varepsilon_{E}+\varepsilon_{E'}: E F\oplus E F'\oplus E'F\oplus E'F'\to 1_{\calC_2}.
\end{align*}
If $E:\calC_1\to\calC_2$ and $E':\calC_2\to\calC_3$, then
$(E'E, FF')$ is an adjoint pair such that
$\eta_{E'E}=(F\eta_{E'}E)\circ\eta_E$ and $\varepsilon_{E'E}=\varepsilon_{E'}\circ(E'\varepsilon_EF').$
\end{example}

Suppose $(E,F)$, $(E', F')$ are two pairs of adjoint functors from $\calC_1$ to $\calC_2$. 
For any morphism $x: E\to E'$, the \emph{left transpose} ${}^\vee x: F'\to F$ 
is the composition of the chain of morphisms
$$\xymatrix{F'\ar[r]^-{\eta_{E}F'} &FEF'\ar[r]^{FxF'} &
F E'F'\ar[r]^-{F\varepsilon_{E'}} &F.}$$
For any morphism $y: F'\to F$, the \emph{right transpose} $y^\vee: E\to E'$ 
is the composition 
$$\xymatrix{E\ar[r]^-{E\eta_{E'}} &EF'E'\ar[r]^-{Ey E'} &EFE'\ar[r]^-{\varepsilon_{E}E'} &E'.}$$

\medskip

\subsubsection{Operators on the center}

Let $\calC_1$, $\calC_2$ be two $\k$-categories, and $(E,F,\eta_E,\vep_E,\eta_F,\vep_F)$ a pair of biadjoint functors $\calC_1\to\calC_2$. The isomorphisms $1_{\calC_2}E=E=E1_{\calC_1}$ yield a canonical 
$(\Z(\calC_1),\Z(\calC_2))$-bimodule structure on $\End(E)$.
Let $\Z(\calC_2)\to\End(E)$, $z\mapsto zE$ and $\Z(\calC_1)\to\End(E)$, $z\mapsto Ez$ 
denote the corresponding $\k$-algebra homomorphisms.

\begin{definition}[\cite{B}]\label{df:center}
For each $x\in\End(E)$ we define a map 
$$Z_E(x):\Z(\calC_2)\to Z(\calC_1)$$
by sending an element $z\in \Z(\calC_2)$ to the composed morphism
$$\xymatrix{1_{\calC_1}\ar[r]^{\eta_E} &F1_{\calC_2}E \ar[r]^{Fzx} &F1_{\calC_2}E\ar[r]^{\varepsilon_F} &1_{\calC_1}.}$$
We define $Z_F(x): Z(\calC_1)\to Z(\calC_2)$ for each $x\in\End(F)$ in the same manner with the role of $E$ and $F$ exchanged.
\end{definition}

\smallskip

\smallskip

The proof of the following proposition is standard and is left to the reader.

\smallskip

\begin{proposition} \label{prop:bernstein}
Let $(E,F,\eta_E,\vep_E,\eta_F,\vep_F)$, $(E',F',\eta_{E'},\vep_{E'},\eta_{F'},\vep_{F'})$
be two pairs of biadjoint functors. Let $x\in\End(E)$, $x'\in\End(E')$. Then, we have

(a)  $Z_E(x):\Z(\calC_2)\to \Z(\calC_1)$ is $\k$-linear,

(b) 
$Z_{E'E}(x\circ\phi)=Z_{E'}(x')\circ Z_E(x)$
and
$Z_{E\oplus E'}(x\oplus x)=Z_E(x)+Z_{E'}(x')$,

(c) the map $Z_E:\End(E)\to \Hom_\k(\Z(\calC_2),\Z(\calC_1))$ is
$(\Z(\calC_1),\Z(\calC_2))$-bilinear,

(d) let $\rho: E\to E'$ be an isomorphism with
$\rho^\vee={}^\vee\rho,$
then $Z_{E'}(\rho\circ x\circ\rho^{-1})=Z_E(x)$.

\qed
\end{proposition}

\smallskip

\medskip

\subsection{Symmetric algebras}

Let $A,B,C$ be $\k$-algebras.

\subsubsection{Kernels}
There is an equivalence of categories between the category of 
$(A,B)$-bimodules and 
the categories of functors from $B\Mod$ to $A\Mod$. It associates
an $(A,B)$-bimodule $K$ with the functor 
$\Phi_K: B\Mod\to A\Mod$ given by $N\mapsto K\otimes_B N$. 
We say that $K$ is the kernel of $\Phi_K$. Since $\Phi_K(B)=K$, the kernel is uniquely determined by the functor $\Phi_K$. 
For two $(A,B)$-bimodule $K$, $K'$ we have $\Hom_{A,B}(K,K')\simeq \Hom(\Phi_K,\Phi_{K'})$ given by $f\mapsto f\otimes_B\id.$

\medskip

\subsubsection{Induction and restriction} 
We'll call \emph{$B$-algebra} a $\k$-algebra $A$ with a $\k$-algebra homomorphism $i:B\to A$.
We consider the restriction and induction functors
$$\Res^A_B: A\Mod\to B\Mod, \quad \Ind^A_B=A\otimes_B-: B\Mod\to A\Mod.$$
The pair $(\Ind^A_B, \Res^A_B)$ is adjoint with the co-unit $\varepsilon: \Ind^A_B\Res^A_B\to 1$ 
represented by the $(A,A)$-bimodule homomorphism $\mu: A\otimes_BA\to A$ given by the multiplication,
and the unit $\eta: 1\to\Res^A_B\Ind^A_B$ represented by the morphism $i$, which is $(B,B)$-bilinear. 
Let $A^B$ be the centralizer of $B$ in $A$. For any $f\in A^B$ we set
\begin{equation}\label{eq:mu}
\mu_f: A\otimes_BA\to A,\quad a\otimes a'=afa'.
\end{equation}

\medskip

\subsubsection{Frobenius and symmetrizing forms} We refer to \cite{R08} for more details on this section.

Let $A$ be a $B$-algebra that is projective and finite as $B$-module.
A morphism of $(B,B)$-bimodules $t: A\to B$ is called a \emph{Frobenius form} 
if the morphism of $(A,B)$-bimodules
$\hat t: A\to \Hom_B(A,B),$ $a\mapsto (a'\mapsto t(a'a))$
is an isomorphism. If such a form exist, we say that $A$ is a \emph{Frobenius $B$-algebra}.
If we have $t(aa')=t(a'a)$ for each $a\in A$, $a'\in A^B$ then $t$ is called a \emph{symmetrizing form}
and $A$ a \emph{symmetric $B$-algebra}.

Given $t:A\to B$ a Frobenius form, the composition of the isomorphism 
$A\otimes_BA\simto\Hom_B(A,B)\otimes A$ given by 
$a\otimes a'\mapsto \hat t(a)\otimes a'$ and the canonical isomorphism 
$\Hom_B(A,B)\otimes_B A\simto \End_B(A)$ 
yields an isomorphism $A\otimes_BA\simto\End_B(A)$. The preimage of the identity under this map is the 
\emph{Casimir element} $\pi\in (A\otimes_BA)^A$. We have
$(t\otimes 1)(\pi)=(1\otimes t)(\pi)=1.$

There is a bijection between the set of Frobenius forms and the set of adjunctions 
$(\Res^A_B,\Ind^A_B)$ given as follows. Given a 
Frobenius form $t: A\to B$, the counit $\hat\varepsilon: \Res^A_B\Ind^A_B\to 1_B$ is represented by the
$(B,B)$-linear map $t: A\to B$ and the unit
$\hat\eta: 1_A\to \Ind^A_B\Res^A_B$ is represented by the unique $(A,A)$-linear map 
$\hat\eta: A\to A\otimes_BA$ 
such that $\hat\eta(1_A)=\pi$. This yields an adjunction for $(\Res^A_B,\Ind^A_B)$. 
Conversely, if $\hat\varepsilon$ and $\hat\eta$ are counit and unit for $(\Res^A_B,\Ind^A_B)$, 
then the $(B,B)$-linear map $t:A\to B$ which represents $\hat\varepsilon$ is a Frobenius form. 

\smallskip

Recall that $\Tr(A)$ is a $\Z(A)$-module.
We equip $\CF(A)$ with the dual $\Z(A)$-action.
Let us recall a few basic facts.

\begin{proposition} Let $A,B,C$ be $\k$-algebras which are projective and finite as $\k$-modules. 

(a) If $t:A\to B$ and $t': B\to C$ are symmetrizing forms then 
$t'\circ t:A\to C$ is again a symmetrizing form.

(b) A symmetrizing form $t:A\to\k$ induces a nondegenerate 
$\Z(A)$-bilinear form $t:\Z(A)\times\Tr(A)\to\k$
such that $t(a,a')=t(aa')$.

(c)
If $t:A\to\k$ is a symmetrizing form then $\Tr(A)$ is a faithful $\Z(A)$-module and $\CF(A)$
is a free $\Z(A)$-module of rank 1 generated by the obvious map $t:\Tr(A)\to\k$.
\end{proposition}

\begin{proof}
Part $(a)$ is proved in \cite[lem.~2.10]{R08}.
Part $(b)$ is obvious. For $(c)$, assume that $t:A\to\k$ is a symmetrizing form.
If $a,b\in \Z(A)$ have the same image in $\End_\k(\Tr(A))$, then for each $a'\in \Tr(A)$ we have
$t(aa'-ba')=0$, from which we deduce that $a=b$ by part $(b)$. 
For the second claim see, e.g., \cite[lem.2.5]{Broue}.
\end{proof}

\medskip

\section{The center of quiver-Hecke algebras}

\subsection{Quiver Hecke algebras}\label{sec:QH}

Assume that $\bfk=\bigoplus_{n\in\bbN}\bfk^n$ is noetherian and $\bbN$-graded 
and that $\bfk^0$ is a field. 
We may abbreviate $\k=\bfk^0$ and 
we'll identify $\k$ with the quotient $\bfk/\bfk^{>0}$ without mentionning it explicitly.

\subsubsection{Cartan datum}
A \emph{Cartan datum} consists of a finite-rank free abelian group $P$ called the \emph{weight lattice}
whose dual lattice, called the \emph{co-weight lattice}, is denoted $P^\vee$, of a finite set of vectors
$\Phi=\{\alpha_1,\dots,\alpha_n\}\subset P$ called \emph{simple roots} and of a finite set of vectors 
$\Phi^\vee=\{\alpha^\vee_1,\dots,\alpha^\vee_n\}\subset P^\vee$
called \emph{simple coroots}. 
Let $Q_+=\bbN\Phi\subset P$ be the semigroup generated by the simple roots and
$P_+\subset P$ be the subset of dominant weights, i.e., the set of weights $\Lambda$
such that
$\Lambda_i=\langle\alpha^\vee_i,\Lambda\rangle\geqslant 0$ for all $i\in I$. 
We'll call \emph{Bruhat order} the partial order on $P$ such that $\lambda\leqslant\mu$ whenever 
$\mu-\lambda\in Q_+$.

Set $I=\{1,\dots,n\}$ and let $\langle\bullet,\bullet\rangle$ be the canonical pairing on $P^\vee\times P$. 
The $I\times I$ matrix $A$ with entries $a_{ij}=\langle\alpha_i^\vee,\alpha_j\rangle$ is assumed to be a 
\emph{generalized Cartan matrix}.
We'll assume that the Cartan datum is non-degenerate, i.e., the simple roots are linearly independent,
and symmetrizable, i.e., 
there exist non-zero integers $d_i$ such that $d_i\,a_{ij}=d_j\,a_{ji}$ for all $i,j$.
The integers $d_i$ are unique up to an overall common factor. They can be assumed positive and 
mutually prime.

Let $(\bullet|\bullet)$ be the symmetric bilinear form on 
$\frakh^*=\bbQ\otimes_\bbZ P$ given by $(\alpha_i|\alpha_j)=d_ia_{ij}$.
Let $\frakg$ be the symmetrizable Kac-Moody algebra over $\k$ 
associated with the generalized Cartan matrix $A$ and the lattice of integral weights $P$.
Let $\frakh,\frakn^+\subset\frakg$ be the Cartan subalgebra and the maximal nilpotent subalgebra
spanned by the positive root vectors of $\frakg$.
For any dominant weight $\Lambda\in P_+$, let $V(\Lambda)$ be the corresponding
integrable simple $\frakg$-module. For each $\lambda\in P$ let 
$V(\Lambda)_\lambda\subseteq V(\Lambda)$ be the weight subspace of weight $\lambda$.

\medskip

\subsubsection{Quiver Hecke algebras}Fix an element $c_{i,j,p,q}\in\bfk$ 
for each $i,j\in I$, $p,q\in\bbN$ such that $\deg(c_{i,j,p,q})=-2d_i(a_{ij}+p)-2d_jq$ 
and $c_{i,j,-a_{ij},0}$ is invertible.
Fix a matrix $Q=(Q_{ij})_{i,j\in I}$ with entries in $\bfk[u,v]$ such that
$$Q_{ij}(u,v)=Q_{ji}(v,u),\quad Q_{ii}(u,v)=0,\quad
Q_{ij}(u,v)=\sum_{p,q\geqslant 0} c_{i,j,p,q}\,u^pv^q\ \text{if}\ i\neq j.$$

\smallskip

\begin{definition}\label{relKLR}
The \emph{quiver Hecke algebra} (or QHA) of 
rank $n\geqslant 0$ associated with $A$ and $Q$
is the $\bfk$-algebra $R(n;Q,\bfk)$ generated by $e(\nu),$ $x_k$, $\tau_l$ with 
$\nu\in I^n$, $k,l\in[1,n]$, $l\neq n,$ satisfying the
following defining relations
\begin{flalign*}
(a)\ & e(\nu)\,e(\nu')=\delta_{\nu,\nu'}\,e(\nu),\ \ \sum_{\nu}e(\nu)=1,
&\\
(b)\ &x_k\,x_l=x_l\,x_k,\ \ x_k\,e(\nu)=e(\nu)\,x_k,
&\\
(c)\ &\tau_l\,e(\nu)=e(s_l(\nu))\,\tau_l, \ \ \tau_k\,\tau_l=\tau_l\,\tau_k \ \text{if}\ |k-l|>1,
&\\
(d)\ &\tau_l^2\,e(\nu)=Q_{\nu_l,\nu_{l+1}}(x_l,x_{l+1})\,e(\nu),
&\\
(e)\ &(\tau_kx_l-x_{s_k(l)}\tau_k)\,e(\nu)=\delta_{\nu_k,\nu_{k+1}}\,(\delta_{l,k+1}-\delta_{l,k})\,e(\nu),
&\\
(f)\ &(\tau_{k+1}\,\tau_k\,\tau_{k+1}-\tau_k\,\tau_{k+1}\,\tau_k)\,e(\nu)=
\delta_{\nu_k,\nu_{k+2}}\,\partial_{k,k+2}Q_{\nu_k,\nu_{k+1}}(x_k,x_{k+1})\,e(\nu),
\end{flalign*}
where $\partial_{k,l}$ is the Demazure operator on
$\bfk[x_1,x_2,\dots,x_n]$ which is defined by
$$\partial_{k,l}(f)=(f-(k,l)(f)/(x_{k}-x_l).$$

\end{definition}

\smallskip

The algebra $R(n;Q,\bfk)$ is free as a $\bfk$-module. It
admits a $\bbZ$-grading given by $$\deg (e(\nu))=0,\quad 
\deg (x_ke(\nu))=2d_{\nu_k},\quad \deg (\tau_ke(\nu))=-d_{\nu_k}\,a_{\nu_k,\nu_{k+1}}.$$
For $\alpha\in Q_+$ such that $\height(\alpha)=n$, we set
$$I^\alpha=\{\nu=(\nu_1,\dots ,\nu_n)\in I^n\,;\,\alpha_{\nu_1}+\cdots +\alpha_{\nu_n}=\alpha\}.$$
The idempotent
$e(\alpha)=\sum_{\nu\in I^\alpha}e(\nu)$ is central in $R(n;Q,\bfk)$. 
Given $\nu\in I^n$, $\nu'\in I^m$ we write $\nu\nu'\in I^{n+m}$ for their concatenation.
Set
$e(\alpha,\nu')=\sum_{\nu\in I^\alpha}e(\nu\nu')$ and $e(n,\nu')=\sum_{\nu\in I^n}e(\nu\nu').$
The quiver Hecke algebra of rank $\alpha$ is the algebra 
$$R(\alpha;Q,\bfk)=e(\alpha)\,R(n;Q,\bfk)\,e(\alpha).$$

\medskip

\subsubsection{Cyclotomic quiver Hecke algebras}\label{sec:CQHA}
Given a dominant weight $\Lambda\in P_+$ 
we set
$$I_\Lambda=\{(i,p)\,;\,i\in I,\,p=1,\dots,\Lambda_i\}.$$
For a future use, let 
\begin{equation}\label{canonical-map}I_\Lambda\to I,\qquad t\mapsto i_t\end{equation} 
denote the canonical map such that $(i,p)\mapsto i$.
Fix a family of commuting formal variable $\{c_t\,;\,t\in I_\Lambda\}$.
Let $\bfkg$ be the $\bbN$-graded ring given by
$$\bfkg=\k[c_{t}\,;\,t\in I_\Lambda],\qquad\deg(c_{ip})=2pd_i.$$
We'll abbreviate $\bfku=\bfkg$ and we'll write $c_{i0}=1$.

Now, fix a $\bbN$-graded $\bfku$-algebra $\bfk$.
Let $c_t$ denote both the element in $\bfku$ above and its image in $\bfk$ by the 
canonical map $\bfku\to\bfk$ (which is homogeneous of degree 0).
Then, set 
\begin{equation} a^\Lambda_i(u)=\sum_{p=0}^{\Lambda_i}c_{ip}\,u^{\Lambda_i-p}\in\bfk[u].
\end{equation}
The monic polynomial $a^\Lambda_i(u)$ is called the $i$-th \emph{cyclotomic polynomial}
associated with $\bfk$.

For each $\alpha\in Q_+$ and $1\leqslant k\leqslant \height(\alpha)$, we set
\begin{equation}
a^{\Lambda}_\alpha(x_k)=\sum_{\nu\in I^\alpha}a^\Lambda_{\nu_k}(x_k)e(\nu).
\end{equation}
Note that $a^{\Lambda}_\alpha(x_k)e(\nu)$ is a homogeneous element of 
$R(\alpha;Q,\bfk)$ with degree $2d_{\nu_k}\Lambda_{\nu_k}$.

\smallskip

\begin{definition} The \emph{cyclotomic quiver Hecke algebra} (or CQHA)  of rank $\alpha$ and level 
$\Lambda$ is the quotient $R^\Lambda(\alpha; Q,\bfk)$ of
$R(\alpha;Q,\bfk)$ by the two-sided ideal generated by $a^{\Lambda}_\alpha(x_1).$
\end{definition}

\smallskip

To simplify notation, we write 
$R(\alpha)=R(\alpha;\bfk)=R(\alpha; Q,\bfk)$ and 
$R^\Lambda(\alpha)=R^\Lambda(\alpha;\bfk)=R^\Lambda(\alpha; Q,\bfk)$.
We may also write $R=\bigoplus_\alpha R(\alpha)$,  $R(\bfk)=\bigoplus_\alpha R(\alpha;\bfk)$, $R^\Lambda=\bigoplus_\alpha R^\Lambda(\alpha)$, etc.
The following is proved in \cite[cor.~4.4, thm.~4.5]{KK}.

\smallskip

\begin{proposition}\label{prop:projective}
The $\bfk$-algebra $R^\Lambda(\alpha;\bfk)$ is free of finite type as a $\bfk$-module.
\qed
\end{proposition}

\smallskip

\begin{remark}
A morphism of $\bbN$-graded 
$\bfku$-algebras $\bfk\to\bfh$ yields canonical graded $\bfh$-algebra isomorphisms
$\bfh\otimes_\bfk R(\alpha;\bfk)\to R(\alpha;\bfh)$ and
$\bfh\otimes_\bfk R^\Lambda(\alpha;\bfk)\to R^\Lambda(\alpha;\bfh)$.
\end{remark}

\smallskip

\begin{example}\label{ex:base-change}
$(a)$ 
Set
$\Ru^\Lambda(\alpha)=R^\Lambda(\alpha;\bfku).$ 
We call $\Ru^\Lambda(\alpha)$ 
the \emph{global} (or \emph{universal}) CQHA.

\smallskip

$(b)$ If $\bfk=\k$ then $a^\Lambda_i(u)=u^{\Lambda_i}$ for each $i$.
We call $R^\Lambda(\alpha;\k)$ 
the \emph{local} (or \emph{restricted}) CQHA.

\smallskip

$(c)$ 
For each $i\in I$ we fix an element $c_i\in\bfk$ of degree $2d_i$. 
Let $\bfk'$ denote the new $\bfku$-algebra structure on $\bfk$ such that the corresponding
cyclotomic polynomial is
$a^\Lambda_i(u-c_i)$.
Set $Q'_{ij}(u,v)=Q_{ij}(u-c_i,v-c_j).$
Then, the assignment 
$e(\nu), x_ke(\nu), \tau_le(\nu)\mapsto e(\nu), (x_k+c_{\nu_k})e(\nu), \tau_le(\nu)$
extends uniquely to a $\bfk$-algebra isomorphism
$R^\Lambda(\alpha; Q,\bfk)\simto R^\Lambda(\alpha; Q',\bfk').$
In particular, fix $i\in I$ and assume that $\Lambda=\omega_i$ is the $i$-th fundamental weight.
Assume also that condition \eqref{Q} below is satisfied.
Set $a^\Lambda_i(u)=u+c_i.$ 
Then, we have a $\bfk$-algebra isomorphism
\begin{equation}R^{\omega_i}(\alpha; Q,\bfk)\simeq \bfk\otimes_\k R^{\omega_i}(\alpha;Q,\k).\end{equation}
\end{example}

\smallskip

\begin{definition} For each $k\in[1,n-1]$ 
the $k$-th \emph{intertwiner operator} is the element $\varphi_k\in R^\Lambda(n)$ defined by
$\varphi_ke(\nu)=\tau_ke(\nu)$ if $\nu_k\neq\nu_{k+1}$ and by the following formulas 
if $\nu_k=\nu_{k+1}$
$$\begin{aligned}
\varphi_ke(\nu)&=(x_k\tau_k-\tau_kx_k)e(\nu)=(\tau_kx_{k+1}-x_{k+1}\tau_k)e(\nu)\\
&=((x_k-x_{k+1})\tau_k+1)e(\nu)=(\tau_k(x_{k+1}-x_k)-1)e(\nu).
\end{aligned}$$
\end{definition}

\smallskip

We have, see \cite[sec.~5.1]{K} for details,
\begin{itemize}
\item
$x_{s_k(\ell)}\,\varphi_k\,e(\nu)=\varphi_k\,x_\ell\,e(\nu)$,
\item
$\{\varphi_k\}$ satisfies the braid relations,
\item
if $w\in\frakS_n$ satisfies $w(k+1)=w(k)+1$ then $\varphi_w\,\tau_k=\tau_{w(k)}\,\varphi_w$,
\item 
$\varphi_k^2\,e(\nu)=e(\nu)$ if $\nu_k=\nu_{k+1}$ and $\varphi_k^2\,e(\nu)=Q_{\nu_k,\nu_{k+1}}(x_k,x_{k+1})\,e(\nu)$ if $\nu_k\neq\nu_{k+1}$.
\end{itemize}

\medskip

\subsubsection{Induction and restriction}\label{ss:indres}
Let $i\in I$ and $\alpha\in Q_+$ of height $n$.
Set $\lambda=\Lambda-\alpha$ and $\lambda_i=\langle\alpha_i^\vee,\lambda\rangle$. 

We have a $\bbZ$-graded $\bfk$-algebra embedding 
$\iota_i: R(\alpha)\hookrightarrow R(\alpha+\alpha_i)$
given by $e(\nu),x_k,\tau_l\mapsto e(\nu,i),x_k,\tau_l$ for each $\nu\in I^\alpha$
with $1\leqslant k\leqslant n$ and $1\leqslant l\leqslant n-1$.
It induces a $\bbZ$-graded $\bfk$-algebra homomorphism
$\iota_i: R^\Lambda(\alpha)\to R^\Lambda(\alpha+\alpha_i).$

The restriction and induction functors form an adjoint pair $(F'_i,E'_i)$ with
$$\begin{aligned}
&E'_i: R^\Lambda(\alpha+\alpha_i)\grmod\to R^\Lambda(\alpha)\grmod,\quad N\mapsto e(\alpha,i)N,\\
&F'_i: R^\Lambda(\alpha)\grmod\to R^\Lambda(\alpha+\alpha_i)\grmod,\quad 
M\mapsto R^\Lambda(\alpha+\alpha_i)e(\alpha,i)\otimes_{R^\Lambda(\alpha)}M.\\
\end{aligned}$$
The counit $\varepsilon'_{i,\lambda}:F'_iE'_i1_\lambda\to 1_\lambda$ 
and the unit $\eta'_{i,\lambda}: 1_\lambda\to E'_iF'_i1_\lambda$ are
represented respectively by the multiplication map $\mu$ and the map $\iota_i$
$$\varepsilon'_{i,\lambda}:R^\Lambda(\alpha)e(\alpha-\alpha_i,i)
\otimes_{R^\Lambda(\alpha-\alpha_i)}e(\alpha-\alpha_i,i)R^\Lambda(\alpha)
\to R^\Lambda(\alpha),$$
$$\eta'_{i,\lambda}:  R^\Lambda(\alpha)\to e(\alpha,i)R^\Lambda(\alpha+\alpha_i)e(\alpha,i).$$

Finally, let $\sigma'_{ij,\lambda}:F'_iE'_j1_\lambda\to E'_jF'_i1_\lambda$ 
be the morphism represented by the linear map
$$\begin{aligned}
\RL(\alpha-\alpha_j+\alpha_i)e(\alpha-\alpha_j,i)\otimes_{\RL(\alpha-\alpha_j)}
e(\alpha-\alpha_j,j)\RL(\alpha)&\to e(\alpha-\alpha_j+\alpha_i,j)
\RL(\alpha+\alpha_i)e(\alpha,i),\\ x\otimes y&\mapsto x\tau_n y.
\end{aligned}$$
For $j=i$, the element
$\tau_n\in \RL(\alpha+\alpha_i)$ centralizes the subalgebra 
$e(\alpha-\alpha_i,i^2)\RL(\alpha+\alpha_i)e(\alpha-\alpha_i,i^2)$, so we have
$\sigma'_{ii,\lambda}=\mu_{\tau_n}$, see \eqref{eq:mu}.

\smallskip
\begin{theorem}[\cite{KK}]\label{thm:KK}
For each $\alpha\in Q_+$ of height $n$, we have

(a) if $\lambda_i\geqslant 0$, then the following
morphism of endofunctors on $R^\Lambda(\alpha)\Mod$ is an isomorphism
$$
\rho'_{i,\lambda}=\sigma'_{ii,\lambda}+\sum_{k=0}^{\lambda_i-1}(E'_ix^k)\circ\eta'_{i,\lambda}: 
F'_iE'_i1_\lambda\oplus\bigoplus_{k=0}^{\lambda_i-1}\bfk x^k\otimes 1_\lambda\to E'_iF'_i 1_\lambda,$$

(b) if $\lambda_i\leqslant 0$,  then the following morphism of endofunctors on $R^\Lambda(\alpha)\Mod$ 
is an isomorphism
$$
\rho'_{i,\lambda}=\big(\sigma'_{ii,\lambda},\,
\vep'_{i,\lambda}\circ(F'_ix^0),\dots ,\,\vep'_{i,\lambda}\circ(F'_ix^{-\lambda_i-1})\big): 
F'_iE'_i1_\lambda\to E'_iF'_i 1_\lambda\oplus
\bigoplus_{k=0}^{-\lambda_i-1}\bfk (x^{-1})^k\otimes 1_\lambda.$$
\end{theorem} 

\smallskip

The theorem can  be rephrased as follows

\smallskip

\begin{itemize}
\item assume $\lambda_i\geqslant 0$: for any $z\in e(\alpha,i)\RL(\alpha+\alpha_i)e(\alpha,i)$ 
there are unique elements 
$\pi(z)\in \RL(\alpha)e(\alpha-\alpha_i,i)\otimes_{\RL(\alpha-\alpha_i)}e(\alpha-\alpha_i,i)\RL(\alpha)$ 
and $p_k(z)\in\RL(\alpha)$ such that
\begin{equation}\label{eq:lambdapos-bis}
z=\mu_{\tau_n}(\pi(z))+\sum_{k=0}^{\lambda_i-1}p_k(z)\,x^k_{n+1}.
\end{equation}

\medskip
\item  assume $\lambda_i\leqslant 0$: for any $z\in e(\alpha,i)\RL(\alpha+\alpha_i)e(\alpha,i)$
and any $z_0,\dots , z_{-\lambda_i-1}\in\RL(\alpha)$, there is a unique element 
$y\in \RL(\alpha)e(\alpha-\alpha_i,i)\otimes_{\RL(\alpha-\alpha_i)}e(\alpha-\alpha_i,i)\RL(\alpha)$ such that 
\begin{equation}\label{eq:adj-bim-neg}
\mu_{\tau_n}(y)=z,\quad 
\mu_{x^k_{n}}(y)=z_k,\quad \forall k\in[0,-\lambda_i-1].
\end{equation}
\end{itemize}

\smallskip

For a future use, let us introduce the following notation.
Assume that $\lambda_i\leqslant 0$ and that $z\in e(\alpha,i)\RL(\alpha+\alpha_i)e(\alpha,i)$.
For each $\ell\in [0,-\lambda_i-1]$ let
\begin{equation}\label{rk:centralpi}
\tilde z, \tilde\pi_\ell\in \RL(\alpha)e(\alpha-\alpha_i,i)\otimes_{\RL(\alpha-\alpha_i)}
e(\alpha-\alpha_i,i)\RL(\alpha)
\end{equation}
be the unique elements such that 
$$\mu_{\tau_n}(\tilde z)=z,\quad \mu_{x_n^k}(\tilde z)=0,\quad 
\mu_{\tau_n}(\tilde\pi_\ell)=0,\quad\mu_{x_n^k}(\tilde\pi_\ell)=\delta_{k,\ell}.$$

\smallskip

\begin{theorem}[\cite{K}]\label{thm:Kbiadjoint}
The pair $(E'_i,F'_i)$ is adjoint with the counit $\hat\vep'_{i,\lambda}: E'_iF'_i1_\lambda\to 1_\lambda$ and the unit 
$\hat\eta'_{i,\lambda}: 1_\lambda\to F'_iE'_i1_\lambda$ represented by the morphisms
$$\begin{aligned}
&\hat\vep'_{i,\lambda}: e(\alpha,i)\RL(\alpha+\alpha_i)e(\alpha,i)\to\RL(\alpha)\\
&\hat\eta'_{i,\lambda}: \RL(\alpha)\to \RL(\alpha)e(\alpha-\alpha_i,i)\otimes_{\RL(\alpha-\alpha_i)}
e(\alpha-\alpha_i,i)\RL(\alpha)\end{aligned}$$ 
such that
\begin{itemize}
\item
$\hat\vep'_{i,\lambda}(z)=p_{\lambda_i-1}(z)$ if $\lambda_i> 0$ and
$\mu_{x_n^{-\lambda_i}}(\tilde z)$ if $\lambda_i\leqslant 0,$
\item
$\hat\eta'_{i,\lambda}(1)=-\pi(x^{\lambda_i}_{n+1})$ if $\lambda_i\geqslant 0$ and
$\tilde \pi_{-\lambda_i-1}$ if $\lambda_i< 0.$
\end{itemize}

\qed
\end{theorem}

We abbreviate $\vep'_i=\vep'_{i,\lambda}$, $\eta'_i=\eta'_{i,\lambda},$
$\hat\vep'_i=\hat\vep'_{i,\lambda}$, $\hat\eta'_i=\hat\eta'_{i,\lambda}$, etc, 
when $\lambda$ is clear from the context.

\smallskip

\begin{corollary} \label{cor:vepdeg}
The linear maps $\vep'_i$, $\eta'_i$ are homogeneous of degree zero.
The linear maps $\hat\vep'_i$, $\hat\eta'_i$ are homogeneous of degree 
$2d_i(1-\lambda_i)$, $2d_i(1+\lambda_i)$ respectively.
The linear map $\sigma'_{ij}$ is homogeneous of degree $-d_ia_{ij}.$
\qed
\end{corollary}

\medskip

\subsubsection{The symmetrizing form}
For each $\alpha\in Q_+$ we set 
$$d_{\Lambda,\alpha}=(\Lambda|\Lambda)-(\Lambda-\alpha|\Lambda-\alpha).$$
We'll need the following result from \cite[rem.~3.19]{W1}.

\smallskip

\begin{proposition} \label{prop:symmetrizing-form}
The $\bfk$-algebra $R^\Lambda(\alpha)$ is symmetric and admits a symmetrizing form 
$t_{\Lambda,\alpha}$ which is homogeneous of degree $-d_{\Lambda,\alpha}$.
\qed
\end{proposition}

\smallskip

The definition of $t_{\Lambda,\alpha}$ is given in Definition \ref{def:symmetrizing}.
We'll abbreviate $t_\alpha=t_{\Lambda,\alpha}$ and $t_\Lambda=\sum_\alpha t_\alpha$.
Since we have not found any proof of the proposition in the literature, we have given one in Appendix \ref{app:symmetrizing}.

%

\medskip

\subsection{Categorical representations}
Let $\bfk$ be an $\bbN$-graded commutative ring as in Section \ref{sec:QH}.
Write $\frakg_\bfk=\bfk\otimes_\k\frakg$. Fix an integer $\ell$.

\subsubsection{Definition}
For each $\lambda\in P$, let $\calC_\lambda$ be an $(\ell\bbZ)$-graded $\k$-category.
Set $\calC=\bigoplus_\lambda\calC_\lambda$ and denote by $1_\lambda$ the obvious
functor $1_\lambda :\calC\to\calC_\lambda.$
For each $i,j\in I$, $\lambda\in P$ we fix
\begin{itemize}
\item a $\bbZ$-graded $\k$-algebra homomorphism 
$\bfk^{[\ell]}\to\Z(\calC/\ell\bbZ)$, 
\item a functor
$1_{\lambda-\alpha_i}F_i=F_i1_\lambda$ 
with a right adjoint 
$1_\lambda E_i[\ell d_i(1-\lambda_i)]=E_i1_{\lambda-\alpha_i}[\ell d_i(1-\lambda_i)],$
\item morphisms of functors $x_i1_\lambda:F_i1_\lambda\to F_i1_\lambda[2\ell d_i]$ and 
$\tau_{ij}1_\lambda:F_iF_j1_\lambda\to F_jF_i1_\lambda[-\ell d_ia_{ij}].$
\end{itemize}
Thus $\calC/\ell\bbZ$  is a $\bfk$-category and the functors
$F_i1_\lambda$, $E_i1_\lambda$ are $\bfk$-linear.
Let 
$$\vep_i1_\lambda:F_iE_i1_\lambda\to 1_\lambda[\ell d_i(1+\lambda_i)],\qquad 
\eta_i1_\lambda:1_\lambda\to E_iF_i1_\lambda[\ell d_i(1-\lambda_i)]$$ 
be the counit and the unit of the adjoint pair $(1_\lambda F_i,\,E_i1_\lambda[-\ell d_i(1+\lambda_i)])$.
We'll abbreviate 
$$E_i=\bigoplus_\lambda E_i1_\lambda,\quad F_i=\bigoplus_\lambda F_i1_\lambda,\quad 
F_\alpha=\bigoplus_{\nu\in I^\alpha}F_\nu,\quad\text{etc},$$ 
where $F_\nu=F_{\nu_1}F_{\nu_2}\dots F_{\nu_n}$ for $\nu=(\nu_1,\nu_2,\dots,\nu_n)$.
Next, we define the following morphisms
\begin{itemize}
\item
$\sigma_{ij}=(E_jF_i\vep_j)\circ(E_j\tau_{ji}E_i)\circ(\eta_jF_iE_j):F_iE_j\to E_jF_i,$

\item 
$\rho_i1_\lambda=\sigma_{ii}1_\lambda+
\sum_{l=0}^{-\lambda_i-1}(\vep_i1_\lambda)\circ(x_i^lE_i1_\lambda):
F_iE_i1_\lambda
\to E_iF_i1_\lambda\oplus\bigoplus_{l=0}^{-\lambda_i-1}
1_\lambda[\ell d_i(1+2l+\lambda_i)]$
if $\lambda_i\leqslant 0$,

\item 
$\rho_i1_\lambda=
\sigma_{ii}1_\lambda+\sum_{l=0}^{\lambda_i-1}(E_ix_i^l1_\lambda)\circ(\eta_i1_\lambda):
F_iE_i1_\lambda
\oplus\bigoplus_{l=0}^{\lambda_i-1} 1_\lambda[\ell d_i(1+2l-\lambda_i)]\to E_iF_i1_\lambda$
if $\lambda_i\geqslant 0$.

\end{itemize}

\smallskip

\begin{definition}
A \emph{categorical  representation} of $\frakg_\bfk$ of degree $\ell$ in $\calC$ is a tuple
$\calC_\lambda,$ $E_i,$ $F_i,$ $\vep_i,$ $\eta_i,$ $x_i,$ $\tau_{ij}$ as above
such that the following hold
\begin{itemize}
\item the assignment $e(\nu)\mapsto 1_{F_\nu},$
$x_ke(\nu)\mapsto x_{\nu_k}1_{F_\nu},$ $\tau_le(\nu)\mapsto\tau_{\nu_l,\nu_{l+1}}1_{F_\nu}$ for each $\nu\in I^\alpha$
defines a $\bbZ$-graded $\bfk^{[\ell]}$-algebra homomorphism
$R(\alpha;\bfk)^{[\ell]}\to\End_{\calC/\ell \bbZ}(F_\alpha)$,
\item the morphisms $\rho_i1_\lambda$, $\sigma_{ij},$ $i\neq j$, are 
isomorphisms.
\end{itemize}
Morphisms of categorical representations are defined in the obvious way.
\end{definition}

\smallskip

We'll call the map $R(\alpha;\bfk)^{[\ell]}\to\End_{\calC/\ell \bbZ}(F_\alpha)$ the \emph{canonical homomorphism} associated with the categorical 
representation of $\frakg_\bfk$ in $\calC$.

Unless specified otherwise, a categorical representation will be of degree 1.
Note that, given a categorical representation of $\frakg_\bfk$ in $\calC$,
there is a canonical categorical representation of $\frakg_\bfk$ of degree $\ell$ in $\calC^{[\ell]}$
called its \emph{$\ell$-twist} such that the $\bbZ$-graded $\bfk^{[\ell]}$-algebra homomorphism
$$R(\alpha;\bfk)^{[\ell]}\to\End_{\calC^{[\ell]}/\ell\bbZ}(F_\alpha)=\End_{\calC/\bbZ}(F_\alpha)^{[\ell]}$$
is equal to the homomorphism $R(\alpha;\bfk)\to\End_{\calC/\bbZ}(F_\alpha)$
associated with the $\frakg_\bfk$-action on $\calC$.

We'll also use the following definitions 
\begin{itemize}
\item $\calC$ is \emph{integrable} if $E_i$, $F_i$ are locally nilpotent for all $i$, 
\item $\calC$ is \emph{bounded above} if the set of weights of $\calC$ is contained
in a finite union of cones of type $\mu-Q_+$ with  $\mu\in P,$
\item the \emph{highest weight subcategory} $\calC^\hw\subset\calC$ is the full subcategory given by
$$\calC^\hw=\{M\in \calC\,;\,E_i(M)=0,\,\forall i\in I\}.$$
\end{itemize}

\smallskip

\begin{remark}
$(a)$ Taking the left transpose of the morphisms of functors 
$$x_i1_\lambda:F_i1_\lambda\to F_i1_\lambda[2\ell d_i],\quad \tau_{ij}1_\lambda:F_iF_j1_\lambda\to F_jF_i1_\lambda[-\ell d_ia_{ij}]$$ 
we get the morphisms of functors
$$1_\lambda {}^\vee x_i:1_\lambda E_i\to 1_\lambda E_i[2\ell d_i],\quad 1_\lambda{}^\vee \tau_{ij}1_\lambda:1_\lambda E_iE_j\to 1_\lambda E_jE_i[-\ell d_ia_{ij}].$$
We'll abbreviate  $x_i={}^\vee x_i$ and $\tau_{ij}={}^\vee \tau_{ij}$.

\smallskip

$(b)$ Forgetting the grading at each place we define as above a categorical  representation of $\frakg_\bfk$ in a (not graded) $\k$-category $\calC$.

\smallskip

$(c)$ For each short exact sequence of $\bbZ$-graded $\k$-categories $0\to\calI\to\calC\to\calB\to 0$
such that $\calI\to\calC$ is a morphism of categorical representations of $\frakg_\bfk$,
there is a unique categorical representation of $\frakg_\bfk$ on $\calB$ such that
$\calC\to\calB$ is morphism of categorical representations.

\smallskip

$(d)$ Given a categorical representation of $\frakg_\bfk$ on $\calC$, there is a unique categorical representation of
$\frakg_\bfk$ on $\calC^c$ such that the canonical fully faithful functor $\calC\to\calC^c$ is a morphism of categorical representations.
Recall that the objects of the idempotent completion $\calC^c$ are the pairs $(M,e)$ where $M$ is an object of $\calC$
and $e$ is an idempotent of $\End_\calC(M)$, and that 
$\Hom_{\calC^c}\big((M,e),(N,f)\big)=f\Hom_{\calC}(M,N)\,e$.
Then, we have $F_i(M,e)=(F_i(M),F_i(e))$, $E_i(M,e)=(E_i(M),E_i(e))$
and $x_i1_\lambda$, $\tau_{ij}1_\lambda$ are defined in a similar way.

\end{remark}

\medskip

\subsubsection{The minimal categorical representation}
Fix a dominant weight $\Lambda\in P_+$ and an $\bbN$-graded $\bfku$-algebra $\bfk$.
Given $\alpha\in Q_+$ we write $\lambda=\Lambda-\alpha$.
Recall that we abbreviate $R^\Lambda(\alpha)=R^\Lambda(\alpha;\bfk)$.
Let $\bfk\calA^{\Lambda}_{\lambda}=R^\Lambda(\alpha)\grmod$ 
be the $\bbZ$-graded abelian $\k$-category 
consisting of the finitely generated $\bbZ$-graded 
$R^\Lambda(\alpha)$-modules and
let $\bfk\calV^{\Lambda}_{\lambda}=R^\Lambda(\alpha)\grproj$ be the full subcategory
formed by the projective $\bbZ$-graded modules.
When  there is no confusion we'll abbreviate 
$\calA^{\Lambda}_{\lambda}=\bfk\calA^{\Lambda}_{\lambda}$ and 
$\calV^{\Lambda}_{\lambda}=\bfk\calA^{\Lambda}_{\lambda}$.
Let  $\calA^{\Lambda}$, $\calV^{\Lambda}$ be the categories 
$$\calA^{\Lambda}=\bigoplus_\lambda\calA^{\Lambda}_{\lambda},\qquad
\calV^{\Lambda}=\bigoplus_\lambda \calV^{\Lambda}_{\lambda}.$$
Fix an integer $\ell$. Let $\calV^{\Lambda,[\ell]}=(\calV^\Lambda)^{[\ell]}$ be the $\ell$-twist of $\calV^\Lambda$ and
$R^\Lambda(\alpha)^{[\ell]}$ be the $\ell$-twist of $R^\Lambda(\alpha)$. Thus,  $R^\Lambda(\alpha)^{[\ell]}$ is a $(\ell\bbZ)$-graded $\bfk^{[\ell]}$-algebra and
$\calV^{\Lambda,[\ell]}_{\lambda}$ is the category of finitely generated projective $(\ell\bbZ)$-graded modules, i.e., 
$$\calV^{\Lambda,[\ell]}_{\lambda}=R^\Lambda(\alpha)^{[\ell]}\grproj.$$

\smallskip

\begin{definition} The \emph{minimal categorical representation} of $\frakg_\bfk$ of highest weight $\Lambda$
and degree $\ell$ is the representation on
$\calV^{\Lambda,[\ell]}$ given by
\begin{itemize}
\item $E_i1_\lambda=E'_i[-\ell d_i(1+\lambda_i)]$, 
\item $F_i1_\lambda=F'_i$, 
\item $\vep_i1_\lambda=\vep'_{i,\lambda}$ and $\eta_i1_\lambda=\eta'_{i,\lambda}$,
\item $x_i1_\lambda\in\Hom(F_i1_\lambda,F_i1_\lambda[2\ell d_i])$ 
is represented by the right multiplication by
$x_{n+1}$ on $R^\Lambda(\alpha+\alpha_i)e(\alpha,i),$
\item $\tau_{ij}1_\lambda\in\Hom(F_iF_j1_\lambda,F_jF_i1_\lambda[-\ell d_ia_{ij}])$
is represented by the right multiplication by
$\tau_{n+1}$ on $R^\Lambda(\alpha+\alpha_i+\alpha_j)e(\alpha,ji)$.
\end{itemize}
\end{definition}

\smallskip

The category $\calA^{\Lambda,[\ell]}_\lambda$ is Krull-Schmidt 
with a finite number of indecomposable projective objects.
The category $\calV^{\Lambda,[\ell]}/(\ell\bbZ)$ is the category of $(\ell\bbZ)$-graded finitely generated projective $R^{\Lambda,[\ell]}$-modules
with morphisms which are not necessarily homogeneous. We'll call it the category of all
\emph{$(\ell\bbZ)$-gradable} projective modules.

\smallskip

\begin{example}\label{ex:sl2}
We'll abbreviate $\frake=\fraks\frakl_2$.
Assume that $\frakg=\frake,$ $\Lambda=k\omega_1$ and $\alpha=n\alpha_1$
with $k,n\geqslant 0$.
In this case we write $\calV^k=\calV^\Lambda$ and
$\calV^k_{k-2n}=\calV^\Lambda_\lambda.$
Consider the polynomial ring $\Z^k=\k[c_1,\dots,c_k]$ with $\deg(c_p)=2p$ for all $p$.
Let $\Hu^k_n$ be the \emph{global cyclotomic affine nil Hecke algebra} of rank $n$ and level $k$, i.e., 
the $\bbZ$-graded $\Z^k$-algebra denoted by $H_{n,k}$ in \cite[sec.~4.3.2]{R12}.
Note that $\Hu_0^k=\Z^k$ and $\bfku=\Z^k$.
Given an $\bbN$-graded $\Z^k$-algebra $\bfk$, 
the cyclotomic quiver Hecke algebra $R^\Lambda(\alpha;\bfk)$ is isomorphic to
$\bfk\otimes_{\Z^k}\Hu^k_n$ as a $\bbZ$-graded $\bfk$-algebra by \cite[lem.~4.27]{R12}. 
In particular, we have $\Ru^\Lambda(\alpha)=\Hu^k_n$.
For each integer $\ell$, we abbreviate $\Z^{k,[\ell]}=(\Z^k)^{[\ell]}$ and $\Hu^{k,[\ell]}_n=(\Hu^{k}_n)^{[\ell]}$.
We have
\begin{equation}\label{sl2}
\calV^{k,[\ell]}=\bigoplus_{n\geqslant 0}\big(\bfk^{[\ell]}\otimes_{\Z^{k,[\ell]}}\Hu^{k,[\ell]}_n\big)\grproj.
\end{equation}
We'll also identify  $R^\Lambda(\alpha;\k)$ with the \emph{local cyclotomic affine nil Hecke algebra} 
of rank $n$ and level $k$,
which is the quotient $\H^k_n=\k\otimes_{\Z^k} \Hu^k_n$ of $\Hu^k_n$ by the ideal $(c_1,\dots ,c_k)$.
\end{example}

\medskip

\subsubsection{Factorization}\label{sec:factorization}
Fix a dominant weight $\Lambda\in P_+$.
Recall the map $I_\Lambda\to I$ introduced in \eqref{canonical-map}.
For any $t\in I_\Lambda$ we abbreviate $\omega_t=\omega_{i_t}$ and $d_t=d_{i_t}$.
Set $\bfhu=\k[y_t\,;\,t\in I_\Lambda]$ where $y_{t}$ is a formal variable of degree $\deg(y_t)=2d_t$.
The graded ring $\bfhu$ has a natural structure of $\bbN$-graded $\bfku$-algebra such that 
the element $c_{ip}\in\bfhu$ is given by $c_{ip}=e_p(y_{i1},\dots,y_{i\Lambda_i}).$ The corresponding cyclotomic polynomials are 
\begin{equation}\label{a1}a^\Lambda_i(u)=\prod_{p=1}^{\Lambda_i}(u+y_{ip}),\quad\forall i\in I.
\end{equation}
Let $\bfhu'$ be the fraction field of $\bfhu$.
We'll consider the condition
\begin{equation}\label{Q}
Q_{ij}(u,v)=r_{ij}\,(u-v)^{-a_{ij}}\ \text{for some}\ r_{ij}\ \text{such that} \ r_{ij}=(-1)^{a_{ij}}\,r_{ji}\ \text{for all}\ i\neq j.
\end{equation}

\smallskip

\begin{theorem}\label{thm:factorization}
If the condition \eqref{Q} is satisfied, then the following hold

(a) for each integer $n>0$, there is an $\bfhu'$-algebra isomorphism 
$$R^\Lambda(n,\bfhu')\to\bigoplus_{(n_t)}\Mat_{\frakS_n/\prod_t\frakS_{n_t}}\big(
\bigotimes_{t\in I_\Lambda}R^{\omega_t}(n_t,\bfhu')\big),$$
where $(n_t)$ runs over the set of $I_\Lambda$-tuples of non-negative integers with sum $n$,

(b)
there is an isomorphism $\bfhu'\otimes_\bfku\calV^\Lambda\to\bfhu'\otimes_\bfhu\bigotimes_{t\in I_\Lambda}\calV^{\omega_t}$
of (not graded) categorical $\frakg_{\bfhu'}$-representations taking the functor
$F_\alpha$ to $\bigoplus_{(\alpha_t)}\bigotimes_tF_{\alpha_t}$,
the sum being over all $I_\Lambda$-tuples $(\alpha_t)$ with sum $\alpha$.
The canonical homomorphism
$\bigotimes_{t\in I_\Lambda}R(\alpha_t;\bfhu')\to\End\big(\bigotimes_{t\in I_\Lambda}F_{\alpha_t}\big)$
is the composition of the inclusion $\bigotimes_{t\in I_\Lambda}R(\alpha_t;\bfhu')\subset R(\alpha;\bfhu')$ underlying $(a)$
and of the canonical homomorphism
$R(\alpha;\bfhu')\to\End(F_\alpha).$
\end{theorem}

\begin{proof} Let $n=\height(\alpha)$. 
Fix $M\in R^\Lambda(\alpha;\bfhu')\mod$ and $g(u)\in\bfhu'[u]$.
From \cite[p.~715-716]{KK} we get
$$g(x_a)e(\nu)M=0\,\Rightarrow\, Q_{\nu_a,\nu_{a+1}}(x_a,x_{a+1})g(x_{a+1})e(s_a(\nu))M=0, \quad\forall a\in [1,n),\quad\forall\nu\in I^n.$$
Set $Q(u,v)=\prod_{i\neq j}Q_{ij}(u,v)$. We deduce that
\begin{equation}\label{rel-spec}
g(x_a)e(\nu)M=0\,\Rightarrow\, Q(x_a,x_{a+1})g(x_{a+1})e(s_a(\nu))M=0, \quad\forall a\in [1,n).
\end{equation}
Now, assume that the polynomial $Q(u,v)\in\bfhu'[u,v]$ has the following form 
\begin{equation}\label{S}
Q(u,v)=r\prod_{\lambda\in S}(u-v-\lambda)
\end{equation} for some finite family $S$ of elements of $\bfhu'$ and some element $r\in\bfhu'$.
Let $\text{sp}_{e(\nu)M}(x_a)\subset\bfhu'$ be the set of $\lambda\in\bfhu'$ such that the operator $x_a-\lambda\id\in\End_{\bfhu'}(e(\nu)M)$ is not invertible. 
Since $x_a,$ $x_{a+1}$ commute with each other, from \eqref{rel-spec}, \eqref{S} we deduce that 
$$\text{sp}_{e(\nu)M}(x_{a+1})\subseteq\text{sp}_{e(s_a(\nu))M}(x_a)-S\cup\{0\}.$$
Next, recall that $g(x_1)(e(\nu)M)=0$ if $g(u)=a_{\nu_1}^\Lambda(u).$ We deduce that
$$\text{sp}_{e(\nu)M}(x_a)\subseteq\{-y_{\nu_a,p}\,;\,p=1,\dots,\Lambda_a\}-\bbN S,\qquad\forall a\in[1,n].$$
Assume further that the condition \eqref{Q} holds. Then we have $S=\{0\}$, hence
\begin{equation}\label{spectrum}\text{sp}_{e(\nu)M}(x_{a})\subseteq\{-y_{\nu_a,p}\,;\,p=1,\dots,\Lambda_a\},\qquad\forall a\in[1,n].
\end{equation}

In the rest of the proof we write $\tilde I=I_\Lambda$ to simplify the notation.
For each $n$-tuple $\tilde\nu\in \tilde I^n$ set
$$M_{\tilde\nu}=\{m\in M\,;\,(x_k+y_{\tilde\nu_k})^Dm=0,\,\forall k\in[1,n],\,\forall D\gg 0\}.$$
Considering the decomposition of the regular module, we deduce that there is a complete collection of orthogonal idempotents 
$\{e(\tilde\nu)\,;\,\tilde\nu\in \tilde I^n\}$ in $R^\Lambda(\alpha,\bfhu')$  such that
$e(\tilde\nu)M=M_{\tilde\nu}$.
The map $\tilde I\to I$ in \eqref{canonical-map} yields, in the obvious way, a map
\begin{equation} \label{canonical-map2}\tilde I^n\to I^n,\quad\tilde\nu\mapsto\nu.
\end{equation}
The following properties are immediate
\begin{itemize}
\item $e(\tilde\nu)\,e(\nu')=e(\nu')\,e(\tilde\nu)=\delta_{\nu,\nu'}\,e(\tilde\nu)$ for each $\nu'\in I^n,$ 
\item $x_l\,e(\tilde\nu)=e(\tilde\nu)\,x_l$,
\item $\varphi_k\,e(\tilde\nu)=e(s_k(\tilde\nu))\,\varphi_k$,
\item $\tau_k\,e(\tilde\nu)=e(\tilde\nu)\,\tau_k$ if $\tilde\nu_{k}=\tilde\nu_{k+1}$,
\end{itemize}
where $k$, $l$, $\mu,$ $\nu$ run over the sets $[1,n)$, $[1,n]$, $\tilde I^n$ and $I^n$ respectively.
In particular, the idempotents $e(\tilde\nu)$ with $\tilde\nu\in\tilde I^n$ refine the idempotents $e(\nu)$ with $\nu\in I^n$.

\smallskip

\begin{lemma}\label{varphi-invertible}
For each $M\in R^\Lambda(\alpha;\bfhu')\mod$, the map
$\varphi_k\,e(\tilde\nu):e(\tilde\nu)M\to e(s_k(\tilde\nu))M$
is invertible whenever $\tilde\nu_{k}\neq\tilde\nu_{k+1}$.
\end{lemma}

\begin{proof}
The lemma is an immediate consequence of the following relations, for each $\nu\in I^n$,
$$\varphi_r^2\,e(\nu)=1\ \text{if}\ \nu_k=\nu_{k+1},\quad \varphi_k^2\,e(\nu)=Q_{\nu_k,\nu_{k+1}}(x_k,x_{k+1})\,e(\nu)\ \text{if}\ \nu_k\neq\nu_{k+1}.$$

\end{proof}

\smallskip

Set $\tilde Q_{+}=\bbN \tilde I$. Fix an element $\tilde\alpha=\sum_{t\in I_\Lambda}a_t\cdot t$ in $\tilde Q_+$.
Set $\height(\tilde\alpha)=\sum_ta_t$, and assume that $\height(\tilde\alpha)=n$.
The map $\tilde I\to I$ in \eqref{canonical-map} yields a map $\tilde Q_+\to Q_+$.
We'll also assume that $\tilde\alpha$ maps to $\alpha$, i.e., that $\sum_{t\in I_\Lambda}a_t\cdot i_t=\alpha$.

Now, consider the set $\tilde I^{\tilde\alpha}=\{\tilde\nu\in \tilde I^n\,;\,\sum_k\tilde\nu_{k}=\tilde\alpha\}.$
We'll say that two $n$-tuples $\tilde\nu,\tilde\nu'\in \tilde I^n$ are equivalent, and we write $\tilde\nu\sim\tilde\nu'$, if 
we have $\tilde\nu,\tilde\nu'\in\tilde I^{\tilde\alpha}$ for some $\tilde\alpha\in\tilde Q_+$.
Then, we define the idempotent $e(\tilde\alpha)\in R^\Lambda(\alpha,\bfhu')$ by 
$e(\tilde\alpha)=\sum_{\tilde\nu\in \tilde I^{\tilde\alpha}}e(\tilde\nu).$

Next, fix a total order on $\tilde I$ and set
$\tilde I^n_+=\{\mu\in \tilde I^n\,;\,i<j\Rightarrow\mu_i\leqslant \mu_j\}.$
For any tuple $\tilde\nu\in \tilde I^n$ there is a unique element $\tilde\nu^+\in\tilde I^n_+$
such that $\tilde\nu\sim\tilde\nu^+$. 
Let $\tilde\alpha^+$ be the unique element in $\tilde I^n_+\cap\tilde I^{\tilde\alpha}$.
The part $(a)$ of the theorem is a consequence of the following lemma.

\smallskip

\begin{lemma} The following hold

(a) $e(\tilde\nu)\,R^\Lambda(\alpha,\bfhu')\,e(\tilde\nu')=0$ unless $\tilde\nu\sim\tilde\nu'$,

(b) $e(\tilde\alpha)\,R^\Lambda(\alpha,\bfhu')\,e(\tilde\alpha)\simeq
\Mat_{\tilde I^{\tilde\alpha}}\big(e(\tilde\alpha^+)\,R^\Lambda(\alpha,\bfhu')\,e(\tilde\alpha^+)\big)$ as $\bfhu'$-algebras,

(c) $e(\tilde\alpha^+)\,R^\Lambda(\alpha,\bfhu')\,e(\tilde\alpha^+)\simeq\bigotimes_{t\in \tilde I}R^{\omega_t}(a_t,\bfhu')$ as $\bfhu'$-algebras.
\end{lemma}

\begin{proof}
Let $\tilde I^\alpha\subset\tilde I^n$ be the inverse image of $I^\alpha$ by the map \eqref{canonical-map2}.
It is not difficult to prove that
\begin{itemize}
\item
the $\bfhu'$-algebra $R^\Lambda(\alpha,\bfhu')$ is generated by the set of elements 
\begin{equation}\label{generators1}
\{\tau_h\,e(\tilde\nu),\,
\varphi_k\,e(\tilde\nu),\, 
x_l\,e(\tilde\nu)\,;\,h,k\in[1,n),\,l\in[1,n]\,,\tilde\nu\in \tilde I^{\alpha},\,\tilde\nu_{h}=\tilde\nu_{h+1},\,\tilde\nu_{k}\neq\tilde\nu_{k+1}\},
\end{equation}
\item
the $\bfhu'$-algebra
$e(\tilde\alpha^+)\,R^\Lambda(\alpha,\bfhu')\,e(\tilde\alpha^+)$
is generated by the subset
\begin{equation}\label{generators2}
\{\tau_h\,e(\tilde\alpha^+),\,
x_l\,e(\tilde\alpha^+)\,;\,h\in[1,n),\, l\in[1,n],\,
\tilde\alpha^+_h=\tilde\alpha^+_{h+1}\}.
\end{equation}
\end{itemize}
Part $(a)$ is immediate using the generators in \eqref{generators1}. 
Let us concentrate on part $(b)$.
For each tuple $\tilde\nu\in \tilde I^{\tilde\alpha},$ we fix a sequence of simple reflections $s_{l_1},s_{l_2},\dots,s_{l_j}$ in $\frakS_n$ such that
\begin{itemize}
\item
$\tilde\nu=s_{l_1}s_{l_2}\cdots s_{l_j}(\tilde\alpha^+)$,
\item
the $l_h$-th and $(l_h+1)$-th entries of $s_{l_{h+1}}s_{l_{h+2}}\cdots s_{l_j}(\tilde\alpha^+)$
are different for each $h\in[1,j]$.
\end{itemize}
In particular $w=s_{l_1}s_{l_2}\cdots s_{l_j}$ is a reduced decomposition and $w$ is minimal in his left coset in $\frakS_n$ relatively to the stabilizer of
$\tilde\alpha^+$. Hence, we have
\begin{equation}\label{varphi0}\varphi_{l_1}\varphi_{l_2}\cdots\varphi_{l_j}=\varphi_w=\tau_{l_1}\tau_{l_2}\cdots\tau_{l_j}
\end{equation}
and the element
\begin{align}\label{varphi}\pi_{\tilde\nu}=e(\tilde\nu)\,\varphi_{w}\,e(\tilde\alpha^+)\in e(\tilde\nu)R^\Lambda(\alpha,\bfhu')e(\tilde\alpha^+)
\end{align}
depends only on $\tilde\nu$ and not on the choice of $l_1,l_2,\dots,l_j$.
Further, it is invertible by Lemma \ref{varphi-invertible}.
We deduce that there is an $\bfhu'$-algebra isomorphism 
\begin{equation}\label{isom3}
\begin{split}
\begin{aligned}
\Mat_{\tilde I^{\tilde\alpha}}\big(e(\tilde\alpha^+)\,R^\Lambda(\alpha,\bfhu')\,e(\tilde\alpha^+)\big)&\to
e(\tilde\alpha)\,R^\Lambda(\alpha,\bfhu')\,e(\tilde\alpha),\\
E_{\tilde\nu,\,\tilde\nu'}(m)&\mapsto\pi_{\tilde\nu}\, m \,\pi_{\tilde\nu'}^{-1}
\end{aligned}
\end{split}
\end{equation}
Here $E_{\tilde\nu,\,\tilde\nu'}(m)$ is the elementary matrix with a $m$ at the spot $(\tilde\nu,\tilde\nu')$ and 0's elsewhere.
Part $(b)$ is proved.

To prove $(c)$ we must check that the map
$$\bigotimes_{t\in \tilde I}R(a_t,\bfhu')\to e(\tilde\alpha^+)\,R(\alpha,\bfhu')\,e(\tilde\alpha^+),\quad 
a_1\otimes a_2\otimes\cdots\mapsto e(\tilde\alpha^+)\,a_1a_2\cdots\,e(\tilde\alpha^+)$$
factors to an $\bfhu'$-algebra isomorphism
\begin{equation}\label{isom1}\bigotimes_{t\in \tilde I}R^{\omega_t}(a_t,\bfhu')\simto e(\tilde\alpha^+)\,R^\Lambda(\alpha,\bfhu')\,e(\tilde\alpha^+).\end{equation}
To do that, in order to simplify the notation, we'll assume that 
$$\sharp \tilde I=3,\quad\Lambda=2\omega_i+\omega_j,\quad i\neq j\in I.$$
The proof of the general case is very similar.
Write $\tilde I=\{a,b,c\}$ with $a<b<c$ such that the map \eqref{canonical-map} takes $a,b,c$
to $i,i,j$ respectively.
We have $\tilde\alpha^+=(a^{n_a}b^{n_b}c^{n_c})$ with $n_a+n_b+n_c=n$.
To simplify we'll also assume that $n_a,n_b,n_c>0$. 
We have
$$a_i^{\Lambda}(u)=(u+y_a)(u+y_b),\quad a_j^{\Lambda}(u)=u+y_c,\quad a_k^{\Lambda}(u)=1,\quad\forall k\neq i,j.$$

First, we must prove that the following relations hold in $e(\tilde\alpha^+)\,R^\Lambda(n,\bfhu')\,e(\tilde\alpha^+)$
\begin{itemize}
\item
$(x_1+y_a)e(\tilde\alpha^+)=0$ if $\alpha^+_1=i,$ and $e(\tilde\alpha^+)=0$ else,
\item
$(x_{1+n_a}+y_b)e(\tilde\alpha^+)=0$ if $\alpha^+_{1+n_a}=i$, and $e(\tilde\alpha^+)=0$ else,
\item
$(x_{1+n_a+n_b}+y_c)e(\tilde\alpha^+)=0$ if $\alpha^+_{1+n_a}=j,$ and $e(\tilde\alpha^+)=0$ else.
\end{itemize}
The first one is obvious, because 
$$(x_1+y_a)(x_1+y_b)e(\alpha^+)=0\ \text{if}\ \alpha^+_1=i,\quad (x_1+y_c)e(\alpha^+)=0\ \text{if}\ \alpha^+_1=j,\quad e(\alpha^+)=0\ \text{else}$$
and $(x_1+y_b)e(\tilde\alpha^+)$, $(x_1+y_c)e(\tilde\alpha^+)$ are invertible in $e(\tilde\alpha^+)\,R^\Lambda(\alpha,\bfhu')\,e(\tilde\alpha^+)$.
To prove the second relation, note that
$$\varphi_1\varphi_2\cdots\varphi_{n_a}(x_{1+n_a}+y_b)e(\tilde\alpha^+)=(x_{1}+y_b)e(\tilde\mu)\varphi_1\varphi_2\cdots\varphi_{n_a}.$$
where $\tilde\mu=s_1s_2\cdots s_{n_a}(\tilde\alpha^+)$.
Since $\mu_1=\alpha^+_{1+n_a}$, we deduce that
$$\varphi_1\varphi_2\cdots\varphi_{n_a}(x_{1+n_a}+y_b)e(\tilde\alpha^+)=0\ \text{if}\ \alpha^+_{1+n_a}=i,\ \text{and}\ 
\varphi_1\varphi_2\cdots\varphi_{n_a}e(\tilde\alpha^+)=0\ \text{else}.$$
Further, by Lemma \ref{varphi-invertible} the operator 
$$\varphi_1\varphi_2\cdots\varphi_{n_a}e(\tilde\alpha^+):e(\tilde\alpha^+)R^\Lambda(\alpha,\bfhu')\,e(\tilde\alpha^+)\to
e(\tilde\mu)R^\Lambda(\alpha,\bfhu')\,e(\tilde\alpha^+)$$ is invertible.
Thus, we have
$$(x_{1+n_a}+y_b)e(\tilde\alpha^+)=0\ \text{if}\ \alpha^+_{1+n_a}=i,\ \text{and}\ 
e(\tilde\alpha^+)=0\ \text{else},$$
proving the second relation.
The third one is proved in a similar way, using the product of intertwiners $\varphi_{1}\varphi_2\cdots\varphi_{n_a+n_b}$ instead of
$\varphi_{1}\varphi_2\cdots\varphi_{n_a}$.

The relations above imply that the homomorphism \eqref{isom1} is well-defined. We must check that it is invertible.
The surjectivity is immediate using the set of generators of  $e(\tilde\alpha^+)\,R^\Lambda(\alpha,\bfhu')\,e(\tilde\alpha^+)$ in \eqref{generators2}.
Let us concentrate on the injectivity.
It is enough to construct a left inverse to \eqref{isom1}.
To do that, recall that $(a)$, $(b)$ yields an isomorphism
\begin{equation}\label{isom4}
\bigoplus_{\tilde\alpha}\Mat_{\tilde I^{\tilde\alpha}}\big(e(\tilde\alpha^+)\,R^\Lambda(\alpha,\bfhu')\,e(\tilde\alpha^+)\big)\simto
R^\Lambda(\alpha,\bfhu')
\end{equation}
where the sum is over the set of all $\tilde\alpha\in\tilde Q_+$ which map to $\alpha$ by \eqref{canonical-map}.

\smallskip

\begin{claim}
(a) Fix $\tilde\nu\in \tilde I^{\tilde\alpha}$ and $w\in\frakS_n$ as in \eqref{varphi}. 
For each $h,k\in[1,n),$ $l\in[1,n]$ such that $\tilde\nu_{h}=\tilde\nu_{h+1}$ and $\tilde\nu_{k}\neq\tilde\nu_{k+1},$
the map \eqref{isom4} satisfies the following relations
\begin{align*}
E_{\tilde\nu,\tilde\nu}(x_{w^{-1}(l)}\,e(\tilde\alpha^+))\mapsto x_l\,e(\tilde\nu),\quad
E_{s_k(\tilde\nu),\tilde\nu}(e(\tilde\alpha^+))\mapsto\varphi_k\,e(\tilde\nu),\quad
E_{\tilde\nu,\tilde\nu}(\tau_{w^{-1}(h)}\,e(\tilde\alpha^+))\mapsto \tau_h\,e(\tilde\nu).
\end{align*}

(b) The assignment
\begin{align*}
x_l\,e(\tilde\nu)\mapsto E_{\tilde\nu,\tilde\nu}(x_{w^{-1}(l)}),\quad
\varphi_k\,e(\tilde\nu)\mapsto E_{s_k(\tilde\nu),\tilde\nu}(1),\quad
\tau_h\,e(\tilde\nu)\mapsto E_{\tilde\nu,\tilde\nu}(\tau_{w^{-1}(h)}),
\end{align*}
where $h,k,l,\tilde\nu$ are as above,
yields an $\bfhu'$-algebra homomorphism
\begin{equation}\label{isom5}
R^\Lambda(\alpha,\bfhu')\to\bigoplus_{\tilde\alpha}\Mat_{\tilde I^{\tilde\alpha}}\big(
\bigotimes_{t\in \tilde I}R^{\omega_t}(a_t,\bfhu')\big).
\end{equation}
\end{claim}

\begin{proof} Part $(a)$ follows from \eqref{isom3} and from the following computations
\begin{align*}
x_l\,e(\tilde\nu)&=x_l\,\pi_{\tilde\nu}\,e(\tilde\alpha^+)\,\pi_{\tilde\nu}^{-1}\\
&=\pi_{\tilde\nu}\,x_{w^{-1}(l)}\,e(\tilde\alpha^+)\,\pi_{\tilde\nu}^{-1}\\
\varphi_k\,e(\tilde\nu)&=\varphi_k\,\pi_{\tilde\nu}\,e(\tilde\alpha^+)\,\pi_{\tilde\nu}^{-1}\\
&=\varphi_k\,e(\tilde\nu)\,\varphi_w\,e(\tilde\alpha^+)\,\pi_{\tilde\nu}^{-1}\\
&=e(s_k(\tilde\nu))\,\varphi_k\,\varphi_w\,e(\tilde\alpha^+)\,\pi_{\tilde\nu}^{-1}\\
&=\pi_{s_k(\tilde\nu)}\,e(\tilde\alpha^+)\,\pi_{\tilde\nu}^{-1}\\
\tau_h\,e(\tilde\nu)&=\tau_h\,\pi_{\tilde\nu}\,e(\tilde\alpha^+)\,\pi_{\tilde\nu}^{-1}\\
&=\tau_h\,e(\tilde\nu)\,\varphi_w\,e(\tilde\alpha^+)\,\pi_{\tilde\nu}^{-1}\\
&=e(\tilde\nu)\,\tau_h\,\varphi_w\,e(\tilde\alpha^+)\,\pi_{\tilde\nu}^{-1}\\
&=e(\tilde\nu)\,\varphi_w\,\tau_{w^{-1}(h)}\,e(\tilde\alpha^+)\,\pi_{\tilde\nu}^{-1}\\
&=\pi_{\tilde\nu}\,\tau_{w^{-1}(h)}\,e(\tilde\alpha^+)\,\pi_{\tilde\nu}^{-1}.
\end{align*}
Here, we have used the equality $w^{-1}(h+1)=w^{-1}(h)+1$, which follows from the definition of $w$ in \eqref{varphi0},
and the equalities $s_h(\tilde\nu)=\tilde\nu$ and $s_{w^{-1}(h)}(\tilde\alpha^+)=\tilde\alpha^+,$ which follow from the identities $\tilde\nu_h=\tilde\nu_{h+1}$
and $\tilde\nu=w(\tilde\alpha^+)$.

Now, let us concentrate on $(b)$.
Since the elements $x_l\,e(\tilde\nu),$ $\varphi_k\,e(\tilde\nu),$ $\tau_h\,e(\tilde\nu)$ above generate $R^\Lambda(\alpha,\bfhu')$
by \eqref{generators1},  it is enough to check that the defining relations of $R^\Lambda(\alpha,\bfhu')$ given in Section \ref{sec:CQHA}
are satisfied. This is obvious.
Note that the element $\tau_{w^{-1}(h)}$ belongs indeed to $\bigotimes_{t\in \tilde I}R^{\omega_t}(a_t,\bfhu')$ because
$s_{w^{-1}(h)}(\tilde\alpha^+)=\tilde\alpha^+.$ 

\end{proof}

\smallskip

Composing \eqref{isom4} and \eqref{isom5}, we get an $\bfhu'$-algebra homomorphism
$$\bigoplus_{\tilde\alpha}\Mat_{\tilde I^{\tilde\alpha}}\big(e(\tilde\alpha^+)\,R^\Lambda(\alpha,\bfhu')\,e(\tilde\alpha^+)\big)
\to\bigoplus_{\tilde\alpha}\Mat_{\tilde I^{\tilde\alpha}}\big(\bigotimes_{t\in \tilde I}R^{\omega_t}(a_t,\bfhu')\big).$$
For each $\tilde\alpha$ it restricts to an $\bfhu'$-algebra homomorphism
$$e(\tilde\alpha^+)\,R^\Lambda(\alpha,\bfhu')\,e(\tilde\alpha^+)
\to\bigotimes_{t\in \tilde I}R^{\omega_t}(a_t,\bfhu')$$
which is a left inverse to \eqref{isom1}.

\end{proof}

\smallskip

Finally, the part $(b)$ of the theorem follows easily from the isomorphism in $(a)$.

\end{proof}

\medskip

\subsubsection{The isotypic filtration}
Recall that $\frake=\mathfrak{sl}_2$.
A \emph{$(P\times\bbZ)$-graded} $\k$-category $\calC$ is a direct sum of categories of the form
$\calC=\bigoplus_{\lambda\in P}\calC_\lambda$ where each $\calC_\lambda$ 
is a $\bbZ$-graded $\k$-category.
We'll call $\lambda$ the \emph{P-degree} of $\calC_\lambda$.

Given $i\in I$ there is an $\fraks\frakl_2$-triple $\frake_i\subset\frakg$. In this section, we study the restriction of a 
categorical $\frakg_\bfk$-representation to such an $\fraks\frakl_2$-triple. 
To simplify the notation, in this section we fix an element $i\in I$ and identify $\frake=\frake_i$.

By a \emph{$(P\times\bbZ)$-graded} categorical $\frake_\bfk$-representation  on $\calC$
we'll mean a representation such that for each $\lambda\in P$ the
$\bfk$-subcategory $\calC_\lambda$ has the weight 
$\lambda_i$ relatively to the $\frake$-action.
In particular, if $\calC$ is an integrable categorical representation of $\frakg_\bfk$, restricting the 
$\frakg$-action on $\calC$ to $\frake_i$ yields
an integrable $P\times\bbZ$-graded categorical $\frake$-representation on $\calC$ of degree $d_i$. 

Now, fix an integer $k\in\bbN$.
Given a $(P\times\bbZ)$-graded $\k$-category $\calM$ such that 
$\calM_\mu=0$ whenever $\mu_i\neq k$ and
a $\bbZ$-graded
$\k$-algebra homomorphism $\Z^{k,[d_i]}\to\Z(\calM/\bbZ)$,
we equip the tensor product $\calV^{k,[d_i]}\otimes_{\Z^{k,[d_i]}}\calM$ with the
$(P\times\bbZ)$-graded categorical $\frake$-representation of degree $d_i$ such that
\begin{itemize}
\item
$\frake$ acts on the left factor,
\item the summand $\calV^{k,[d_i]}_n\otimes_{\Z^{k,[d_i]}}\calM_\mu$
has the $P$-degree $\mu-n\alpha_i$ for each $n\in\bbN$.
\end{itemize}

\smallskip

\begin{proposition}\label{prop:isotypic}
Fix an integrable  categorical representation of $\frakg_\bfk$ on $\calC$ of degree 1
which is bounded above. For each $i\in I$ there is a decreasing filtration 
$\dots\subseteq\calC_{\geqslant_1}\subseteq\calC_{\geqslant 0}=\calC$ of $\calC$
by full $(P\times\bbZ)$-graded integrable categorical
$\frake$-representations of degree $d_i$ which are closed under taking direct summands and such that 
\begin{itemize}
\item for each $k\geqslant 0$ there is a $P\times\bbZ$-graded $\k$-category $\calM_k$ with a $\bbZ$-graded 
$\k$-algebra homomorphism $\bfk\otimes_\k\Z^{k,[d_i]}\to\Z(\calM_k/\bbZ)$ and 
a $\bfk$-linear equivalence of $P\times\bbZ$-graded categorical $\frake_{\Z^k}$-representations 
$(\calC_{\geqslant k}/\calC_{>k})^c\simeq(\calV^{k,[d_i]}\otimes_{\Z^{k,[d_i]}}\calM_k)^c$ of degree $d_i$,
\item for each $\lambda\in P$ we have
$\calC_{\geqslant k}\cap\calC_\lambda=0$ for $k$ large enough.
\end{itemize}
\end{proposition}

\begin{proof}
Since $\calC$ is bounded above,
by a standard argument
we may assume that the set of weights of $\calC$ is 
contained in a  cone $\mu-Q_+$ for some $\mu\in P,$
see e.g., \cite[lem.~2.1.10]{Ku}.
Next, for each coset $\pi\in P/\bbZ\alpha_i$, the $\frakg$-action on $\calC$
yields a $P\times\bbZ$-graded categorical $\frake$-representation on 
$\calC_\pi=\bigoplus_{\lambda\in\pi}\calC_\lambda$
of degree $d_i$. 
Since  $\calC$ decomposes as the direct sum of $\k$-categories
$\calC=\bigoplus_\pi\calC_\pi$, it is enough to define a filtration of $\calC_\pi$ satisfying the properties 
above for each $\pi$.
Note that $\pi\cap(\mu-Q_+)$ is a cone of the form $\nu-\bbN\alpha_i$ for some weight $\nu\in\mu-Q_+$.
Thus the claim follows from \cite[thm.~4.22]{R12}, applied to the 
$\frake$-representation on 
$\calC_\pi$.

\end{proof}

\smallskip

We call $(\calC_{\geqslant k})$ the  \emph{$i$-th isotypic filtration} of $\calC$. 
For each $k\in\bbN$, $\lambda\in P$ let $\calC_{\geqslant k,\,\lambda}\subset\calC_{\geqslant k}$
be the weight $\lambda$ subcategory given by
$\calC_{\geqslant k,\,\lambda}=\calC_{\geqslant k}\cap\calC_\lambda.$
We also define
\begin{equation}\label{iso}\calC_k=(\calV^{k,[d_i]}\otimes_{\Z^{k,[d_i]}}\calM_k)^c,\qquad
\calC_{k,\,\lambda}=\bigoplus_{n,\mu}
(\calV^{k,[d_i]}_{k-2n}\otimes_{\Z^{k,[d_i]}}\calM_{k,\mu})^c,
\end{equation}
where the sum runs over all $\mu\in P$ and $n\in\bbN$ 
such that $\lambda=\mu-n\alpha_i$.
Recall that we have assumed that $\calM_{k,\mu}=0$ whenever $\mu_i\neq k$.
Note also that $\calC_k^\hw=\calM_k^c$, and that $\calC_k^\hw$ is the full subcategory
of $(\calC/\calC_{>k})^c$ consisting of the objects $M$ such that $E_i(M)=0$.

\medskip

\subsubsection{The loop operators} \label{sec:loop-operators}
Consider a categorical representation of $\frakg_\bfk$ of degree $\ell$
on a $\bbZ$-graded $\k$-category $\calC$. 
For each $i\in I$ and $r\in\bbN$, we define $\bfk$-linear operators 
$x_{ir}^\pm$ on the $\bbZ$-graded $\bfk$-module $\Tr(\calC/\bbZ)$ such that the maps
$$x_{ir}^+:  \Tr(\calC_\lambda/\bbZ)\to\Tr(\calC_{\lambda-\alpha_i}/\bbZ),\qquad
x_{ir}^-:  \Tr(\calC_\lambda/\bbZ)\to\Tr(\calC_{\lambda+\alpha_i}/\bbZ), $$
are given by
$$x_{ir}^+=\Tr_{E_i}(x_i^r),\quad
x_{ir}^-=\Tr_{F_i}(x_i^r),$$
see Definition \ref{df:trace} for the notation.
The map $x^\pm_{ir}$ is homogeneous of degree $2r\ell d_i$.
Let us quote the following fact for future use.

\smallskip

\begin{proposition}\label{lem:functoriality} 
The maps $x_{ir}^\pm$ are functorial, i.e., a morphism $\calC\to\calC'$ 
of categorical representations of $\frakg_\bfk$ yields a $\bfk$-linear map
$\Tr(\calC/\bbZ)\to\Tr(\calC'/\bbZ)$ which intertwines the operators $x_{ir}^\pm$
on $\Tr(\calC/\bbZ)$ and on $\Tr(\calC'/\bbZ)$.

\qed
\end{proposition}

\medskip

\subsubsection{The loop operators on the trace of the minimal categorical representation}
Fix a dominant weight $\Lambda\in P_+$.
For $\alpha\in Q_+$ we write $\lambda=\Lambda-\alpha$.
Recall the base ring $\bfku$ from Section \ref{sec:CQHA}.
Let $\bfk$ be an $\bbN$-graded $\bfku$-algebra.
In this section we consider the particular case of the minimal categorical representation $\calV^\Lambda$.

\begin{definition}
Let $L\frakg$ be the $\bbN$-graded Lie $\k$-algebra generated by elements
$x_{ir}^\pm$, $h_{ir}$ with $i\in I$, $r\in\bbN$ satisfying the following relations
\begin{itemize}
\item[$(a)$] $[h_{ir},h_{js}]=0$,
\item[$(b)$] $[x_{ir}^+,x_{js}^-]=\delta_{ij}\,h_{i,r+s}$,
\item[$(c)$] $[h_{ir},x_{js}^\pm]=\pm a_{ij}\,x^\pm_{j,r+s}$,
\item[$(d)$] $\sum_{p=0}^{m}(-1)^p(\begin{smallmatrix}m\cr p
\end{smallmatrix})[x_{i,r+p}^\pm,x_{j,s+m-p}^\pm]=0$ with $i\neq j$ and $m=-a_{ij}$,
\item[$(e)$] $[x^\pm_{i,r},x^\pm_{i,s}]=0,$
\item[$(f)$] $[x^\pm_{i,r_{1}},[x^\pm_{i,r_{2}},\dots [x^\pm_{i,r_{m}},x^\pm_{j,r_0}]\dots]]=0$\ with $i\neq j$, $r_p\in\bbN$ and $m=1-a_{ij}$.
\end{itemize}
The grading is given by $\deg(x_{ir}^\pm)=\deg(h_{ir})=2r$.
\end{definition}

\smallskip

There is a unique Lie $\k$-algebra anti-involution $\varpi$ of $L\frakg$ such that
$$\varpi(h_{ir})=h_{ir},\qquad \varpi(x_{ir}^\pm)=x_{ir}^\mp.$$

\smallskip

We define the \emph{$\ell$-twist} of $L\frakg$ to be the Lie $\k$-subalgebra $L\frakg^{[\ell]}\subset L\frakg$
generated by the set of elements $\{x_{i,r\ell}^\pm,\,h_{i,r\ell}\,;\,i\in I\,,r\in\bbN\}$. We write
$L\frakg^{[\ell]}_\bfk=\bfk^{[\ell]}\otimes_\k L\frakg^{[\ell]}.$

\smallskip

\begin{theorem}\label{thm:action} If $\frakg$ is symmetric and \eqref{Q} is satisfied,
then the operators $x^\pm_{i,\ell r}$with  $i\in I$, $r\in\bbN$ define a $\bbZ$-graded representation of
$L\frakg_\bfk^{[\ell]}$ on $\Tr(\calV^{\Lambda,[\ell]}/\bbZ).$
\end{theorem}

\begin{proof} It is enough to set $\ell=1$. 
The proof is similar to a proof in \cite{BHLZ14}, \cite{BGHL14}.
However, our setting differs from loc. cit. and cannot be reduced to it, because, in our case, $\frakg$ may have any symmetric type
and because we do not require that all axioms of a strong categorical representation to be satisfied by $\calV^\Lambda$.
We have written
an independent proof in Appendix \ref{app:symmetrizing}.

\end{proof}

\smallskip

\begin{remark}
$(a)$ If $\frakg$ is not symmetric or \eqref{Q} is not satisfied, then a more general version of the theorem is given in Proposition \ref{prop:B5} below.

\smallskip

$(b)$ Assume that $\frakg$ is symmetric.
If $\frakg$ is \emph{simply laced}, then the relations $(d)$ and $(e)$ are equivalent to the following :
the bracket $[x_{i,r}^\pm,x_{j,s}^\pm]$  depends only on the sum $r+s$ and not on the integers $r$ or $s$.
If $\frakg$ is of finite type then $L\frakg$ is the \emph{current algebra} $\frakg[t]$ with $\deg(t)=2$.
In general $L\frakg$ is \emph{bigger} than $\frakg[t]$. For instance if $\frakg$ is not of type $A_1^{(1)}$ it contains
the center of the universal central extension of $\frakg[t]$, which is infinite dimensional. 
See \cite[sec.~3]{MRY}, \cite[sec.~13]{VV98} and \cite[sec.~1.3]{E} for more details.
\end{remark}

\medskip

\subsubsection{The loop operators on the center.}
Since the functors $E_i$, $F_i$ on $\calV^\Lambda$ are biadjoint, we also have operators 
$$Z^+_{ir}: Z(\calV^\Lambda_\lambda/\bbZ)\to Z(\calV^\Lambda_{\lambda+\alpha_i}/\bbZ),\quad Z^-_{ir}: Z(\calV^\Lambda_\lambda/\bbZ)\to Z(\calV^\Lambda_{\lambda-\alpha_i}/\bbZ)$$
defined by
$$Z_{ir}^+=Z_{F_i}(x_i^r),\quad Z_{ir}^-=Z_{E_i}(x_i^r),$$
see Definition \ref{df:center} for the notation.

The proposition below gives some more explicite description of the operators $x_{ir}^\pm$ and $Z^\pm_{ir}$ in terms of the algebras $\RL(\alpha)$.
Fix $v_k$, $v^\vee_k$ such that
$$\hat\eta'_{i,\lambda}(1)=\sum_kv_k^\vee\otimes v_k,\quad
v_k^\vee\in\RL(\alpha)e(\alpha-\alpha_i,i),\quad v_k\in e(\alpha-\alpha_i,i)\RL(\alpha).$$

\begin{proposition}\label{prop:loop-minimal}
For each $\alpha\in Q_+$ of height $n$ and each $f\in R^\Lambda(\alpha)$, $g\in \Z(R^\Lambda(\alpha)),$
we have 

(a) $\Tr(\calV^\Lambda/\bbZ)=\Tr(R^\Lambda)$,
$\Z(\calV^\Lambda/\bbZ)=\Z(R^\Lambda)$ as a $\bbZ$-graded $\bfk$-module and $\bfk$-algebra respectively,

(b) $x^-_{ir}(\Tr(f))=\Tr(e(\alpha,i)\,x_{n+1}^r\,f)\in \Tr(R^\Lambda(\alpha+\alpha_i))$, 

(c) $x^+_{ir}(\Tr(f))=\sum_k\hat\vep'_{i,\lambda+\alpha_i}(x_n^rv_kfv_k^\vee)\in\Tr(\RL(\alpha-\alpha_i))$,

(d) $Z_{ir}^+(g)=\hat\vep'_{i,\lambda+\alpha_i}(x_{n}^r\,g\,e(\alpha-\alpha_i,i))
\in \Z(R^\Lambda(\alpha-\alpha_i))$,

(e) $Z_{ir}^-(g)=\mu((1\otimes x_{n+1}^r\,g)\,\hat\eta'_{i,\lambda-\alpha_i}(1))
\in \Z(R^\Lambda(\alpha+\alpha_i)).$
\end{proposition}

The proof is standard and left to the reader. We only make a few comments.

First, the element $f$ in $(b)$ is 
viewed as an element of $R^\Lambda(\alpha+\alpha_i)$ via the 
map $\iota_i:R^\Lambda(\alpha)\to R^\Lambda(\alpha+\alpha_i).$
Hence $e(\alpha,i)x_{n+1}^rf$ belongs to $R^\Lambda(\alpha+\alpha_i)$ and 
$x^-_{ir}(f)$ to $\Tr(R^\Lambda(\alpha+\alpha_i)).$

Next, the equality $(c)$ should be interpreted in the following way:
$f$ is identified with the $\RL(\alpha)$-module endomorphism of
$\RL(\alpha)$ given by $m\mapsto mf,$ and
$E_i$ is represented by the bimodule $e(\alpha-\alpha_i,i)\RL(\alpha)$ which
is a projective $\RL(\alpha-\alpha_i)$-module by \cite[thm.~4.5]{KK}. 
Then $x^+_{ir}(\Tr(f))$ is the class in $\Tr(\RL(\alpha-\alpha_i)\proj)\simeq\Tr(\RL(\alpha-\alpha_i))$
of the endomorphism
$$\varphi:e(\alpha-\alpha_i,i)\RL(\alpha)\to e(\alpha-\alpha_i,i)\RL(\alpha),\quad m\mapsto x_n^rmf.$$
Explicitly, 
since $(\hat\vep'_iE_i)\circ(E_i\hat\eta'_i)=E_i$, we have
$m=\sum_k\hat\vep'_i(mv_k^\vee)v_k$.
Hence there is a surjective $\RL(\alpha-\alpha_i)$-module morphism
$$p:\bigoplus_k\RL(\alpha-\alpha_i)\to e(\alpha-\alpha_i,i)\RL(\alpha),\quad 
(m_k)_k\mapsto\sum_km_k v_k,$$
with a splitting $i$ given by 
$$i:e(\alpha-\alpha_i,i)\RL(\alpha)\to\bigoplus_k\RL(\alpha-\alpha_i),\quad m\mapsto(\hat\vep'_i(mv_k^\vee))_k,$$
i.e., we have $p\circ i=1$. Consider the endomorphism $\tilde\varphi$ given by 
$$\tilde\varphi=i\circ\varphi\circ p: \bigoplus_k\RL(\alpha-\alpha_i)\to \bigoplus_k\RL(\alpha-\alpha_i),
\quad (m_k)_k\mapsto 
\big(\hat\vep'_i(x_n^r\sum_lm_lv_lfv_k^\vee)\big)_k.$$
Since $m_l\in \RL(\alpha-\alpha_i)$, it commutes with $x_n^r$.
Since $\hat\vep'_i$ is $\RL(\alpha-\alpha_i)$-linear, we deduce that 
$\hat\vep'_i(x_n^r\sum_lm_lv_lfv_k^\vee)=\sum_lm_l\hat\vep'_i(x_n^rv_lfv_k^\vee)$. 
So $\tilde\varphi$ is the right multiplication by the matrix $(\hat\vep'_i(x_n^rv_lfv_k^\vee))_{l,k}$.
Therefore we have 
$$\begin{aligned}
\Tr(\varphi)=\Tr(\varphi\circ p\circ i)=\Tr(\tilde\varphi)
=\sum_k\hat\vep'_i(x_n^rv_kfv_k^\vee).
\end{aligned}$$
We obtain
$x^+_{ir}(\Tr(f))=\sum_k\hat\vep'_i(x_n^rv_kfv_k^\vee).$

Finally, in parts $(d)$, $(e)$ we have $Z_{ir}^+(g)=Z_{F_i}(x_{n}^rg)$
and $Z_{ir}^-(g)=Z_{F_i}(x_{n+1}^rg)$.
Note that $e(\alpha-\alpha_i,i)=\eta'_{i,\lambda+\alpha_i}(1)$ and
$\mu=\vep'_{i,\lambda-\alpha_i}$.

\qed
\smallskip

Let $r(\alpha,i)$ be as in \eqref{tr}. The following proposition relates the operator $x_{ir}^\pm$ on $\Tr(\calV^\Lambda/\bbZ)$ with the transpose of the operator
$Z^\mp_{ir}$ on $\Z(\calV^\Lambda/\bbZ)$ under the pairing given by the symmetrizing form $t_\Lambda$ in Proposition \ref{prop:symmetrizing-form}
\begin{equation}\label{pairing1}
\Z(\calV^\Lambda/\bbZ)\times\Tr(\calV^\Lambda/\bbZ)\to\bfk,\quad (a,b)\mapsto t_{\Lambda}(ab).
\end{equation}

\smallskip

\begin{proposition}\label{prop:transpose}
Let $f\in \Tr(R^\Lambda(\alpha))$, $g\in \Z(R^\Lambda(\alpha+\alpha_i))$ and
$h\in \Z(R^\Lambda(\alpha-\alpha_i))$. We have

$(a)$ $t_{\alpha+\alpha_i}(g\,x_{ir}^-(f))=r(\alpha,i)\,t_\alpha(Z^+_{ir}(g)\,f),$

$(b)$ $t_{\alpha-\alpha_i}(h\,x_{ir}^+(f))=r(\alpha-\alpha_i,i)^{-1}t_{\alpha}(Z^-_{ir}(h)\,f).$

\end{proposition}

\begin{proof} Write $f=\Tr(f')$ with $f'\in R^\Lambda(\alpha)$. 
Write $\hat\eta'_i(1)=\sum_kv_k^\vee\otimes v_k$ as above.
By \eqref{eq:trec} we have
\begin{align*}
t_{\alpha+\alpha_i}(g\,x_{ir}^-(f))
&=t_{\alpha+\alpha_i}(g\,e(\alpha,i)\,x_{n+1}^r\,f')\\
&=r(\alpha,i)\,t_{\alpha}\,\hat\vep'_{i,\lambda}(e(\alpha,i)\,x_{n+1}^r\,g\,f')\\
&=r(\alpha,i)\,t_\alpha(Z^+_{ir}(g)\,f)\\
t_{\alpha-\alpha_i}(h\,x_{ir}^+(f))
&=t_{\alpha-\alpha_i}(\sum_kh\hat\vep'_i(x_n^rv_kf'v_k^\vee))\\
&=t_{\alpha-\alpha_i}(\sum_k\hat\vep'_i(hx_n^rv_kf'v_k^\vee))\\
&=r(\alpha-\alpha_i,i)^{-1}t_\alpha(\sum_khx_n^rv_kf'v_k^\vee)\\
&=r(\alpha-\alpha_i,i)^{-1}t_\alpha(\sum_kv_k^\vee h x_n^rv_kf')\\
&=r(\alpha-\alpha_i,i)^{-1}t_\alpha(\mu((1\otimes x_{n}^r\,h)\,\hat\eta'_{i,\lambda}(1))f')\\
&=r(\alpha-\alpha_i,i)^{-1}t_\alpha(Z^-_{ir}(h)\, f).
\end{align*}
\end{proof}

\medskip

\subsubsection{The loop operators and the isotypic filtration}\label{sec:loop-iso}
In this section we consider an integrable categorical representation of $\frakg_\bfk$ of degree 1 
on a $\bbZ$-graded $\k$-category $\calC$ which is bounded above.
As a preparation for Section \ref{ss:cyclic}, we study the behavior of the loop 
operators on $\Tr(\calC/\bbZ)$ with respect to the isotypic filtration.

Given $i\in I$ and $k\geqslant 0$, the $i$-th isotypic filtration of $\calC$
yields $P\times\bbZ$-graded integrable $\frake_{\Z^k}$-categorical 
representations $\calC_{\geqslant k}$, $\calC_k$ of degree $d_i$ such that
$\calC_k=(\calC_{\geqslant k}/\calC_{>k})^c.$
By Proposition \ref{prop:trace}$(f)$, Theorem \ref{thm:action} and Proposition \ref{lem:functoriality} 
it also yields an exact sequence of $\bbZ$-graded $\bfk\otimes_k L\frake^{[d_i]}_{\Z^k}$-modules
\begin{equation}\label{exact1}
\Tr(\calC_{>k}/\bbZ)\to
\Tr(\calC_{\geqslant k}/\bbZ)\to
\Tr(\calC_{k}/\bbZ)\to 0.
\end{equation}
Let $h_k:\calC_{\geqslant k}\to \calC$ be the canonical inclusion. 
Consider the 
$\bbZ$-graded $\bfk\otimes_\k L\frake_{\Z^k}^{[d_i]}$-submodule
$\Tr(\calC/\bbZ)_{\geqslant k}$ of $\Tr(\calC/\bbZ)$ given by
$\Tr(\calC/\bbZ)_{\geqslant k}=\Tr(h_k)(\Tr(\calC_{\geqslant k}/\bbZ)).$
Since $\Tr(\calC/\bbZ)_{> k}\subseteq\Tr(\calC/\bbZ)_{\geqslant k},$
we have an exact sequence 
\begin{equation}\label{exact2}
0\to\Tr(\calC/\bbZ)_{>k}\to
\Tr(\calC/\bbZ)_{\geqslant k}\to
\Tr(\calC/\bbZ)_{k}\to 0
\end{equation}
and the map 
$\Tr(h_k)$
factors to a surjective $\bbZ$-graded $\bfk\otimes_\k L\frake_{\Z^k}^{[d_i]}$-module homomorphism
\begin{equation}\label{quotient}
\Tr(\calC_{k}/\bbZ)=\Tr(\calV^{k,[d_i]}\otimes_{\Z^{k,[d_i]}}\calM_k/\bbZ)\to\Tr(\calC/\bbZ)_{k}.
\end{equation}

Now, assume that $\calM_{k,\mu}^c=B_{k,\mu}\grproj$ for some
$\bbZ$-graded $\bfk\otimes_\k\Z^{k,[d_i]}$-algebra $B_{k,\mu}$ and set
$B_k=\bigoplus_\mu B_{k,\mu}$.
From \eqref{sl2}, \eqref{iso} we deduce that
\begin{equation}
\calC_k=\bigoplus_{n\in\bbN}
(\Hu^{k,[d_i]}_n\otimes_{\Z^{k,[d_i]}} B_k)\grproj.
\end{equation}
This yields a $\bbZ$-graded $\bfk$-vector space isomorphism
\begin{equation}\label{trck}\Tr(\calC_{k}/\bbZ)\simeq
\bigoplus_{n\in\bbN}\Tr(\Hu^{k,[d_i]}_n)\otimes_{\Z^{k,[d_i]}}\Tr(B_k)
\end{equation}
and the $L\frake_{\Z^k}^{[d_i]}$-action on $\Tr(\calC_k/\bbZ)$ is given by explicit formulas as in Section 
\ref{sec:sl2} below.

Finally, for a future use we'll write 
$$\Tr(\calC_\lambda)_{\geqslant k}=\Tr(\calC_\lambda)\cap\Tr(\calC)_{\geqslant k}$$
and we define 
$\Tr(\calC_\lambda)_{k}$ in the obvious way.

\smallskip

\begin{remark}
We conjecture that the left arrow in sequence \eqref{exact1} is injective if $\calC=\calV^\Lambda$.
Given an idempotent $e$ of a finite dimensional algebra $A$ such that $AeA$ is a stratifying ideal of $A$
in the sense of Cline, Parshall and Scott, we have
a long exact sequence by \cite[thm.~3.1]{Ke98}
$$\dots\to HH_{1}(B)\to HH_0(eAe)\to HH_0(A)\to HH_0(B)\to 0$$
where $B=A/AeA$. Now, set $\frakg=\fraks\frakl_3$ and let 
$\Lambda=\theta=\alpha_1+\alpha_2$ be the highest positive root.
Then $A=R^\Lambda(\theta)$ has a primitive idempotent $e$ such that
$\calC_{>0}=eAe\proj$, $\calC=A\proj$ and $\calC_0=B\proj$.
In this case, the left arrow in \eqref{exact1} is indeed injective, 
although the ideal $AeA$ is not a stratifying ideal of $A$.
\end{remark}

\medskip

\subsubsection{The $\fraks\frakl_2$-case} \label{sec:sl2}
We'll use the same notation as in Example \ref{ex:sl2}.
Thus, we have $\Lambda=k\omega_1$, $\alpha=n\alpha_1$ and
$d_{k,n}=2n(k-n)$.
Let $\bfk$ be a $\Z^{k,[\ell]}$-algebra. 
Recall \eqref{sl2} yields
$$\calV^{k,[\ell]}=\bigoplus_{n\geqslant 0}(\bfk\otimes_{\Z^{k,[\ell]}}\Hu^{k,[\ell]}_n)\grproj.$$
For a future use, let us quote the following.

\smallskip

\begin{proposition}\label{prop:sl2}
(a) The $\bfk\otimes_\k L\frake_{\Z^k}^{[\ell]}$-module 
$\Tr(\calV^{k,[\ell]})$ is generated by $\Tr(\calV^{k,[\ell]}_k)$.

(b) We have $\Tr(\Hu^{k,[\ell]}_n)=V\otimes_\k\Z^{k,[\ell]}$ as a $\bbZ$-graded $\Z^{k,[\ell]}$-module, 
with $V^d=0$ if $d\notin[0,\ell\,d_{k,n}]$
and $V^{\ell\,d_{k,n}}=\k$.
\end{proposition}

\begin{proof}
It is enough to set $\ell=1$ and $\bfk=\Z^k$.
Then, we have 
$$\calV^k=\bigoplus_{n=0}^k\calV^k_{k-2n},\qquad
\calV^k_{k-2n}=\Hu^k_n\grproj.$$
Fix formal variables $y_1,\dots,y_k$ of degree 2 such that $c_p=e_p(y_1,\dots,y_k)$
for $p=1,2,\dots,k$
and $$\Z^k=\k[y_1,y_2,\dots,y_k]^{S_k}.$$
We have $\Z^k$-linear isomorphisms, see e.g., \cite{R08},
$$\Z(\Hu^k_n)\simeq 
\begin{cases}\k[y_1,y_2,\dots,y_k]^{S_n\times S_{k-n}}&\ \text{if}\ n\in[0,k],\\
0&\ \text{else}.\end{cases}$$ 
Since the $\Z^k$-algebra $\Hu^{k}_n$ is symmetric, we have $\Tr(\Hu^{k}_n)\simeq \Z(\Hu^k_n)^*$ 
where $(\bullet)^*$ is the dual as a $\Z^k$-module. This proves part $(b)$.

By Proposition \ref{prop:transpose}, under the isomorphism
$\Tr(\Hu^{k}_n)\simeq \Z(\Hu^k_n)^*$ the operators $x^+_{ir},$ $x^-_{ir}$ on 
$\bigoplus_n\Tr(\Hu^{k}_n)$ are identified with the transpose of the operators
$Z_{ir}^-,$ $Z_{ir}^+$ on $\bigoplus_n\Z(\Hu^k_n).$
For each $p=1,\dots,k-1$, let $\partial_{s_p}$ be the Demazure operator on
$\k[y_1,y_2,\dots,y_k]$
associated with the simple reflection 
$s_p=(p,p+1)$, which is defined by
$\partial_{s_p}(f)=(f-s_p(f))/(y_{p+1}-y_p)$.
The formulas in \cite{R08}, \cite{R12} for the symmetrizing form
of $\Hu_n^k$ imply that the 
Bernstein operators are given by the following explicit formulas for all $f\in \Z(\Hu^k_n)$
$$Z_{ir}^-(f)=\partial_{s_1}\circ\dots\circ\partial_{s_n}
\Big(y_{n+1}^r\, f \prod_{p=n+2}^k(y_{n+1}-y_p)\Big)
\in \Z(\Hu^k_{n+1}),$$
$$Z_{ir}^+(f)=\partial_{s_{k-1}}\circ\dots\circ\partial_{s_n}
\Big(y_{n}^r\, f \prod_{p=1}^{n-1}(y_{p}-y_n)\Big)
\in \Z(\Hu^k_{n-1}).$$

Now, the $\Z^k$-algebra $\Z(\Hu^k_n)$ is equipped with the symmetrizing form
$\partial_{w_n}:\Z(\Hu^k_n)\to \Z^k$, where $w_n\in S_k$
is the unique element of minimal length in the longest coset in $S_k/S_n\times S_{k-n}$.
It yields a non-degenerate bilinear form 
$$\tau^k_n:\Z(\Hu^k_n)\times \Z(\Hu^k_n)\to \Z^k,\quad 
(a,b)\mapsto \partial_{w_n}(ab).$$
The bilinear form $\tau^k_n$ identifies $\Z(\Hu^{k}_n)^*$ with $\Z(\Hu^k_n)$ as a $\Z^k$-module.


Finally, taking the transpose of $Z_{ir}^-$, $Z_{ir}^+$ with respect to $\tau^k_n,$ we get
the following formulas for the operators $x^+_{ir},$ $x^-_{ir}$, which are viewed as
$\Z^k$-linear operators on $\bigoplus_n\Z(\Hu^k_n)$ :
$$x^+_{ir}(f)=\sum_{p=n}^k(p,n)(y_{n}^rf)\in \Z(\Hu^k_{n-1}),$$
$$x^-_{ir}(f)=\sum_{p=1}^{n+1}(p,n+1)(y_{n+1}^rf)\in \Z(\Hu^k_{n+1}).$$

We can now prove part $(a)$ of the proposition.
For each $n\geqslant 0$, let 
$$\Lambda^+(n)=
\{\lambda\in\bbN^n\,;\,\lambda=(\lambda_1\geqslant\lambda_2\geqslant\cdots\geqslant\lambda_n)\}$$
be the set of dominant weights. We abbreviate
$x_{i\lambda}^-=x_{i\lambda_1}^-x_{i\lambda_2}^-\cdots x_{i\lambda_n}^-$.
We must check that 
\begin{equation}\label{surjectivity}
\sum_{\lambda\in\Lambda^+(n)}x_{i\lambda}^-(\Z^k)=\Z(\Hu^k_n).
\end{equation} 
This follows from the equality 
$\Z(\Hu^k_n)=\sum_{\lambda\in\Lambda^+(n)}\Z^k\cdot h_\lambda(y_1,y_2,\dots,y_n)$
and the formula
$$x_{i\lambda}^-(f)=f\cdot h_\lambda(y_1,y_2,\dots,y_n),\quad
\forall f\in \Z^k.$$
\end{proof}

\medskip

\subsection{The center and the cocenter of the minimal categorical representation}\label{ss:cyclic}

Let $\frakg$ be the symmetrizable Kac-Moody algebra over $\k$ associated with the Cartan datum
$(P,P^\vee,\Phi,\Phi^\vee)$.
Fix a dominant weight $\Lambda\in P_+$.
Let $\bfk$ be an $\bbN$-graded $\bfku$-algebra as in Section \ref{sec:QH}.
We equip $\Tr(R^\Lambda)$, $\Z(R^\Lambda)$ with the
$\bbZ$-gradings such that for each $d\in\bbZ$ we have
\begin{align*}
&\Tr(R^\Lambda)^d=\{\Tr(x)\,;\,x\in R^{\Lambda,d}\},\\
&\Z(R^\Lambda)^d=R^{\Lambda,d}\cap \Z(R^\Lambda).
\end{align*}

\smallskip

\subsubsection{The grading of the cocenter}\label{sec:grading}
In this section we set $\bfk=\k$.
For $\alpha\in Q_+$ we write $\lambda=\Lambda-\alpha$ and
$R^\Lambda(\alpha)=R^\Lambda(\alpha;\k)$.

\smallskip

\begin{theorem}\label{thm:grading-trace}
Assume that $\bfk=\k$. Then, for each $\alpha\in Q_+$ we have

(a) if $\Tr(R^\Lambda(\alpha))^d\neq 0$, then $d\in[0, d_{\Lambda,\alpha}]$,

(b) if $\Tr(R^\Lambda(\alpha))^{d_{\Lambda,\alpha}}=0$, then $V(\Lambda)_\lambda=0,$

(c) $\dim\Tr(R^\Lambda(\alpha))^0=\dim V(\Lambda)_\lambda.$
\end{theorem}

\begin{proof} 
For each $i\in I$ and $k\geqslant 0$
the $i$-th isotypic filtration of $\calV^\Lambda$ yields
$P\times\bbZ$-graded categorical $\frake$-representations
$\calV^\Lambda_{\geqslant k}$, $\calV^\Lambda_k$ of degree $d_i$ such that
$\calV^\Lambda_{\geqslant 0}=\calV^\Lambda$,
$\calV^\Lambda_{\lambda,\geqslant k}=0$ if $k$ is large enough (depending on $\lambda$) and
$\calV^\Lambda_k=(\calV^\Lambda_{\geqslant k}/\calV^\Lambda_{>k})^c$.
Further, taking the trace we get
$\bbZ$-graded $L\frake^{[d_i]}_{\Z^k}$-submodules
$\Tr(R^\Lambda)_{\geqslant k}\subset\Tr(R^\Lambda)$ such that
$\Tr(R^\Lambda)_{\geqslant 0}=\Tr(R^\Lambda)$ and
$\Tr(R^\Lambda(\alpha))_{\geqslant k}=0$ if $k$ is large enough.

Now, fix an indecomposable object $P\in \calV^\Lambda_{\lambda}$ and an element
$v_P\in\Tr(\End_{\calV^\Lambda/\bbZ}(P))$ which is homogeneous of degree $d$.
Let $v$ be the image of $v_P$ in $\Tr(R^\Lambda(\alpha)).$
We must prove that either $d\in [0,d_{\Lambda,\alpha}]$ or $v=0$.
Since  $R^\Lambda(0)\simeq\k,$ we may assume that $\alpha\neq 0$.
The Grothendieck group of $\calV^\Lambda_{\lambda}/\bbZ$ is isomorphic to a 
$\bbZ$-lattice in $V(\Lambda)_{\lambda}$ by \cite{KK}.
Since $\alpha\neq 0$, this weight subspace 
does not contain any highest weight vector for the $\frakg$-action on $V(\Lambda)$.

Let $\epsilon_i$ be the Kashiwara function on the set of simple objects in 
$\calA^{\Lambda}/\bbZ$, which is defined as in \cite[sec.~3.2]{KL09}. 
For each indecomposable objects $P,Q\in\calV^\Lambda$ we have
$\epsilon_i(\top(P))>\epsilon_i(\top(Q))$ if and only if there is an integer $k$ such that
$P\in\calV^\Lambda_{\geqslant k}$ and $Q\notin\calV^\Lambda_{\geqslant k}$.

Now, we may choose $i\in I$ such that the integer $m=\epsilon_i(\top(P))$ is positive. 
Let $k_0$ be maximal such that $P\in\calV^\Lambda_{\geqslant k_0}$ and set $m=\epsilon_i(\top(P))$. 
Note that $\epsilon_i(\top(Q))$ is bounded above as $Q$ runs over the set of all indecomposable
objects in $\calV^\Lambda_{\lambda}$ and that
$\epsilon_i(\top(Q))>m$ if and only if $Q\in \calV^\Lambda_{>k_0}$.

The proof is an induction on $\alpha$ and a descending induction on $m$.
We'll assume that 

\begin{itemize}

\item  $\deg(\Tr(R^\Lambda(\beta))\subset[0,d_{\Lambda,\beta}]$ for each 
$\beta\in Q_+$ such that $\beta<\alpha,$

\smallskip

\item $\deg(\Tr(R^\Lambda(\alpha))_{k})\subset[0,d_{\Lambda,\alpha}]$ for each $k>k_0$,

\end{itemize}

\smallskip

\noindent and we must check that
$\deg(\Tr(R^\Lambda(\alpha)_{k_0})\subset[0,d_{\Lambda,\alpha}]$.

By \eqref{quotient}, \eqref{trck} there is a finitely generated $P\times\bbZ$-graded
$\Z^{k_0,[d_i]}$-algebra $B_{k_0}$ such that $B_{k_0,\mu}=0$ if $\mu_i\neq k_0$
and a surjective $\bbZ$-graded $L\frake^{[d_i]}_{\Z^{k_0}}$-module 
homomorphism
$$\Tr(\calV^\Lambda_{k_0}/\bbZ)\simeq
\bigoplus_{n\in\bbN}\Tr(\Hu^{k_0,[d_i]}_n)\otimes_{\Z^{k_0,[d_i]}}\Tr(B_{k_0})
\to\Tr(R^\Lambda)_{k_0}.$$
Let $M_{k_0,\mu}$ be the image of 
$\Tr(\Hu^{k_0,[d_i]}_0)\otimes_{\Z^{k_0,[d_i]}}\Tr(B_{k_0,\mu})$.
Set $M_{k_0}=\bigoplus_\mu M_{k_0,\mu}$.
Recall that $\Tr(\Hu^{k_0,[d_i]}_0)=\Z^{k_0,[d_i]}$ and that
$\sum_{\mu\in\Lambda^+(n)}x_{i\mu}^-(\Tr(\Hu^{k_0,[d_i]}_0))=\Tr(\Hu^{k_0,[d_i]}_n)$
by \eqref{surjectivity}.
Thus, there is a (unique) surjective $\bbZ$-graded $\k$-vector space homomorphism
\begin{equation}\label{map}
\bigoplus_{n\in\bbN}\Tr(\Hu^{k_0,[d_i]}_n)\otimes_{\Z^{k_0,[d_i]}}M_{k_0}\to\Tr(R^\Lambda)_{k_0}
\end{equation}
such that $x_{i\mu}^-(f)\otimes u\mapsto f\,x_{i\mu}^-(u)$
for each $f\in \k$, $\mu\in\Lambda^+(n)$, $u\in M_{k_0}$.
Further, by definition of $k_0$ we have $k_0=\lambda_i+2m$ and \eqref{map} yields a surjective map
$$\Tr(\Hu^{k_0,[d_i]}_n)\otimes_{\Z^{k_0,[d_i]}}M_{k_0,\lambda+m\alpha_i}
\to\Tr(R^\Lambda(\alpha+(n-m)\alpha_i))_{k_0}.$$
Now, a short computation yields
\begin{align*}\label{form2}
d_i\,d_{k_0,m}+d_{\Lambda,\alpha-m\alpha_i}
&=d_{\Lambda,\alpha}.
\end{align*}
Thus by Proposition \ref{prop:sl2}, to prove part $(a)$ of the theorem we must check that 
$$\deg(M_{k_0,\lambda+m\alpha_i})\subset[0,d_{\Lambda,\alpha-m\alpha_i}].$$
Since $B_{k_0,\lambda+m\alpha_i}$ is the endomorphism ring of an object of 
$\calV^\Lambda_{k_0,\lambda+m\alpha_i}$ and $\Tr(B_{k_0,\lambda+m\alpha_i})$ maps onto 
$M_{k_0,\lambda+m\alpha_i},$
we have 
$$\deg(M_{k_0,\lambda+m\alpha_i})\subset\deg(\Tr(R^\Lambda(\alpha-m\alpha_i)).$$
Since $m>0$, we have
$\deg(\Tr(R^\Lambda(\alpha-m\alpha_i)))
\subset[0,d_{\Lambda,\alpha-m\alpha_i}]$ by the inductive hypothesis.
This finishes the proof of $(a)$.

To prove part $(b)$, note that by Proposition \ref{prop:symmetrizing-form} the symmetrizing form on 
$R^\Lambda(\alpha)$ yields a non-degenerate bilinear form
$\Z(R^\Lambda(\alpha))^0\times\Tr(R^\Lambda(\alpha))^{d_{\Lambda,\alpha}}\to\k.$
We deduce that
$$V(\Lambda)_\lambda\neq 0\Rightarrow R^\Lambda(\alpha)\neq 0
\Rightarrow\Z(R^\Lambda(\alpha))^{0}\neq 0
\Rightarrow\Tr(R^\Lambda(\alpha))^{d_{\Lambda,\alpha}}\neq 0.
$$

Finally, let us prove $(c)$. We identify $\frakg$ with its image in 
$L\frakg$.
The operators $x_{i0}^\pm$ with $i\in I$ define a representation of $\frakg$ on $\Tr(R^\Lambda)^d$
for each $d$.
Since $\Tr(\Hu^{k}_n)^0=\k$ for each $n\in[0,k]$,
the map \eqref{map} yields a surjective $\frake$-module homomorphism
$V(k)\otimes_{\k}(M_{k_0})^0\to\Tr(R^\Lambda)_{k_0}^0.$
Here $V(k)$ is the $k+1$-dimensional representation of $\frake$.
See the proof of Proposition \ref{prop:sl2} for more details.
Thus, the proof of $(a)$ above implies that
the representation of $\frakg$ on $\Tr(R^\Lambda)^0$ is cyclic.
Hence $\Tr(R^\Lambda)^0\simeq V(\Lambda)$ as a $\frakg$-module.

\end{proof}

\medskip

\subsubsection{The grading of the center}
We use the same notation as in Section \ref{sec:grading}.
The pairing \eqref{pairing1} gives a non-degenerate bilinear form
\begin{equation*}\label{pairing}
\Z(R^\Lambda(\alpha))^d\times\Tr(R^\Lambda(\alpha))^{d_{\Lambda,\alpha}-d}\to\k.
\end{equation*}
From Theorem \ref{thm:grading-trace} we deduce that

\medskip

\begin{corollary}\label{cor:grading-center}
If $\bfk=\k$ then we have $\Z(R^\Lambda(\alpha))^d=0$ for any $d\notin[0, d_{\Lambda,\alpha}]$.

\qed
\end{corollary}

\smallskip

We have $\Z(R^\Lambda(\alpha))^0\neq\{0\}$ whenever $V(\Lambda)_{\lambda}\neq 0$.

\smallskip

\begin{conj}\label{conj:multiplicity-one}
If $\bfk=\k$ then we have  
$\Z(R^\Lambda(\alpha))^0\simeq\Tr(R^\Lambda(\alpha))^{d_{\Lambda,\alpha}}\simeq\k$
whenever $V(\Lambda)_{\lambda}\neq 0$.
\end{conj}

\smallskip

If $\frakg$ is symmetric of finite type then the conjecture holds, see Remark \ref{rem:4.6} below.

\medskip

\subsubsection{The cocenter is a cyclic module}

Consider a categorical representation of $\frakg_\bfk$ on a 
$\bbZ$-graded $\k$-category $\calC.$ 
Then, the $\bfk$-linear operator $x^\pm_{ir}$ on $\Tr(\calC/\bbZ)$
is well-defined for each $i\in I$, $r\in\bbN$.
Let $h:\Tr(\calC^\hw/\bbZ)\to\Tr(\calC/\bbZ)$
be the trace of the canonical embedding $\calC^\hw/\bbZ\subset\calC/\bbZ.$
Let $\Tr(\calC/\bbZ)^{\cyc}\subset\Tr(\calC/\bbZ)$
be the $\bfk$-submodule generated by the image of $h$
under the action of all operators $x^\pm_{ir}$.

\smallskip
 
\begin{proposition} \label{prop:cyclic}
If $\calC=\calV^\Lambda$ then we have $\Tr(\calC/\bbZ)=\Tr(\calC/\bbZ)^{\cyc}$.
\end{proposition}

\smallskip

The $i$-th isotypic filtration of $\calC$ 
yields a categorical $\frake_{\Z^k}$-representation $\calC_k$ for each $k\geqslant 0$.
Let us quote the following consequence of \eqref{trck} and Proposition \ref{prop:sl2}.

\smallskip

\begin{lemma}\label{cyc2}  
For all $i\in I$, $k\geqslant 0$ we have $\Tr(\calC_k/\bbZ)=\Tr(\calC_k/\bbZ)^{\cyc}$.
\end{lemma}

\smallskip
 
Now, we can prove Proposition \ref{prop:cyclic}. 
 
\smallskip

\begin{proof} 
Fix $\alpha\in Q_+$ and set $\lambda=\Lambda-\alpha$.
Fix an indecomposable object $P\in \calC^\Lambda_{\lambda}$ and an element
$v_P\in\Tr(\End_{\calV^\Lambda/\bbZ}(P))$.
Let $v$ be the image of $v_P$ in $\Tr(R^\Lambda(\alpha)).$
We must prove that $v\in\Tr(\calC_\lambda/\bbZ)^{\cyc}.$
Since $\calC=\calV^\Lambda$, we have $\Tr(\calC_\lambda/\bbZ)=\Tr(R^\Lambda(\alpha))$. 
Since  $\Tr(R^\Lambda(0))\subseteq \Tr(R^\Lambda)^{\cyc},$ we may assume that $\alpha\neq 0$.
Let $i$, $m$, $k_0$ be as in the proof of Theorem \ref{thm:grading-trace}.
The proof is an induction on $\alpha$ and a descending induction on $m$.

We have $v\in\Tr(R^\Lambda(\alpha))_{\geqslant k_0}$.
We'll assume that 

\begin{itemize}

\item  $\deg(\Tr(R^\Lambda(\beta))\subset\Tr(R^\Lambda)^{\cyc}$ for each 
$\beta\in Q_+$ such that $\beta<\alpha,$

\smallskip

\item $\Tr(R^\Lambda(\alpha))_{>k_0}\subset\Tr(R^\Lambda)^{\cyc}.$

\end{itemize}

\smallskip

\noindent We must check that
$\Tr(R^\Lambda(\alpha)_{\geqslant k_0}\subset\Tr(R^\Lambda)^{\cyc}$.

Fix $\beta$ as above
and fix an element $x\in U(L\frake_\bfk)$ of weight $\alpha-\beta$.
Consider the following commutative diagram whose rows are
exact sequences of $L\frake_\bfk$-modules
\begin{equation}\label{diag}
\xymatrix{
0\ar[r]&\Tr(R^\Lambda(\alpha))_{>k_0}\ar[r]^-{f}&
\Tr(R^\Lambda(\alpha))_{\geqslant k_0}\ar[r] ^-{g}&
\Tr(R^\Lambda(\alpha))_{k_0}\ar[r]&0\\
&& \Tr(R^\Lambda(\beta))_{\geqslant k_0}\ar[r]^-{g}\ar[u]^-{x}&
\Tr(R^\Lambda(\beta))_{k_0}\ar[r]\ar[u]^-{x}&0.
}
\end{equation}

Now, set $\bar v=g(v)\in\Tr(R^\Lambda(\alpha))_{k_0}$. 
By definition of $\Tr(R^\Lambda)_{k_0},$
there is a surjective $U(L\frake_\bfk)$-module homomorphism
$\Tr(\calV^\Lambda_{k_0}/\bbZ)\to\Tr(R^\Lambda)_{k_0}.$
Hence, by Lemma \ref{cyc2}, we can choose $\beta$, $x$
such that there is an element $\bar u\in \Tr(R^\Lambda(\beta))_{k_0}$ with $\bar v=x\cdot \bar u.$ 
Then, fix $u\in \Tr(R^\Lambda(\beta))_{\geqslant k_0}$ such that $g(u)=\bar u$.
Then, we have $v-x\cdot u=f(v')$ for some $v'\in \Tr(R^\Lambda(\alpha))_{>k_0}$. 

Finally, we can apply both recursive hypotheses. We deduce that
$u\in\Tr(R^\Lambda(\beta))\subset\Tr(R^\Lambda)^{\cyc}$ and
$v'\in\Tr(R^\Lambda(\alpha))_{>k_0}\subset\Tr(R^\Lambda)^{\cyc}$.
Hence, we have
$v=v-x\cdot u+x\cdot u=f(v')+x\cdot u\in \Tr(R^\Lambda)^{\cyc}.$

\end{proof}

\medskip

\subsection{The symmetric case}

Let $\bfk$ be a commutative $\bbN$-graded ring as in Section \ref{sec:QH}.
Let $\frakg$ be the symmetrizable Kac-Moody algebra over $\k$ associated with the Cartan datum
$(P,P^\vee,\Phi,\Phi^\vee)$.

\subsubsection{Weyl modules}\label{sec:Weyl}
Assume that $\frakg$ is of finite type. 
Fix a dominant weight $\Lambda$.
The \emph{local Weyl module} $\W(\Lambda)$ (over $\k$) is the $\bbZ$-graded $L\frakg$-module
generated by a nonzero element $|\Lambda\rangle$ of degree 0 with the following defining relations
\begin{itemize}
\item $\frakn^+[t]\cdot |\Lambda\rangle=0$,
\item $(f_i)^{\Lambda_i+1}\cdot |\Lambda\rangle=0$,
\item $h\cdot |\Lambda\rangle=\langle h,\Lambda\rangle\,|\Lambda\rangle$ for all $h\in\frakh$,
\item $t\frakh[t]\cdot |\Lambda\rangle=0$.
\end{itemize}

The \emph{global Weyl module} $\Wu(\Lambda)$  is the $\bbZ$-graded $L\frakg$-module
generated by a nonzero element $|\Lambda\rangle$ satisfying the first three relations above. 
Consider the formal series
$\Psi_i(z)=\sum_{r\geqslant 0}\Psi_{ir}\,z^r$, $i\in I$,
given by
$\Psi_i(z)=\exp\big(-\sum_{r\geqslant 1}h_{ir}\,z^r/r\big)$.
Then, there is a unique $\bfku$-module structure on $\Wu(\Lambda)$ such that the representation
of $L\frakg$ is $\bfku$-linear and we have
$\Psi_{ip}\cdot |\Lambda\rangle=c_{ip}\,|\Lambda\rangle$ for each $(i,p)\in I\times\bbN$.
For any $\k$-algebra homomorphism $\bfku\to\bfk$ 
we set $\W(\Lambda,\bfk)=\bfk\otimes_{\bfku}\Wu(\Lambda)$.

The Weyl modules are \emph{universal} in the following sense.
Let $M$ be a $\bbZ$-graded integrable $L\frakg_\bfk$-module
containing an element $m$ of weight $\Lambda$ 
which is annihilated by $\frakn^+[t]$. 
Then there is a unique 
$\bfku$-algebra structure on $\bfk$ and a unique 
$\bbZ$-graded $L\frakg_\bfk$-module homomorphism
$\W(\Lambda,\bfk)\to M$ such that $|\Lambda\rangle\mapsto m$.

Let $\Lambda_\min$ be the unique minimal element in the poset $\{\lambda\in P_+\,;\,\lambda\leqslant\Lambda\}.$
Let $\bfhu'$ be as in Section \ref{sec:factorization}. 
The following is well-known, see \cite{N10}, \cite[thm.~1.1]{KN12}.

\smallskip

\begin{proposition}\label{prop:weyl}
(a)
$\dim_\k(\W(\Lambda))<\infty$.

(b)
$\Wu(\Lambda)$ is a free $\bfku$-module of finite rank.

(c)
$\top(\W(\Lambda))=\W(\Lambda)^0\simeq V(\Lambda)$ as a $\bbZ$-graded $\frakg$-module.

(d) if
$\frakg$ is symmetric, then $$\soc(\W(\Lambda))=\W(\Lambda)^{d_{\Lambda,\Lambda_\min}}\simeq V(\Lambda_\min)[-d_{\Lambda,\Lambda_\min}]$$
as a $\bbZ$-graded $\frakg$-module.

(e)
$\W(\Lambda,\k)\simto\W(\Lambda)$ as a $\bbZ$-graded $L\frakg$-module.

(f) 
$\Lambda=\omega_i$ $\Rightarrow
\Wu(\Lambda)\simeq\bfku\otimes_\k\W(\Lambda)$ 
a $\bbZ$-graded $L\frakg_\bfku$-module.

(g)
$\Wu(\Lambda,\bfhu')\simeq\bfhu'\otimes_\bfhu\bigotimes_{i}\Wu(\omega_i)^{\otimes\Lambda_i}$
as $L\frakg_{\bfhu'}$-modules
such that $|\Lambda\rangle\mapsto 1\otimes\bigotimes_i(w_{\omega_i})^{\otimes\Lambda_i}$.

\qed
\end{proposition}

\smallskip

Now, assume that $\frakg$ is symmetric and \eqref{Q} is satisfied.
We consider the $\bbZ$-graded representation of
$L\frakg_\bfk$ on $\Tr(R^\Lambda)$ in Theorem \ref{thm:action}.

\smallskip

\begin{theorem}\label{cor:weyl}
If $\frakg$ is symmetric of finite type and \eqref{Q} is satisfied,
then there is a unique $\bbZ$-graded $L\frakg_\bfk$-module isomorphism
$\W(\Lambda,\bfk)\simto\Tr(R^\Lambda)$
such that $|\Lambda\rangle\mapsto|\Lambda\rangle$.
\end{theorem}

\begin{proof}
From Proposition \ref{prop:cyclic}, we deduce that the element
$|\Lambda\rangle=\Tr(1)$ of $\Tr(R^\Lambda(0))$ is a generator of the $L\frakg_\bfk$-module
$\Tr(R^\Lambda)$.
Thus, it is enough to prove that there is a $\bbZ$-graded $L\frakg_\bfku$-module isomorphism 
$\underline \W(\Lambda)\simto\Tr(\Ru^\Lambda)$
such that $|\Lambda\rangle\mapsto|\Lambda\rangle$.
To do this, note that, since $\Wu(\Lambda)$ is universal, there is a unique 
$\bbZ$-graded $L\frakg_\bfku$-module homomorphism
\begin{equation}\label{phi}\underline\phi^\Lambda:\Wu(\Lambda)\to\Tr(\Ru^\Lambda),\quad|\Lambda\rangle\mapsto |\Lambda\rangle.\end{equation}
By Proposition \ref{prop:cyclic}, we deduce that $\underline\phi^\Lambda$ is onto.

First, we consider the map 
$\phi^\Lambda:\W(\Lambda)\to\Tr(R^\Lambda(\k))$
given by $\phi^\Lambda=1\otimes\underline\phi^\Lambda.$ 
Since $\underline\phi^\Lambda$ is surjective, the map $\phi^\Lambda$ is also surjective.
To prove that it is injective, we must check that $\phi^\Lambda(\soc(\W(\Lambda)))\neq 0$.
Since $\phi^\Lambda$ is surjective and
$\soc(\W(\Lambda))\simeq V(\Lambda_\min)[-d_{\Lambda,\Lambda_\min}]$ as a 
$\bbZ$-graded $L\frakg$-module,
it is enough to prove that $\Tr(R^\Lambda(\k))^{d_{\Lambda,\Lambda_\min}}\neq 0.$
The weight subspace $V(\Lambda)_{\Lambda_\min}$ is non-zero.
Thus the injectivity of $\phi^\Lambda$ follows from Theorem \ref{thm:grading-trace}$(b)$.

Now, we prove that $\underline\phi^\Lambda$ is injective.
To do so, since $\Wu(\Lambda)$ is a free $\bfku$-module and since $\phi^\Lambda$ is invertible, it is enough to check that 
$\Tr(\Ru^\Lambda)$ is free as a $\bfku$-module.
To do so, note that by Theorem \ref{thm:factorization} and
Example \ref{ex:base-change}$(c),$ 
the $\bfhu'$-algebras $R^\Lambda(\bfhu')$
and $\bfhu'\otimes_\k\bigotimes_{i}R^{\omega_i}(\k)^{\otimes\Lambda_i}$
are Morita equivalent. We deduce that there is an $\bfhu'$-linear isomorphism
\begin{equation}\label{A}
\Tr(R^\Lambda(\bfhu'))\to
\bfhu'\otimes_\k\bigotimes_i\Tr(R^{\omega_i}(\k))^{\otimes\Lambda_i},
\quad
|\Lambda\rangle\mapsto 1\otimes\bigotimes_i|\omega_i\rangle^{\otimes\Lambda_i}.
\end{equation}
Further, by Theorem \ref{thm:factorization}
the map $\eqref{A}$ is $L\frakg_{\bfhu'}$-linear.
Next, by Proposition \ref{prop:weyl}$(g)$ we have an
$L\frakg_{\bfhu'}$-module isomorphism
\begin{equation}\label{B}
\W(\Lambda,\bfhu')\to\bfhu'\otimes_\k\bigotimes_{i}\W(\omega_i)^{\otimes\Lambda_i},\quad 
|\Lambda\rangle\mapsto 1\otimes\bigotimes_i|\omega_i\rangle^{\otimes\Lambda_i}.
\end{equation}
Since the maps \eqref{A}, \eqref{B} are $L\frakg_{\bfhu'}$-linear, from the unicity of the morphism $\underline\phi^\Lambda$ in \eqref{phi}
we deduce that the following square is commutative
$$\xymatrix{\W(\Lambda,\bfhu')\ar[r]^-{\eqref{B}}\ar[d]_-{1\otimes\underline\phi^\Lambda}&\bfhu'\otimes_\k\bigotimes_{i}\W(\omega_i)^{\otimes\Lambda_i}
\ar[d]^-{1\otimes\bigotimes_i(\phi^{\omega_i})^{\otimes\Lambda_i}}\\
\Tr(R^\Lambda(\bfhu'))\ar[r]^-{\eqref{A}}&\bfhu'\otimes_\k\bigotimes_i\Tr(R^{\omega_i}(\k))^{\otimes\Lambda_i}.}$$
Since $\phi^{\omega_i}$ is an isomorphism for each $i$, this implies that the map
$1\otimes\underline\phi^\Lambda$ is also an isomorphism.

Finally, we must check that $\underline\phi^\Lambda$ is an isomorphism.
To do so, note that by construction the map $\underline\phi^\Lambda$
preserves the weight decomposition of $\Wu(\Lambda)$, $\Tr(\Ru^\Lambda).$
Therefore, the claim follows from the following lemma

\smallskip

\begin{lemma} Let $\underline\psi:M\to N$ be a $\bbZ$-graded $\bfku$-module homomorphism
such that $M,N$ are both finitely generated. Assume that the maps 
$1\otimes\underline\psi:\k\otimes_\bfku M\to\k\otimes_\bfku N$ and
$1\otimes\underline\psi:\bfhu'\otimes_\bfku M\to\bfhu'\otimes_\bfku N$ 
are invertible. Then $\psi$ is also invertible.

\qed
\end{lemma}

\end{proof}

\medskip

\subsubsection{Equivariant homology}
For any complex algebraic variety $X$ and any commutative ring $\k$ let 
$H_*(X,\k)$ be the Borel-Moore homology with coefficients in $\k$.
Given an action of a complex linear algebraic group $G$ on $X$,  let $H^G_*(X,\k)$ be the 
$G$-equivariant Borel-Moore homology. 
We'll assume that $X$ admits a locally closed
$G$-equivariant embedding into a smooth projective $G$-variety.
We define it as in \cite[sec.~2.8]{Lu}, but we
assign the degree as in \cite{EG}, so that the fundamental class $[X]$ of $X$ has
degree $2\,\dim X$ if $X$ is pure dimensional. 

Alternatively,  let $D^G(X,\k)$
be the equivariant derived category of
constructible complexes on $X$ with  coefficients in $\k$, see \cite{BL}. Let $\k_X$
and $\k^D_X$ be the constant, and the dualizing sheaf on $X$ . These are objects of $D^G(X,\k)$.
If $M$ is in $D^G(X,\k)$ the $i$-th equivariant cohomology of Y with coefficients in $M$ is by definition
$H^i_G(X,M)=\Ext^i(\k_X,M)$.
In particular, the $G$-equivariant cohomology and Borel-Moore homology of $X$ are defined by
$$H^i_G(X,\k)=H^i_G(X,\k_X),\quad H_i^G(X,\k)=H^{-i}_G(X,\k^D_Y).$$ 
Note that with our conventions, one can have $H_i^G(X,\k)\neq 0$ for $i < 0$.
We'll abbreviate
$$H^i(X,\k)=H^i_{\{1\}}(X,\k),\quad H_i(X,\k)=H_i^{\{1\}}(X,\k).$$ 

The action of the cohomology on the Borel-Moore homology yields a map
$$\cap:H_G^i(X,\k)\otimes_\k H^G_j(X,\k)\to H^G_{j-i}(X,\k).$$
If $X$ is smooth and pure dimensional, the cap product with $[X]$
yields an isomorphism 
$$H_G^i(X,\k)\to H^G_{2\dim X-i}(X,\k).$$

\medskip

\subsubsection{Quiver varieties}
Assume that $\frakg$ is symmetric.
Following Nakajima, to each $\Lambda\in P_+$ and $\alpha\in Q_+$ we associate a quiver variety 
$\frakM(\Lambda,\alpha)$. It is a complex quasi-projective variety equipped with an action of the complex 
algebraic group
$G_\Lambda=\prod_{i\in I}\GL(\Lambda_i)$. Recall that 

\begin{itemize}
\item $\frakM(\Lambda,\alpha)$ is nonsingular, symplectic, possibly empty and is equipped with a 
$G_\Lambda$-equivariant projective morphism to an affine variety
$p:\frakM(\Lambda,\alpha)\to\frakM_0(\Lambda,\alpha),$
\item the variety $\frakM_0(\Lambda,\alpha)$ has a distinguished point denoted by 0 such that
the closed subvariety $\frakL(\Lambda,\alpha)=p^{-1}(0)$ of $\frakM(\Lambda,\alpha)$ is Lagrangian,
\item if $\frakM(\Lambda,\alpha)\neq\emptyset$ then its dimension is $d_{\Lambda,\alpha}$.
\end{itemize}

We put $\frakM(\Lambda)=\bigsqcup_\alpha\frakM(\Lambda,\alpha)$ and 
$\frakL(\Lambda)=\bigsqcup_\alpha\frakL(\Lambda,\alpha).$
We have a canonical
$\k$-algebra isomorphism $H^*_{G_\Lambda}(\bullet,\k)=\bfku$.
Under this isomorphism
the equivariant 
Euler class of the $p$-th fundamental representation of $\GL(\Lambda_i)$ maps to $c_{ip}$ for any
$p\geqslant 0$.
We define
$$H^*_{G_\Lambda}(\frakM(\Lambda),\k)=
\bigoplus_{d\geqslant 0}H^d_{G_\Lambda}(\frakM(\Lambda),\k),$$
$$H_{[*]}^{G_\Lambda}(\frakL(\Lambda),\k)=
\bigoplus_{d\geqslant 0}H_{[d]}^{G_\Lambda}(\frakL(\Lambda),\k)=
\bigoplus_{d\geqslant 0}\bigoplus_\alpha H_{d_{\Lambda,\alpha}-d}^{G_\Lambda}(\frakL(\Lambda,\alpha),\k).$$
Assume also that $\frakg$ is of finite type. The following is well-known.

\smallskip

\begin{proposition} We have

(a) $H_{[d]}^{}(\frakL(\Lambda),\k)=0$ if $d\notin (2\bbZ)\cap[0,d_{\Lambda,\alpha}]$,

(b) $H_{[*]}^{G_\Lambda}(\frakL(\Lambda),\k)$ is free over $\bfku$ and
$H_{[*]}(\frakL(\Lambda),\k)=\k\otimes_{\bfku}H_{[*]}^{G_\Lambda}(\frakL(\Lambda),\k)$,

(c) there is a perfect $\bfku$-bilinear pairing
$H^d_{G_\Lambda}(\frakM(\Lambda),\k)\times 
H_{[d_{\Lambda,\alpha}-d]}^{G_\Lambda}(\frakL(\Lambda),\k)
\to\bfku,$

(d) there is a $\bbZ$-graded $L\frakg_\bfku$-representation on 
$H_{[*]}^{G_\Lambda}(\frakL(\Lambda),\k)$ which is isomorphic to $\Wu(\Lambda)$.
Under this isomorphism
$H_{[*]}^{G_\Lambda}(\frakL(\Lambda,\alpha),\k)$ maps to $\Wu(\Lambda)_{\Lambda-\alpha}$.
\end{proposition}

\begin{proof}
Parts $(b)$, $(c)$ are proved in \cite[thm.~7.3.5]{N00} for equivariant $K$-theory, but the same proof
applies also to equivariant Borel-Moore homology. 
Part $(d)$ follows from \cite{N02}.

\end{proof}

\smallskip

From Theorem \ref{cor:weyl} we deduce that

\smallskip

\begin{theorem}\label{thm:4.5}
If $\frakg$ is symmetric of finite type and \eqref{Q} is satisfied,
then there are  $\bbZ$-graded $L\frakg_\bfk$-module isomorphisms
$\Tr(R^\Lambda)\simeq \bfk\otimes_\bfku H_{[*]}^{G_\Lambda}(\frakL(\Lambda),\k)$.

\qed
\end{theorem}

\smallskip

\begin{remark}\label{rem:4.6} 
$(a)$ Since $\frakL(\Lambda,\alpha)$ is connected, we deduce from Theorem \ref{thm:4.5}
that $\Tr(R^\Lambda(\alpha))^{d_{\Lambda,\alpha}}\simeq\k$ whenever $V(\Lambda)_\lambda\neq 0$ if $\frakg$ is symmetric and of finite type. 

$(b)$ We equip $\Z(R^\Lambda)$ with the $L\frakg_\bfk$-representation dual to $\Tr(R^\Lambda)$
relatively to the pairing \eqref{pairing1} and the anti-involution $\varpi$ of $L\frakg_\bfk$.
The action of $x^\pm_{ir}$ on $\Z(R^\Lambda)$ is described in Proposition \ref{prop:transpose}.
Next, we equip $H^*_{G_\Lambda}(\frakM(\Lambda),\k)$ with the $L\frakg_\bfku$-representation dual to
$H_{[*]}^{G_\Lambda}(\frakL(\Lambda),\k)$ relatively to the pairing in $(c)$ above and the anti-involution $\varpi$.
Then, from Theorem \ref{thm:4.5} we deduce that there is  $\bbZ$-graded $L\frakg_\bfk$-module isomorphism
\begin{equation}\label{isom}
 \bfk\otimes_\bfku H^*_{G_\Lambda}(\frakM(\Lambda),\k)\simto\Z(R^\Lambda).\end{equation}
\end{remark}

\medskip

\subsubsection{The multiplicative structure  $H^*_{G_\Lambda}(\frakM(\Lambda),\k)$}
In this section we explain how to construct a $\bfk$-algebra isomorphism
$\Z(R^\Lambda)\simeq \bfk\otimes_\bfku H^*_{G_\Lambda}(\frakM(\Lambda),\k)$ 
from the $L\frakg_\bfk$-linear isomorphism \eqref{isom}. To simplify, we'll assume
that the condition \eqref{Q} holds.

We first introduce a $\bfk$-linear map
$$a:\Z(R)\to \bfk\otimes_\bfku H^*_{G_\Lambda}(\frakM(\Lambda),\k),$$
which is usually called the \emph{Kirwan map}. Recall that, for each $\alpha\in Q_+$, $i\in I$, the 
$G_\Lambda$-variety $\frakM(\Lambda,\alpha)$ is equipped with $G_\Lambda$-equivariant bundles
$\calV_i,$ $\calW_i$ of rank $a_i$, $\Lambda_i$ respectively
defined as in \cite[sec.~2.9]{N00}, where $a_i$ is the coordinate of $\alpha$
along the simple root $\alpha_i$. On the other hand, if $\height(\alpha)=n$ then the center of $R(\alpha)$ is given by
$$\Z(R(\alpha))=\Big(\bigoplus_{\nu\in I^\alpha}\bfk[x_1,\dots,x_n]e(\nu)\Big)^{\frakS_n},$$
where the symmetric group acts on $\bfk[x_1,\dots,x_n]$ and $e(\nu)$ in the obvious way.
Then, the Kirwan map is the unique $\k$-algebra homomorphism such that
\begin{align*}
&a(c_{ip})=(-1)^pc_p(\calW_i),\quad\forall (i,p)\in I_\Lambda,\\
&a\Big(\sum_{\nu\in I^\alpha}\prod_{k=1}^n(1-zx_k)^{-a_{i,\nu_k}}e(\nu)\Big)=\prod_{j\in I}c_z(\calV_j)^{-a_{i,j}}
\end{align*}
where $c_p$ and $c_z=\sum_p(-z)^pc_p$ are the $p$-th $G_\Lambda$-equivariant Chern class and the 
$G_\Lambda$-equivariant Chern polynomial. 

Next, let $b:\Z(R)\to\Z(R^\Lambda)$ be the canonical map,
induced by the quotient map $R\to R^\Lambda$.

Finally, let $\psi$ be the unique isomorphism as in \eqref{isom} which takes the unit of the ring 
$\bfk\otimes_\bfku H^*_{G_\Lambda}(\frakM(\Lambda,0),\k)$ to the unit of $\Z(R^\Lambda(0)).$

\smallskip

\begin{proposition} Assume that the Kirwan map $a$ is onto. 

(a) The element $\psi(1)\in\Z(R^\Lambda(\alpha))$ is invertible, 
and the map $\phi:\bfk\otimes_\bfku H^*_{G_\Lambda}(\frakM(\Lambda,\alpha),\k)\to\Z(R^\Lambda(\alpha))$  given by
$\phi=\psi(1)^{-1}\cdot\psi$ is a $\bfk$-algebra isomorphism.

(b) The map $b$ is surjective.

\end{proposition}

\begin{proof}
Consider formal series $\Psi_i(z)=\exp\big(-\sum_{r\geqslant 1}h_{ir}\,z^r/r\big)$
introduced in Section \ref{sec:Weyl}. Set $\lambda=\Lambda-\alpha$.
Since the condition \eqref{Q} holds, the formal power series $B_{-i,\lambda}(z)\in\Z(R^\Lambda(\alpha))[[z]]$
in Section \ref{sec:CLO} is given by the following formula, see Proposition \ref{prop:A10},
$$B_{-i,\lambda}(z)=
b\Big(\sum_{\nu\in I^\alpha}\prod_{p=1}^{\Lambda}(1+zy_{ip})\,\prod_{k=1}^n(1-zx_k)^{-a_{i,\nu_k}}\,e(\nu)\Big).$$
Here, the $y_{ip}$ are formal variables as in Section \ref{sec:factorization}.
Further, by Lemma \ref{lem:psi-B}, under the representation of $L\frakg_\bfk$ on $\Z(R^\Lambda)$,
the formal series $\Psi_i(z)$ acts on $\Z(R^\Lambda(\alpha))$ by multiplication
by $B_{-i,\lambda}(z)$.
Next, by \cite[sec.~9.2]{N00}, under the representation of $L\frakg_\bfk$ on 
$\bfk\otimes_\bfku H^*_{G_\Lambda}(\frakM(\Lambda),\k)$ the formal series $\Psi_i(z)$ acts by multiplication by
$$c_z(\calW_i)\cup\prod_jc_z(\calV_j)^{-a_{i,j}}.$$
Since, by definition,  the map $\psi$ is $L\frakg_\bfk$-linear, we deduce that
\begin{align*}
b
&\Big(\sum_{\nu\in I^\alpha}\prod_{p=1}^{\Lambda}(1+zy_{ip})\,\prod_{k=1}^n(1-zx_k)^{-a_{i,\nu_k}}\,e(\nu)\Big)
\cdot\psi(\bullet)
=\psi\Big(c_z(\calW_i)\cup\prod_jc_z(\calV_j)^{-a_{i,j}}\cup\bullet\Big)=\\
&=\psi\Big(a\Big(\sum_{\nu\in I^\alpha}
\prod_{p=1}^{\Lambda}(1+zy_{ip})\,\prod_{k=1}^n(1-zx_k)^{-a_{i,\nu_k}}\,e(\nu)\Big)
\cup\bullet\Big).
\end{align*}
We deduce that, for each $z\in\Z(R),$ we have $b(z)\cdot\psi(\bullet)=\psi(a(z)\cup\bullet)$.
Thus, we have $b(z)\cdot\psi(1)=\psi(a(z))$.
Since $\psi$ and $a$ are onto, there is an element $z$ such that $b(z)\cdot\psi(1)=1$, from which we deduce that 
$\psi(1)$ is invertible.
Furthermore, for each $z,z'\in\Z(R),$ we have $b(z)\cdot b(z')\cdot\psi(1)=\psi(a(z)\cup a(z'))$.
Hence, the map $\phi$ is an algebra isomorphism  and 
the diagram below is commutative
\begin{equation*}\label{triangle0}\begin{split}
\xymatrix{&\Z(R(\alpha))\ar[dr]^-{b}\ar[dl]_-{a}&\\
\bfk\otimes_\bfku H^*_{G_\Lambda}(\frakM(\Lambda,\alpha),\k)\ar[rr]^-{\phi}&&\Z(R^\Lambda(\alpha)).}
\end{split}\end{equation*}
The lemma is proved.

\end{proof}

\medskip

\section{The Jordan quiver}
Hopefully, the results above can  be generalized to quivers with loops using the generalized quiver-Hecke algebras introduced in \cite{KKP}.
In this section we consider the particular case of the Jordan quiver.
In this particular case the quiver-Hecke algebra in \cite{KKP} is the degenerate affine Hecke algebra of the symmetric group.

From now on, let $\k$ be a commutative domain,
$\bfk=\k[\hbar,y_1,\dots,y_r]$ and $\bfk'$ be the fraction field of $\bfk$.
Write $M'=\bfk'\otimes_\bfk M$ for any $\bfk$-module $M$.
The ring $\bfk$ is $\bbZ$-graded with 
$\deg(y_p)=\deg(\hbar)=2$.
If $M$ is free $\bbZ$-graded of finite rank, let $\grdim(M)$ be its 
\emph{graded rank}. It is the unique element in $\bbN[t,t^{-1}]$ such that
$\grdim(\bfk[d])=t^{-d}$ and $\grdim(M\oplus N)=\grdim(M)+\grdim(N)$.

\medskip

\subsection{The quiver-Hecke algebra}

For any integer $n>0$
the QHA of rank $n$ associated with the Jordan quiver is the degenerate affine Hecke $\bfk$-algebra $R(n)$, which is generated by
elements $\tau_1,\dots,\tau_{n-1},x_1,\dots,x_n$ with the defining relations
\begin{itemize}
\item[$(a)$] $x_kx_l=x_lx_k$,
\item[$(b)$] $\tau_k\,\tau_l=\tau_l\,\tau_k$ if $|k-l|>1$,
\item[$(c)$] $\tau_l^2=1$,
\item[$(d)$] $\tau_kx_l-x_{s_k(l)}\tau_k=\hbar\,(\delta_{l,k+1}-\delta_{l,k}),$
\item[$(e)$] $\tau_{k+1}\,\tau_k\,\tau_{k+1}=\tau_k\,\tau_{k+1}\,\tau_k$.
\end{itemize}

We write $R(0)=\bfk.$
For any integer $r\geqslant 0$, the CQHA of rank $n$ and level $r$ is the quotient $R^r(n)$ of $R(n)$ by the two-sided ideal generated by
the element $\prod_{p=1}^r(x_1-y_p)$.  The $\bfk$-algebras $R(n)$, $R^r(n)$ are $\bbZ$-graded, with $\deg(\tau_k)=0$ and 
$\deg(x_k)=2$.

The canonical map $R(n)\to R^r(n)$ yields a $\bbZ$-graded $\bfk$-algebra homomorphism
$\Z(R(n))\to \Z(R^r(n)).$ Let $\Z(R^r(n))^{\text{JM}}$ be its image.

We equip $\Z(R^r(n))$, $\Z(R^r(n))^{\text{JM}}$ with the grading such that 
$$\Z(R^r(n))^d=R^r(n)^d\cap \Z(R^r(n)),\quad (\Z(R^r(n))^{\text{JM}})^d=R^r(n)^d\cap \Z(R^r(n))^{\text{JM}}.$$

The canonical inclusion $R(n)\to R(n+1)$ factors to a $(R^r(n),R^r(n))$-bilinear map $\iota:R^r(n)\to R^r(n+1).$
We have the following.

\smallskip

\begin{proposition}\label{Brundan}
$(a)$
$R^r(n)=\bigoplus_{r_i,w} \bfk\,x_1^{r_1}\cdots x_n^{r_n}w$ where $r_1+\cdots + r_n< r$ and $w\in \frakS_n$, 

$(b)$
$R^r(n+1)=\bigoplus_{k=0}^{r-1}\bigoplus_{j=1}^{n+1} R^r(n)\,(j,n+1)\,x_j^{k}$,

$(c)$
for each $z\in R^r(n+1)$ there are unique elements  $\pi(z)\in R^r(n)\otimes_{R^r(n-1)}R^r(n)$ 
and $p_k(z)\in R(n)$ such that $z=\mu_{\tau_n}(\pi(z))+\sum_{k=0}^{r-1}p_k(z)\,x^k_{n+1},$
yielding a $(R^r(n),R^r(n))$-bimodule isomorphism
$$R^r(n+1)=R^r(n)\,\tau_n\,R^r(n)\oplus\bigoplus_{k=0}^{r-1}R^r(n)\,x_{n+1}^k,$$

$(d)$ the canonical map $\Z(R(n))'\to\Z(R^r(n))'$ is surjective and we have
$$\sum_{n\geqslant 0}\dim (\Z\! R^r(n)')\,q^{2n}=\prod_{j\geqslant 1}(1-q^{2j})^{-r},$$

$(e)$  $\Z(R^r(n))^{\text{JM}}$  is a free $\bfk$-module of finite rank such that
$$\sum_{n\geqslant 0}\grdim \big(\Z\! R^r(n)^{\text{JM}}\big)\,q^{2n}=\prod_{p=1}^r\prod_{i=1}^\infty
\big(1-q^{2i}t^{2(ri+p-1-r)}\big)^{-1}.$$
\end{proposition}

\begin{proof}
Parts $(a)$, $(b)$ are well-known.
The proof of part $(c)$ is similar to \cite[lem.~5.6.1]{Kl}, where the case $\hbar=1$ is done.
Indeed, by part $(b)$ it is enough to show that 
$$\bigoplus_{k=0}^{r-1}\bigoplus_{j=1}^{n} R^r(n)\,(j,n+1)\,x_j^{k}= R^r(n)\tau_n R^r(n).$$
By $(b)$ we have $R^r(n)=\bigoplus_{k=0}^{r-1}\bigoplus_{j=1}^{n} R^r(n-1)\,(j,n)\,x_j^{k}$.
Since $\tau_n$ commutes with $R^r(n-1)$ and $\tau_n(j,n)=(j,n)(j,n+1)$ 
we deduce $R^r(n)\tau_n R^r(n)=\bigoplus_{k=0}^{r-1}\bigoplus_{j=1}^{n} R^r(n)\,(j,n+1)\,x_j^{k}$.
Part $(d)$ is proved in \cite{B08}.

Let us concentrate on $(e)$.
First, we prove that $\Z(R^r(n))^{\text{JM}}$ is free of finite rank as a $\bfk$-module.
To do that, set $\bfk_1=\k[y_1,\dots,y_r]$ and consider the $\bfk_1$-algebras 
$$R(n)_1=R(n)/(\hbar-1),\qquad R^r(n)_1=R^r(n)/(\hbar-1).$$
Then, the assignment $\tau_k\mapsto\tau_k$, $x_k\mapsto x_k\otimes\hbar$, 
$y_p\mapsto y_p\otimes\hbar$ yields a $\bfk$-algebra homomorphism
$$R^r(n)\to R^r(n)_1\otimes_{\bfk_1}\bfk=R^r(n)_1[\hbar].$$ It restricts to an inclusion
\begin{equation}\label{inclusion}\Z(R^r(n))^{\text{JM}}\subset\Z(R^r(n))\subset \Z(R^r(n)_1)[\hbar].
\end{equation}

Now, for any $n$-tuple $\mu=(\mu_1,\dots,\mu_n)$ of non-negative integers, let
$$p_\mu(x_1,\dots,x_n)=\sum_\nu x_1^{\nu_1}\cdots x_n^{\nu_n}\in \Z(R^r(n)_1)$$
where $\nu$ runs over the set of all $n$-tuples which are obtained from $\mu$ by permuting its entries.
Let $\scrP_n(r)$ be the set of all partitions $\mu$ such that 
$\ell(\mu)+\sum_i\lfloor\mu_i/r\rfloor\leqslant n.$
By \cite[thm.~3.2]{B08}, the canonical map $R(n)_1\to R^r(n)_1$ yields a  surjection $\Z(R(n)_1)\to \Z(R^r(n)_1)$.
Further, the elements $p_\mu(x_1,\dots,x_n)$ where $\mu$ runs over the set $\scrP_n(r)$,
form a $\bfk_1$-basis of $\Z (R^r(n)_1)$. Therefore, under the inclusion \eqref{inclusion}, the elements 
$p_\mu(x_1,\dots,x_n)\otimes\hbar^{|\mu|}$ where $\mu\in\scrP_n(r)$ yield a 
$\bfk$-basis of $\Z(R^r(n))^{\text{JM}}$.
We deduce that $\Z(R^r(n))^{\text{JM}}$ is free of finite rank as a $\bfk$-module and that
$$\sum_{n\geqslant 0}\grdim (\Z\! R^r(n))^{\text{JM}}\,q^{2n}=\sum_{n\geqslant 0}\sum_{\mu\in\scrP_n(r)}t^{2|\mu|}q^{2n}.$$
To compute the right hand side, note that \cite[p.~243]{B08} yields a bijection
$$\varphi:\Lambda^+_r(n)\to\scrP_n(r)$$
where $\Lambda^+_r(n)$ is the set of $r$-partitions $\lambda=(\lambda^{(1)},\dots,\lambda^{(r)})$ of $n$.
Further, if $\varphi(\lambda)=\mu$ then
$$|\mu|=r|\lambda|-(r+1)\ell(\lambda)+\sum_{p=1}^rp\,\ell(\lambda^{(p)}).$$
We deduce that
\begin{align*}
\sum_{n\geqslant 0}\grdim (\Z\! R^r(n))^{\text{JM}}\,q^{2n}
&=\sum_{n\geqslant 0}\sum_{\lambda\in\Lambda^+_r(n)}t^{2r|\lambda|-2(r+1)\ell(\lambda)+2\sum_{p=1}^rp\,\ell(\lambda^{(p)})}q^{2n},\\
&=\prod_{p=1}^r\sum_{\lambda\in\Lambda^+_1(n)}t^{2r|\lambda|+2(p-r-1)\ell(\lambda)}q^{2|\lambda|},\\
&=\prod_{p=1}^r\prod_{i=1}^\infty\big(1-q^{2i}t^{2(ri+p-1-r)}\big)^{-1}.
\end{align*}

\end{proof}

\medskip

\subsection{The Lie algebra $\scrW$}\label{sec:W}
Let $\scrW$ be the Lie $\bfk'$-algebra generated by elements $C_k,$ 
$D_{-1,k}$, $D_{0,k+1},$ $D_{1,k}$ with $k\geqslant 0$, 
modulo the following definition relations
\begin{itemize}
\item[$(a)$] $[D_{0,l+1}, D_{0,k+1}]=0$,
\item[$(b)$] $[D_{0,l+1}, D_{1,k}]=\hbar\,D_{1,l+k}$ and $[D_{0,l+1}, D_{-1,k}]=-\hbar\,D_{-1,l+k},$
\item[$(c)$] $3[D_{12}, D_{11}]-[D_{13}, D_{10}]+\hbar^2[D_{11},D_{10}]=0$,
\item[$(d)$] $3[D_{-1,2}, D_{-1,1}]-[D_{-1,3}, D_{-1,0}]+\hbar^2[D_{-1,1},D_{-1,0}]=0$,
\item[$(e)$] $[D_{10},[D_{10}, D_{11}]]=[D_{-1,0},[D_{-1,0}, D_{-1,1}]]=0$,
\item[$(f)$] $[D_{-1,k},D_{1,l}]=E_{k+l},$
\item[$(g)$] $C_k$ is central,
\end{itemize}
where the series $E(z)=\sum_{k\geqslant 0}E_k\,z^k$ is given by the following formula
\begin{equation}\label{E(z)}
\begin{split}
\begin{aligned}
&E(z)=C_0+\sum_{k\geqslant 1}\big(\gamma_\hbar(D_{0,k+1})+\gamma_{-\hbar}(D_{0,k+1})+C_k\big)\,z^k,\\
&\gamma_t(D_{0,k+1})=\sum_{p=1}^{k}(\begin{smallmatrix}k\\p\end{smallmatrix})\, D_{0,k-p+1}\,t^{p-1}.
\end{aligned}
\end{split}
\end{equation}

Let $\scrW_{<0},\scrW_{\geqslant 0},\scrW_0\subset\scrW$ be the Lie subalgebras generated by
$\{D_{1,k}\,;\,k\geqslant 0\}$, $\{D_{-1,k}\,;\,k\geqslant 0\}$, and $\{D_{0,k+1},C_k\,;\,k\geqslant 0\}$ respectively.

There is a unique Lie $\bfk'$-algebra anti-involution $\varpi$ of $\scrW$ such that
\begin{equation}\label{varpiW}\varpi(C_k)=C_k,\qquad\varpi(D_{l,k})=D_{-l,k}.\end{equation}

The Lie $\bfk'$-algebra $\scrW$ is $\bbZ$-graded with $\deg(D_{l,k})=2l$ and $\deg(C_k)=0$.
A representation $V$  is \emph{diagonalizable}
if the operator $D_{0,1}/\hbar$ is diagonalizable with integral eigenvalues.
Then $V$ is $\bbZ$-graded and its degree $2n$ component $V_n$ is the eigenspace associated with the eigenvalue $n$.
We'll say that a diagonalizable representation is \emph{quasifinite} if the degree $2n$ component is finite dimensional for each $n$.
Finally, we define the \emph{character} of a quasifinite representation $V$ to be the formal series
in $\bbN[[q,q^{-1}]]$ given by
$$\ch(V)(q)=\sum_{n\in\bbZ}q^{2n}\dim (V_n).$$ 

Given a linear form $\Lambda:\scrW_0\to\bfk'$ and a module $V$, an element $v\in V$ is \emph{primitive} of weight $\Lambda$
if $\scrW_{0}$ acts on $v$ by $\Lambda$ and $\varpi(\scrW_{<0})$ by zero.
We call $\Lambda(C_0)$ the \emph{level} of $\Lambda$.

Let $M(\Lambda)$ be the Verma module with the lowest weight $\Lambda$.
It is the diagonalizable module induced from the one-dimensional $\scrW_{\geqslant 0}$-module spanned by a primitive vector $|\Lambda\rangle$.
Let $V(\Lambda)$ be the top of $M(\Lambda)$. It is an irreducible diagonalizable module.
We call $\Lambda$ the \emph{lowest weight} and $\Lambda(C_0)$ the \emph{level} of $M(\Lambda)$, $V(\Lambda)$.
We call $|\Lambda\rangle$ the highest weight vectors of $M(\Lambda)$, $V(\Lambda)$.

\medskip

\subsection{The loop operators on the center and the cocenter} 
For each $r,n\in\bbN$ with $r\neq 0$ we set $\calC_n^r=R^r(n)\grproj$.
Write $\calC^r=\bigoplus_n\calC^r_n$.
The restriction and induction functors form an adjoint pair $(F,E)$ with
$$\begin{aligned}
&E: R^r(n+1)\grmod\to R^r(n)\grmod,\quad N\mapsto N,\\
&F: R^r(n)\grmod\to R^r(n+1)\grmod,\quad 
M\mapsto R^r(n+1)\otimes_{R^r(n)}M.\\
\end{aligned}$$

Let $\vep:FE\to 1$ and $\eta:1\to EF$ be the counit and unit of the adjoint pair $(F,E)$.
They are represented respectively by the multiplication map $\mu$ and the canonical map $\iota$
$$\varepsilon :R^r(n)\otimes_{R^r(n-1)}R^r(n)\to R^r(n),$$
$$\eta :  R^r(n)\to R^r(n+1).$$

\begin{proposition} 
(a) The pair $(E,F)$ is adjoint with the counit $\hat\vep:EF\to 1$ and the unit $\hat\eta:1\to FE$ represented by the morphisms 
$\hat\vep: R^r(n+1)\to R^r(n)$, 
$\hat\eta: R^r(n)\to R^r(n)\otimes_{R^r(n-1)}R^r(n)$ 
such that
$\hat\vep(z)=p_{r-1}(z)$ and
$\hbar\,\hat\eta(1)=\pi(x^{r}_{n+1})$.

(b) The $\bfk$-algebra $R^r(n)$ is a symmetric algebra. The symmetrizing form $t_{r,n}:R^r(n)\to\bfk$ is the unique $\bfk$-linear map
sending the element $x_1^{r_1}\dots x_n^{r_n}w$ to 1 if $r_1=\cdots=r_n=r-1$  and $w=1$ and to 0 otherwise.
We have $t_{r,n}=\hat\vep\circ\cdots\circ\hat\vep$ ($n$ times).

\end{proposition}

\begin{proof} See \cite[lem.~5.7.2]{Kl} for $(a)$ and
\cite[thm.~A.2]{BK08} for $(b)$.

\end{proof}

\smallskip

We have $\Tr(\calC^r/\bbZ)=\bigoplus_n\Tr(R^r(n))$.
We equip $\Tr(\calC^r/\bbZ)$ with the $\bbZ^2$-grading such that
$\Tr(R^r(n))$ has the weight $2n$ and the order given by the degree of elements of $R^r(n)$.
For each $k\in\bbN$, we define $\bfk$-linear operators 
$x_{k}^\pm$ on $\Tr(\calC^r/\bbZ)=\bigoplus_n\Tr(R^r(n))$ of weights $\pm 2$ and order $2k$ such that the maps
$$x_{k}^+:  \Tr(R^r(n))\to\Tr(R^r(n-1)),\qquad
x_{k}^-:  \Tr(R^r(n))\to\Tr(R^r(n+1)), $$
are given, for each $f\in\End_{\calC^r}(a)$ and $a\in \calC^r$, by
$$x_{k}^+(\Tr(f))=\Tr(x^k(a)\circ(E(f))),\quad
x_{k}^-(\Tr(f))=\Tr(x^k(a)\circ(F(f))).$$ 
Here $x\in\End(F)$ is represented by the right multiplication by $x_{n+1}$ on $R^r(n+1)$, and $x\in\End(E)$ is the left transposed endomorphism.

\smallskip

\begin{theorem} \label{thm:rep-cat}
The assignment $C_k\mapsto p_k(y_1,\dots,y_r)$,
$D_{\mp 1,k}\mapsto x^\pm_k$ defines a representation of level $r$ of $\scrW$ on $\Tr(\calC^r/\bbZ)'$.
\end{theorem}

\begin{proof}
First, we check that the operators $x_0^-,$ $x_1^-,$ $x_2^-,$ $x_3^-$ satisfy the relations $(c)$, $(e)$ of $D_{10},$ $ D_{11},$ $ D_{12},$ $ D_{13}$
in the definition of $\scrW$.
For each $k\in\bbN$ and $f\in R^r(n)$ we have
$x_k^-(\Tr(f))=\Tr(fx_{n+1}^k)$, where the element $f$ in the right hand side is identified with its image by the map $$\iota:R^r(n)\to R^r(n+1).$$

Thus, using the relations
$\tau_{n+1}x_{n+2}\tau_{n+1}=x_{n+1}+\hbar\tau_{n+1}$ and $\tau_{n+1}^2=1$ we get
$$\begin{aligned}
\ [x_1^-,x_0^-](\Tr(f))&=\Tr(fx_{n+2}-fx_{n+1})\\
&=\Tr(fx_{n+2}\tau_{n+1}^2-fx_{n+1})\\
&=\Tr(f\tau_{n+1}x_{n+2}\tau_{n+1}-fx_{n+1})\\
&=\hbar\,\Tr(f\tau_{n+1})
\end{aligned}$$
Similarly, using the relations 
$$\begin{aligned}
\tau_{n+1}x_{n+1}x_{n+2}^2\tau_{n+1}=x_{n+1}^2x_{n+2}+\hbar x_{n+1}x_{n+2}\tau_{n+1},\\
\tau_{n+1}x_{n+2}^3\tau_{n+1}=x_{n+1}^3+\hbar (x_{n+2}^2+x_{n+1}x_{n+2}+x_{n+1}^2)\tau_{n+1}
\end{aligned}$$
we deduce that
$$\begin{aligned}
&[x_2^-,x_1^-](\Tr(f))=\hbar\,\Tr(fx_{n+1}x_{n+2}\tau_{n+1}),\\
&[x_3^-,x_0^-](\Tr(f))=\hbar\,\Tr(f(x_{n+2}^2+x_{n+1}x_{n+2}+x_{n+1}^2)\tau_{n+1}).
\end{aligned}$$
The relation $(e)$ follows from
$$\begin{aligned}
\ [x_0^-,[x_0^-,x_1^-]](\Tr(f))&=\hbar\,\Tr(f(\tau_{n+2}-\tau_{n+1})),\\
&=\hbar\,\Tr(f(\tau_{n+2}\tau_{n+1}^2-\tau_{n+1}\tau_{n+2}^2)),\\
&=\hbar\,\Tr(f(\tau_{n+1}\tau_{n+2}\tau_{n+1}-\tau_{n+2}\tau_{n+1}\tau_{n+2})),\\
&=0.
\end{aligned}$$
To prove $(c)$ we introduce the element $\varphi_l=(x_l-x_{l+1})\tau_l+\hbar$.
We have
$$\varphi_l^2=\hbar^2-(x_l-x_{l+1})^2,\quad
\varphi_l x_k=x_{s_l(k)}\varphi_l,\quad
\varphi_l \tau_l=-\tau_l\varphi_l.$$
We deduce that
$$\begin{aligned}
\ &(3[x_2^-,x_1^-]-[x_3^-,x_0^-]+\hbar^2[x_1^-,x_0^-])(\Tr(f))=\\
&=\hbar\,\Tr(f(3x_{n+1}x_{n+2}-x_{n+2}^2-x_{n+1}x_{n+2}-x_{n+1}^2-\hbar^2)\tau_{n+1})\\
&=\hbar\,\Tr(f((x_{n+1}-x_{n+2})^2-\hbar^2)\tau_{n+1})\\
&=-\hbar\,\Tr(f\varphi_{n+1}^2\tau_{n+1})\\
&=-\hbar\,\Tr(f\varphi_{n+1}\tau_{n+1}\varphi_{n+1})\\
&=\hbar\,\Tr(f\varphi_{n+1}^2\tau_{n+1})\\
&=0.
\end{aligned}$$
We prove that the operators $x_0^+,$ $x_1^+,$ $x_2^+,$ $x_3^+$ satisfy the relations $(d)$, $(e)$ of $D_{-1,0},$ $ D_{-1,1},$ $ D_{-1,2},$ $ D_{-1,3}$ in a similar way.

Next, we prove the relations $(a)$, $(b)$ and $(f)$. To do so, for each $l\geqslant 0$, consider the element $p_l(x_1,\dots,x_n)\in \Z(R^r(n))$ given by
$p_l(x_1,\dots,x_n)=\sum_{i=1}^nx_i^l$ if $n>0$ and
$0$ if $n=0.$ 
Let $x^0_{l+1}$ be the $\bfk$-linear operator on $\Tr(\calC^r/\bbZ)$ given by
\begin{equation}\label{x0}
x_{l+1}^0(\Tr(f))=\hbar\,\Tr(fp_l(x_1,\dots,x_n)),\quad\forall f\in R^r(n).
\end{equation}
Then, under the assignment $D_{0,k+1}\mapsto x^0_{k+1},$ the defining relation $(a)$ of
$\scrW$ is obviously satisfied. Let us concentrate on $(b)$.
We have
\begin{align*}
[x^0_{l+1}, x^-_{k}](\Tr(f))&=\hbar\,\Tr(f(x_{n+1}^kp_l(x_1,\dots,x_{n+1})-x_{n+1}^kp_l(x_1,\dots,x_{n})))\\
&=\hbar\,x^-_{l+k}(\Tr(f)).
\end{align*}
The relation
$[x^0_{l+1}, x^+_{k}]=\hbar\,x^+_{l+k}$ is proved in a similar way.

Finally, let us prove the relation $(f)$ in the definition of the Lie algebra $\scrW$. 
Let $$\hat\vep: R^r(n+1)\to R^r(n),\qquad 
\hat\eta: R^r(n)\to R^r(n)\otimes_{R^r(n-1)}R^r(n)$$ be as above.
For each $k\in\bbN$ we consider the following elements in $\Z(R^r(n))$
$$\begin{aligned}
B_{+,n}^k&=\hat\vep(x^{r-1+k}_{n+1}),\\
B_{-,n}^k&=\begin{cases} 
-\hbar^2\mu_{x_n^{-r-1+k}}(\hat\eta(1)) &\text{ if } k\geqslant r+1,\\
-p_{r-k}(x_{n+1}^r) &\text{ if } 1\leqslant k\leqslant r,\\
1 &\text{ if } k=0.
\end{cases}
\end{aligned}$$
We may abbreviate $B_\pm^k=B_{\pm,n}^k$.
Consider the formal series $B_\pm(z)=\sum_{k\in\bbN}B_\pm^k\,z^k$.
For each $k,l\in\bbN$ we define the operator $E_{k,l}=[x^+_{k},x^-_{l}]$
on $\Tr(\calC^r/\bbZ)'$.

\smallskip

\begin{lemma}\label{lem:5.4} The following hold

(a) $\hat\vep(\tau_n\,a\,\tau_n)=\hat\vep(a)$ for each $a\in R^r(n)$,

(b) $B_+(z)\,B_-(z)=1$,

(c) $E_{k+l}:=E_{k,l}$ depends only on $k+l$ and we have
$E_l=\sum_{k=0}^l(r-k)\,B_+^{l-k}B_-^k,$

(d) $E(z)=r-z\,\d/\d z\log B_-(z)$,

(e) $B_-(z)=\prod_{k=1}^n(1-\hbar^2\,z^2/(1-zx_k)^{2})\prod_{p=1}^r(1-zy_p)$.
\end{lemma}

\begin{proof}
For part $(a)$, note that $a=\mu_{\tau_{n-1}}(\pi(a))+\sum_{k=0}^{r-1}p_k(a)\,x^k_{n}.$
Thus, a direct computation yields
\begin{align*}
\tau_na\tau_n
&=\mu_{\tau_n\tau_{n-1}\tau_n}(\pi(a))-\sum_{k=0}^{r-1}p_k(a)\hbar\sum_{s+t=k-1}x_n^s\tau_nx_n^t+
\sum_{k=0}^{r-1}p_k(a)x_{n+1}^k-\\
&\qquad
-\hbar^2\sum_{k=0}^{r-1}p_k(a)\sum_{t=0}^{k-1}\sum_{a=0}^{t-1}x_n^{k-1-t+a}x_{n+1}^{t-1-a}.
\end{align*}
We deduce that $p_{r-1}(\tau_n\,a\,\tau_n)=p_{r-1}(a)$.

\smallskip

Next, a similar computation as in the proof of Lemma \ref{lem:recrec} $(a)$, $(c)$ yields
\begin{align}
\pi(x_{n+1}^k)&=\sum_{s=0}^{k-r}\big(\hat\vep(x_{n+1}^{k-s})\otimes x^s_n\otimes 1\otimes 1\big)\pi(x_{n+1}^r),\label{rr:pipi}\\
p_a(x_{n+1}^k)&=\sum_{s=0}^{r-a-1}B_{+,n}^{k-a-s}B_{-,n}^s,\label{rr:pp}
\end{align}
and the equality in part $(b)$. 

The proof of part $(c)$ is similar to Proposition \ref{prop:B5}(b), we briefly indicate the key steps. First, the isomorphism in Proposition \ref{Brundan}$(c)$ represents an isomorphism of functors $\rho: EF1_n\to G$ with $G=FE1_n\oplus 1_n^{\oplus r}$. By Lemma \ref{lem:traceop} we have 
$$x^+_{k}x^-_{l}=\Tr_{EF}(x_i^kx_i^l)=\Tr_{G}(\rho^{-1}(x^kx^l)\rho$$
and it is equal to the sum of the trace of $\rho^{-1}(x^kx^l)\rho$ restricted to each direct factor of $G$. The restriction of $\rho^{-1}(x^kx^l)\rho$ to $FE1_n$ is represented by
\begin{align*}
R^r(n)\otimes_{R^r(n-1)}R^r(n)\to R^r(n)\otimes_{R^r(n-1)}R^r(n),\quad
z\mapsto \pi(x_{n+1}^k\mu_{\tau_n}(z)x_{n+1}^l).
\end{align*}
We have
 \begin{align*}
 \pi(x_{n+1}^k\mu_{\tau_n}(z)x_{n+1}^l)=\mu_{x_n^l\tau_nx_n^k}(z)+\hbar\sum_{a=0}^{k+l-1}\mu_{x_n^a}(z)\pi(x_{n+1}^{k+l-1-a})
\end{align*}
Hence the restriction of $\rho^{-1}(x^kx^l)\rho$ to $FE1_n$ is the endomorphism $$x^lx^k+\hbar\sum_{a=0}^{k+l-1}\pi(x_{n+1}^{k+l-1-a})\vep(1\otimes x_n^a),$$
and by \eqref{rr:pipi} its trace is equal to
\begin{align*}
x^-_{l}x^+_{k}+\sum_{s=r+1}^{k+l}(r-s)B^{l+k-s}_{+,n}B^{s}_{-,n}.
\end{align*}
The restriction of $\rho^{-1}(x^kx^l)\rho$ to the $a$-th copy of $1_n$ is represented by the map
$R^r(n)\to R^r(n),$ $z\mapsto zp_a(x_{n+1}^{k+l+a}).$
By \eqref{rr:pp} it is equal to
$\sum_{s=0}^{r-1-a}B^{k+l-s}_{+,n}B^{s}_{-,n}.
$
We conclude that
\begin{align*}
x^+_{k}x^-_{l}&=x^-_{l}x^+_{k}+\sum_{s=r+1}^{k+l}(r-s)B^{l+k-s}_{+,n}B^{s}_{-,n}+\sum_{a=0}^{r-1}\sum_{s=0}^{r-1-a}B^{k+l-s}_{+,n}B^{s}_{-,n}\\
&=x^-_{l}x^+_{k}+\sum_{s=0}^{k+l}(r-s)B^{l+k-s}_{+,n}B^{s}_{-,n}.\end{align*}

\smallskip

To prove $(d)$, note that, using $(b)$, we get
\begin{align*}
-z\,\d/\d z\log B_-(z)
=-\big(\sum_{k\geqslant 0}kB_-^kz^k\big)\big(\sum_{k\geqslant 0}B_+^kz^k\big)
=\sum_{l\geqslant 0}\sum_{k=0}^l\big(-kB^k_-B^{l-k}_+\big)z^l=-r+E(z).
\end{align*}
Finally we concentrate on $(e)$. We have
\begin{align*}
\tau_nx_n^k\tau_n=x_{n+1}^k-\hbar\sum_{p+q=k-1}x_n^p\tau_nx_n^q-\hbar^2\sum_{a=0}^{k-2}(a+1)x_n^ax_{n+1}^{k-2-a}.
\end{align*}
Using $(a)$, we deduce that
$$\hat\vep(x_n^k)=\hat\vep(\tau_nx_n^k\tau_n)=\hat\vep(x_{n+1}^k)-\hbar^2\sum_{a=0}^{k-2}(a+1)x_n^a\hat\vep(x_{n+1}^{k-2-a}).$$
Thus, we have
$$B_{+,n-1}^{k-r+1}=B_{+,n}^{k-r+1}-\hbar^2\sum_{a=0}^{k-2}(a+1)x_n^aB_{+,n}^{k-2-a-r+1}.$$
This yields $$B_{+,n-1}(z)=(1-\hbar^2z^2(1-zx_n)^{-2})B_{+,n}(z).$$
Hence by $(b)$ we get
$$B_{-,n}(z)=(1-\hbar^2z^2(1-zx_n)^{-2})B_{-,n-1}(z).$$
for $n\geqslant 1$. By induction it remains to compute $B_{-,0}(z)$. Since 
$$x_1^r=-\sum_{a=1}^{r}(-1)^ae_a(y_1,...,y_r)x_1^{r-a},$$ 
we have 
$$\pi(x_1^r)=0,\qquad B_{-,0}^a=(-1)^ae_a(y_1,...,y_r)\quad\forall a\in[1,r],\qquad B_{-,0}^a=0\quad\forall a>r.$$ 
Therefore, we have $B_{-,0}(z)=\prod_{a=1}^r(1-zy_a)$.

\end{proof}

\smallskip

We can now finish the proof of the relation $(f)$ of $\scrW$.
According to Lemma \ref{lem:5.4}, the formal series $E(z)=\sum_{l\geqslant 0}E_l\,z^l$ is given by
\begin{align*}
E(z)=r-z\,\text{d}/\text{d}z\log\big( \prod_{p=1}^r(1-zy_p)\,\prod_{k=1}^n\big(1-z(x_k-\hbar)\big)\big(1-z(x_k+\hbar)\big)\big(1-zx_k\big)^{-2}\big).
\end{align*}
Comparing this with formula \eqref{x0} and the identity
\begin{equation}\label{formule1}\sum_{k\geqslant 1}(za)^k=-z(\text{d}/\text{d}z)\log(1-za),
\end{equation}
we deduce that 
$$E(z)=r-\sum_{k\geqslant 1}\sum_{p=1}^n\Big(2x_p^k-(x_p-\hbar)^k-(x_p+\hbar)^k\Big)\,z^k+\sum_{k\geqslant 1}\sum_{p=1}^ry_p^k\,z^k.$$
This implies that
\begin{equation}\label{E(z)1}
\begin{split}
\begin{aligned}
&E(z)=r+\sum_{k\geqslant 1}\big(\gamma_\hbar(D_{0,k+1})+\gamma_{-\hbar}(D_{0,k+1})+p_k(y_1,\dots,y_r)\big)\,z^k,\\
&\gamma_t(D_{0,k+1})=\sum_{p=1}^{k}(\begin{smallmatrix}k\\p\end{smallmatrix})\, D_{0,k-p+1}\,t^{p-1}.
\end{aligned}
\end{split}
\end{equation}
This finishes the proof of the theorem.

\end{proof}

\smallskip

Set $\Z(\calC^r/\bbZ)=\bigoplus_{n}\Z(R^r(n))$.
The symmetrizing form $t_r=\bigoplus_n t_{r,n}$ on $\bigoplus_nR^r(n)'$ yields a non-degenerate $\bfk'$-bilinear form
$$\Z(\calC^r/\bbZ)'\times\Tr(\calC^r/\bbZ)'\to\bfk',\quad(a,b)\mapsto t_r(ab).$$
Taking the transpose with respect to this bilinear form and twisting the action by the anti-involution $\varpi$ in \eqref{varpiW}, we get a representation of 
$\scrW$ on $\Z(\calC^r/\bbZ)'$ of level $r$.
Let $|r\rangle\in \Z(\calC^r/\bbZ)'$ denote the unit of $\Z(R^r(0))'=\bfk'$.
We define a weight $\Lambda_r$ of level $r$ of $\scrW$ by the formula 
\begin{equation}\label{weight1}
\begin{split}
\begin{aligned}
\Lambda_r(C_k)=p_k(y_1,\dots,y_r),\quad
\Lambda_r(D_{0,k+1})=0,\qquad\forall k\geqslant 0.
\end{aligned}
\end{split}
\end{equation}

\begin{proposition}\label{prop:cyclic-Jordan} The following hold

(a) $|r\rangle$ is a primitive vector of $\Z(\calC^r/\bbZ)'$ of weight $\Lambda_r$,

(b) $\Z(\calC^r/\bbZ)'$ is quasifinite of character $\prod_{j\geqslant 1}(1-q^{2j})^{-r}$.

\end{proposition}

\begin{proof}
Part $(a)$ follows from the formula \eqref{x0}, which implies that
$x_{l+1}^0(|r\rangle)=0$ for all $l\geqslant 0.$
Part $(b)$ follows from Proposition \ref{Brundan}$(d)$.

\end{proof}

\medskip

\subsection{The cohomology ring of the moduli space of framed instantons}\label{sec:instantons}

Let $\frakM(r,n)$ be the moduli space of framed rank $r$ torsion free sheaves on $\bbP^2$ with fixed
second Chern class $n$. Set $\frakM(r)=\bigsqcup_n\frakM(r,n)$.  First, let us review a few basic facts on $\frakM(r)$.
See \cite[sec.~3]{NY} for more details.

The group $\GL(r)\times\GL(2)$ acts on $\frakM(r)$ in the obvious way :
$\GL(r)$ acts by changing the framing and $\GL(2)$ via the tautological action on $\bbP^2$ which preserves the line at infinity.
Let $T\subset\GL(r)$, $A\subset\GL(2)$ be the maximal tori, and let  $\bbC^\times\subset A$ be the \emph{hyperbolic} torus $\{\diag(t,t^{-1})\,;\,t\in\bbC^\times\}$.
Set $G=T\times \bbC^\times$ and $G_A=T\times A$.

We identify the $\bbZ$-graded $\k$-algebra $H^*_G(\bullet,\k)$ with $\bfk$ in the obvious way. 
Let $\bfh=H^*_{G_A}(\bullet,\k)$ and write $\bfh'$ for the fraction field of $\bfh$.
From now on, let $\k$ be a field of characteristic zero.

The $G_A$-variety $\frakM(r,n)$ is equivariantly formal and smooth of dimension $d_{r,n}=2rn$.
Thus $H^*_G(\frakM(r),\k)$ is a free $\bbZ$-graded $\bfk$-module isomorphic to
$H^*(\frakM(r),\k)\otimes\bfk.$
The $G_A$-action yields an $\alpha$-partition of $\frakM(r)$ into affine spaces in the sense of
DeConcini-Lusztig-Procesi.
We deduce that
\begin{equation}\label{dim-coh}
\begin{split}
\begin{aligned}
\sum_{d,n\geqslant 0}\dim H^d(\frakM(r,n),\k)\,q^{2n}t^d
&=\sum_{n\geqslant 0}\grdim H^*_G(\frakM(r,n),\k)\,q^{2n}\\
&=\prod_{p=1}^r\prod_{i=1}^\infty\big(1-q^{2i}t^{2(ri+p-1-r)}\big)^{-1}.
\end{aligned}
\end{split}
\end{equation}
In particular, note that the odd cohomology of $\frakM(r)$ vanishes, hence
the $\bfk$-algebra $H^*_G(\frakM(r),\k)$ is commutative.
Let $|r\rangle$ be the unit of $H^*_G(\frakM(r,0),\k)\simeq\bfk$.
First, we prove the following.

\smallskip

\begin{proposition}
(a) There is a representation of the Lie $\bfk'$-algebra $\scrW$ on $H^{*}_G(\frakM(r),\k)'$
which is isomorphic to $V(\Lambda_r)$. This representation is quasifinite of character 
$\prod_{j\geqslant 1}(1-q^{2j})^{-r}$.

(b) 
There is a unique isomorphism 
$\psi: \Z(\calC^r/\bbZ)' \to H^{*}_G(\frakM(r),\k)'$ 
of representations of $\scrW'$ which takes the element $|r\rangle$ to $|r\rangle.$

(c)
For each $n\in\bbN,$ the element $\psi(1)\in H^{*}_G(\frakM(r,n),\k)'$ is invertible and the map 
$\phi':\Z(R^r(n))'\to H^{*}_G(\frakM(r,n),\k)'$ given by
$\phi'(\bullet)=\psi(1)^{-1}\cup\psi(\bullet)$ is a $\bfk'$-algebra isomorphism.
\end{proposition}

\begin{proof} For each $n,k\in\bbN$, we consider the locus
$$\frakB(r,n+k,n)\subset\frakM(r,n+k)\times\bbC^2\times\frakM(r,n)$$
 of triples $(\calE,x,\calF)$ such that
$\calE\subset\calF$ and $\calF/\calE$ is a length $k$ sheaf supported at $x$.
For each $\gamma\in H^*_G(\frakB(r,n+k,n),\k)$ the correspondence $\frakB(r,n+k,n)$ defines two maps
$$\Theta_+(\gamma):H_{G}^{*}(\frakM(r,n),\k)\to H_{G}^{*}(\frakM(r,n+k),\k),$$
$$\Theta_-(\gamma):H_{G}^{*}(\frakM(r,n+k),\k)\to H_{G}^{*}(\frakM(r,n),\k)'.$$
The map $\Theta_-$ uses localized equivariant cohomology, because the projection
$\frakB(r,n+k,n)\to\frakM(r,n)$ is not proper.

Let $\tau_{n+k,n}$, $\tau_n$ be the \emph{tautological bundles} on $\frakM(r,n+k)\times\frakM(r,n)$, $\frakM(r,n)$.
Write $c_i$ for the $i$-th equivariant Chern class. The obvious map
$$H_{G}^*(\frakM(r,n+k)\times\frakM(r,n),\k)\times H^*_{G}(\bbC^2,\k)\to
H^*_{G}(\frakB(r,n+k,n),\k)$$
is denoted by $(a,b)\mapsto a\otimes b$.

Now, we define the action of $C_k$, $D_{1,k}$, $D_{-1,k}$, $D_{0,k+1}$ on an element of $H_{G}^{*}(\frakM(r,n),\k)'$ by
\begin{equation}\label{rep-geom}
\begin{aligned}
&D_{0,1}=\hbar \,n,\\
&C_k=p_k(y_1,\dots,y_r),\\
&D_{1,k}=-\hbar^{2}\,\Theta_+(c_1(\tau_{n+1,n})^k\otimes 1),\\
&D_{-1,k}=(-1)^{r-1}\,\Theta_-(c_1(\tau_{n+1,n})^k\otimes 1),\\
&\sum_{k\geqslant 0}D_{0,k+2}\,z^k=-\hbar\,(\text{d}/\text{d}z)\log\Big(1+\sum_{k\geqslant 1}c_k(\tau_n)\,(-z)^k\Big)\cup \bullet.
\end{aligned}
\end{equation}
The operators $D_{-1,k}$, $D_{1,k}$, $D_{0,k+1}$ in above are equal to the operators
$\hbar^k\,D_{-1,k}$, $\hbar^k\,D_{1,k}$, $\hbar^{k+1}\,D_{0,k+1}$ in \cite[(3.17)]{SV} respectively.
The reason for this normalization by powers of $\hbar$
is to give the term $\hbar$ in the relations $(b)$, $(c)$, $(d)$ of $\scrW$ in Section \ref{sec:W}, which does not appear in the corresponding relations in \cite{AS}.

By \cite[cor.~3.3]{SV} and \cite{AS}, the formulas \eqref{rep-geom} yield a representation of $\bfh'\otimes_{\bfk'}\scrW$ on 
$$H^{*}_{G_A}(\frakM(r),\k)'=\bfh'\otimes_\bfh H^{*}_{G_A}(\frakM(r),\k).$$
Note that \cite{SV} uses equivariant homology rather than equivariant cohomology, but since $\frakM(r)$ is smooth its equivariant homology and 
cohomology are isomorphic by Poincar\' e duality.
We must check that the formulas \eqref{rep-geom} give indeed a representation of $\scrW$ on $H^{*}_{G}(\frakM(r),\k)'$.

The representation of $\bfh'\otimes_{\bfk'}\scrW$ in \cite{SV} depends on parameters $y_1,\dots,y_r,x,y$ which are generators of the field extension $\bfh'$ of $\k$.
Note that $y_p$ is denoted by the symbol $e_p$ in \cite{SV}.
The representation of $\scrW$ we consider here is a specialization along the hyperplane $x=-y=\hbar$ of some integral form of the representation in \cite{SV}.
We must check that the representation in \cite{SV} specializes effectively.

To prove this, note that the main results of \cite{SV} are obtained by explicit computations in the  $\bfh'$-basis of 
$H^{*}_{G_A}(\frakM(r),\k)'$ formed by the fundamental classes of the fixed points of $\frakM(r)$ under the action of the torus $G_A$.
For these computations it is essential that the fixed points are isolated.
Now, it is well-known, and easy to prove, that the fixed points sets $\frakM(r)^G$ and $\frakM(r)^{G_A}$ are the same.
Indeed,  one can easily check that the explicit formulas for the representation of $\bfh'\otimes_{\bfk'}\scrW$
in \cite[cor.~3.3]{SV} in the basis of $H^{*}_{G_A}(\frakM(r),\k)'$ 
have no poles along the hyperplane $x=y$. This follows from the formulas \cite[(3.17), (D.1)-(D.3)]{SV}.

Note however that the series $E(z)$ in \eqref{E(z)} differs from the corresponding one in \cite[(1.70)]{SV}.
We must check that these formulas are compatible.
To do that, let us review quickly the proof in \cite[p.~326-327]{SV}. Our setting
differs from \cite{SV} because in loc.~cit.~we assumed that both parameters $x$, $y$ are generic, while here we have 
$x=\hbar=-y$ with $\hbar$ generic.

Let $\{a_i\,;\,i\in I\}$ and $\{b_j\,;\,j\in J\}$ be as in \cite[p.~327]{SV}. Then, the computation there 
implies that the element $[D_{-1,k},D_{1,l}]=E_{k+l}$ depends only on $k+l$ and yields the following identity
$$E(z)=\sum_{k\geqslant 0}\Big(\sum_{i\in I}a_i-\sum_{j\in J}b_j\Big)z^k.$$
Fix some formal variables $x_1,x_2,\dots,x_n$ such that $c_i(\tau_n)=e_i(x_1,x_2,\dots,x_n)$ for each $i\in[1,n]$.
Then, the same argument as in \cite[p.~327]{SV} using \cite[lem.~D.1]{SV} implies that
$$E(z)=r-\sum_{k\geqslant 1}\sum_{p=1}^n\Big(2x_p^k-(x_p-\hbar)^k-(x_p+\hbar)^k\Big)\,z^k+\sum_{k\geqslant 1}\sum_{p=1}^ry_p^k\,z^k.$$
We deduce that, see \eqref{E(z)1},
\begin{equation}\label{E(z)2}
\begin{split}
\begin{aligned}
&E(z)=r+\sum_{k\geqslant 1}\big(\gamma_\hbar(D_{0,k+1})+\gamma_{-\hbar}(D_{0,k+1})+p_k(y_1,\dots,y_r)\big)\,z^k,\\
&\gamma_t(D_{0,k+1})=\sum_{p=1}^{k}(\begin{smallmatrix}k\\p\end{smallmatrix})\, D_{0,k-p+1}\,t^{p-1}.
\end{aligned}
\end{split}
\end{equation}

We have proved that the formulas in
\eqref{rep-geom} define a representation of $\scrW$ on $H^{*}_G(\frakM(r),\k)'$.
Now, we must check that this representation is irreducible and is isomorphic to $V(\Lambda_r)$.

The irreducibility follows from the main result of \cite{SV}. More precisely, it is proved in \cite[thm.~8.33]{SV} that the representation of
$\scrW$ on $H^{*}_{G_A}(\frakM(r),\k)'$ gives rise to a representation of the $W$-algebra of the affine Kac-Moody algebra $\widehat{\frakg\frakl}_r$
on $H^{*}_{G_A}(\frakM(r),\k)'$. See, e.g., \cite{A} and \cite{FKRW} for some background on $W(\widehat{\frakg\frakl}_r)$.
It is also proved there that the $W(\widehat{\frakg\frakl}_r)$-module $H^{*}_{G_A}(\frakM(r),\k)'$ is isomorphic to the Verma module
with highest weight and level given respectively by
\begin{equation}\label{highest-weight}
a/x-\rho\,(1+y/x)\quad\text{and}\quad -y/x-r.
\end{equation}
Here we have set $a=(y_1,y_2,\dots,y_r)$ and $\rho=(0,-1,\dots,1-r)$.
Then, the irreducibility of $H^{*}_{G_A}(\frakM(r),\k)'$ as a $W(\widehat{\frakg\frakl}_r)$-module is well-known, because a Verma module with a generic highest weight is irreducible.
The same argument proves that $H^{*}_{G}(\frakM(r),\k)'$ is irreducible as a $W(\widehat{\frakg\frakl}_r)$-module.
To prove that it is also irreducible as a $\scrW$-module use \cite[thm.~8.22]{SV} as in \cite[cor.~8.29]{SV}.

Next, we must identify the representation of $\scrW$ on $H^{*}_G(\frakM(r),\k)'$ with $V(\Lambda_r)$.
To do that, it is enough to prove that 
the element $|r\rangle$ of $H^{*}_G(\frakM(r),\k)'$ is primitive of weight $\Lambda_r$. The equality \cite[(3.9)]{SV} yields
$$c_k(\tau_n)\cup|r\rangle=0,\qquad\forall k\geqslant 1.$$
Thus, from \eqref{rep-geom} we deduce that
$D_{0,k+1}(|r\rangle)=0$ for each $k\geqslant 0.$

Finally, we must check the character formula in $(a)$. It is well-known, and follows easily by counting the (isolated) fixed points in $\frakM(r,n)$.
This finishes the proof of $(a)$.

Now, let us concentrate on $(b)$.
Since the element $|r\rangle$ of $\Z(\calC^r/\bbZ)'$ is primitive of weight $\Lambda_r$ and since
$M(\Lambda_r)$ has a simple top isomorphic to $V(\Lambda_r),$ there is a unique surjective
$\scrW$-module homomorphism from the submodule $M\subseteq\Z(\calC^r/\bbZ)'$
generated by $|r\rangle$ to $H^{*}_G(\frakM(r),\k)'$ such that $|r\rangle\mapsto|r\rangle.$
Since $H^{*}_G(\frakM(r),\k)'$ and $\Z(\calC^r/\bbZ)'$ have the same character, we deduce that
$$\Z(\calC^r/\bbZ)'=M=H^{*}_G(\frakM(r),\k)'.$$
Let 
$\psi$ 
be the unique $\scrW'$-module isomorphism 
$$\psi: \,\Z(\calC^r/\bbZ)' \to H^{*}_G(\frakM(r),\k)',\quad |r\rangle\mapsto |r\rangle.$$

Finally, let us prove part $(c)$.
By restriction, the map $\psi$ yields a $\bfk'$-linear isomorphism 
$$\psi:\, \Z(R^r(n))'\to H^{*}_G(\frakM(r,n),\k)'$$
for each $n\in\bbN$.
We must prove that $\psi(1)$ is invertible in $H^{*}_G(\frakM(r,n),\k)'$ and the map 
$$\phi':\Z(R^r(n))'\to H^{*}_G(\frakM(r,n),\k)',\quad\phi'(\bullet)=\psi(1)^{-1}\cup\psi(\bullet)$$
 is a $\bfk'$-algebra isomorphism.
To do so, we consider the diagram
\begin{equation*}\begin{split}
\xymatrix{&\Z(R(n))'\ar[dr]^-{b'}\ar[dl]_-{a'}&\\H^*_G(\frakM(r,n),\k)'&&\Z(R^r(n))'.}
\end{split}\end{equation*}
The map $b'$ is the $\bfk'$-algebra homomorphism induced by the canonical map $R(n)\to R^r(n)$.
It is surjective by Proposition \ref{Brundan}. The map $a'$ is the $\bfk'$-algebra homomorphism given by 
$$a'(e_i(x_1,\dots,x_n))=c_i(\tau_n),\qquad\forall i\in[1,n].$$
Note that
by definition of the representation of $\scrW$ on $H^{*}_G(\frakM(r,n),\k)'$, the formula \eqref{rep-geom} yields
\begin{equation}\label{eqA}
\hbar^{-1}\,D_{0,k+1}=a'(p_k(x_1,\dots,x_n))\cup \bullet\ \text{on}\ H^{*}_G(\frakM(r,n),\k)'.
\end{equation}
Next, by definition of the representation of $\scrW$ on $\Z(\calC^r/\bbZ)'$, the formula \eqref{x0} in the proof of 
Theorem \ref{thm:rep-cat} yields
\begin{equation}\label{eqB}
\hbar^{-1}\,D_{0,k+1}=b'(p_k(x_1,\dots,x_n))\cdot\bullet\ \text{on}\ \Z(R^r(n))'.
\end{equation}
From \eqref{eqA}, \eqref{eqB}, since $\psi$ is $\scrW$-linear, we deduce that
\begin{equation}\label{eqC}
\psi(b'(p_k(x_1,\dots,x_n))\cdot\bullet)=a'(p_k(x_1,\dots,x_n))\cup\psi(\bullet).
\end{equation}
Now, an easy induction using \eqref{eqC} yields
\begin{equation}\label{eqD}
\psi\, b'(z)=a'(z)\cup\psi(1),\quad\forall z\in \Z(R(n))'.
\end{equation}
We also deduce that
\begin{equation}\label{eqE}
\psi(zz')\cup\psi(1)=\psi(z)\cup\psi(z'),\quad\forall z,z'\in \Z(R^r(n))'.
\end{equation}
Now, since $b'$ and $\psi$ are surjective, 
the equality \eqref{eqD} implies that the element $\psi(1)$ is invertible in the 
(commutative) $\bfk'$-algebra $H^{*}_G(\frakM(r,n),\k)'$.
Thus, the map $\phi'$ above is well-defined and it is a $\bfk'$-algebra homomorphism by \eqref{eqE}.
It is clearly bijective because it is injective and both sides are finite dimensional of the same dimension over $\bfk'$.
Further, we have a commutative diagram
\begin{equation}\label{triangle'}\begin{split}
\xymatrix{&\Z(R(n))'\ar[dr]^-{b'}\ar[dl]_-{a'}&\\H^*_G(\frakM(r,n),\k)'&&\ar[ll]_-{\phi'}\Z(R^r(n))'.}
\end{split}\end{equation}
Part $(c)$ of the proposition is proved.

\end{proof}

\smallskip

We can now prove the following, which is one of the main results of this paper.

\smallskip

\begin{theorem}
The canonical map $\Z(R(n))\to H^*_G(\frakM(r,n),\k)$ is a surjective $\bfk$-algebra homomorphism. It factors to
a $\bfk$-algebra isomorphism $\Z(R^r(n))^{\text{JM}}\to H^{*}_G(\frakM(r,n),\k).$
\end{theorem}

\begin{proof} We define the maps $a:\Z(R(n))\to H^*_G(\frakM(r,n),\k)$ and $b:\Z(R(n))\to \Z(R^r(n))$ as in the triangle \eqref{triangle'} above.
Thus, we have $a'=\bfk'\otimes a$ and $b'=\bfk'\otimes b$.
We claim that there is a $\bfk$-linear map $\phi$ making the following triangle to commute
\begin{equation}\label{triangle}\begin{split}
\xymatrix{&\Z(R(n))\ar[dr]^-{b}\ar[dl]_-{a}&\\H^*_G(\frakM(r,n),\k)&&\ar[ll]_-{\phi}\Z(R^r(n))^{\text{JM}}.}
\end{split}\end{equation}
To prove this, it is enough to check that $\Ker(b)\subset\Ker(a)$.
Since the triangle \eqref{triangle'} commutes, we have $\Ker(b')\subset\Ker(a').$
Thus, since $\Z(R(n))$ is free as a $\bfk$-module, the map $x\mapsto 1\otimes x$ yields an inclusion 
$$\Ker(b)\subset(1\otimes\Z(R(n)))\cap\Ker(a').$$
Finally, since $\Z(R(n))$ is free as a $\bfk$-module, the map $x\mapsto 1\otimes x$ yields an isomorphism
$$\begin{aligned}\Ker(a)&\simeq 1\otimes\Ker(a)\\
&\simeq(1\otimes\Z(R(n)))\cap\Ker(a').
\end{aligned}$$
The claim is proved.
Note that, since the map $b'$ is surjective and the triangles \eqref{triangle'}, \eqref{triangle} commute, we have $\phi'=\bfk'\otimes\phi$.
Thus, since $\phi'$ is injective, we deduce that $\phi$ is also injective.
Now, recall that Proposition \ref{Brundan} and \eqref{dim-coh} yield
$$\sum_{n\geqslant 0}\grdim (\Z\! R^r(n))^{\text{JM}}\,q^{2n}=\sum_{n\geqslant 0}\grdim H^*_G(\frakM(r,n),\k)\,q^{2n}.$$
We deduce that the map $\phi$ is an isomorphism $\Z(R^r)^{\text{JM}}\to H^{*}_G(\frakM(r),\k).$

\end{proof}

\smallskip

The theorem above can be reformulated in the following way. Set $\bfk_1=\k[y_1,\dots,y_r]$ and consider the $\bfk_1$-algebras 
$$R(n)_1=R(n)/(\hbar-1),\quad R^r(n)_1=R^r(n)/(\hbar-1).$$ Recall the inclusion 
$\Z(R^r(n))^{\text{JM}}\subset \Z(R^r(n)_1)[\hbar]$ in \eqref{inclusion}.
By \cite{B08} the canonical map $R(n)_1\to R^r(n)_1$ yields a surjection $\Z(R(n)_1)\to \Z(R^r(n)_1)$.
Since $\Z(R(n)_1)$ is $\bbN$-graded, this yields an increasing separated and exhaustive $\bbN$-filtration $F_\bullet$ of 
$\Z(R^r(n)_1)$.
Let Rees($\Z(R^r(n)_1)$) be the corresponding Rees algebra, i.e., 
$$\text{Rees}(\Z(R^r(n)_1))=\sum_{d\geqslant 0} F_{2d}(\Z(R^r(n)_1))\otimes\hbar^d\subset \Z(R^r(n)_1)[\hbar].$$
By construction, the map \eqref{inclusion} identifies the $\bfk$-algebras $\Z(R^r(n))^{\text{JM}}$ and $\text{Rees}(\Z(R^r(n)_1))$.
We deduce the following

\smallskip

\begin{corollary}
There is a $\bfk$-algebra isomorphism $\operatorname{Rees}(\Z(R^r(n)_1))\simeq H^{*}_G(\frakM(r,n),\k).$

\qed
\end{corollary}

\smallskip

\begin{remark}
In the particular case $r=1$ the corollary was already known and follows from \cite{V01}.
\end{remark}

\bigskip
\clearpage


\appendix

\section{The symmetrizing form}\label{app:symmetrizing}

Fix a dominant weight $\Lambda\in P_+$.
Let $\alpha\in Q_+$ and $i,j\in I$. 
Set $\lambda=\Lambda-\alpha$ and $\lambda_i=\langle\alpha^\vee_i,\lambda\rangle$.

\subsection{Bubbles}
Assume that $\alpha$ has the height $n$.

\begin{definition}\label{df:bubbles}
For each $k\in\bbN$ the \emph{bubble} $B_{\pm i,\lambda}^k$ is the element of $\RL(\alpha)$ given by

\begin{itemize}
\item if $\lambda_i\geqslant 0$ we set
$$\begin{aligned}
B_{+i,\lambda}^k&=\begin{cases}
\hat\vep'_{i,\lambda}(x^{\lambda_i-1+k}_{n+1}\,e(\alpha,i))&\text{ if } k\geqslant -\lambda_i+1,\\
1 &\text{ if } k=\lambda_i=0,
\end{cases}\\
B_{-i,\lambda}^k&=\begin{cases} 
\mu_{x_n^{-\lambda_i-1+k}}(\hat\eta'_{i,\lambda}(1)) &\text{ if } k\geqslant \lambda_i+1,\\
-p_{\lambda_i-k}(x_{n+1}^{\lambda_i}\,e(\alpha,i)) &\text{ if } 1\leqslant k\leqslant \lambda_i,\\
1 &\text{ if } k=0,
\end{cases}
\end{aligned}$$

\item if $\lambda_i\leqslant 0$ we set
$$\begin{aligned}
B_{+i,\lambda}^k&=\begin{cases} 
\hat\vep'_{i,\lambda}(x^{\lambda_i-1+k}_{n+1}\,e(\alpha,i)) &\text{ if } k\geqslant -\lambda_i+1,\\
-\mu_{x_n^{-\lambda_i}}(\tilde \pi_{-\lambda_i-k}) &\text{ if } 1\leqslant k\leqslant -\lambda_i,\\
1 &\text{ if } k=0,
\end{cases}\\
B_{-i,\lambda}^k&=\begin{cases}
\mu_{x_n^{-\lambda_i-1+k}}(\hat\eta'_{i,\lambda}(1))&\text{ if } k\geqslant \lambda_i+1,\\
1 &\text{ if } k=\lambda_i=0.
\end{cases}
\end{aligned}$$

\end{itemize}
\end{definition}

Note that $B_{\pm i,\lambda}^0=1$ in all cases. We set by convention $B^k_{\pm,\lambda}=0$ if $k<0$.

\smallskip

\begin{lemma}
The elements $B_{\pm i,\lambda}^k$ are homogenous central element in $\RL(\alpha)$ of degree $2k$. 
\end{lemma}

\begin{proof}
The central and homogenous property follows from the fact that $\vep'_i$, $p_k$, $\eta'_i$ are homogenous 
$\RL(\alpha)$-bilinear morphisms and that the element
$x_{n+1}$ centralizes $\RL(\alpha)$ in $\RL(\alpha+\alpha_i)$.
The degree is given by an explicit computation.

\end{proof}

\smallskip

\begin{remark}
In Khovanov-Lauda's diagrammatic categorification, the element $B_{+i,\lambda}^k$ corresponds to a 
clockwise \emph{bubble} with a \emph{dot} of multiplicity $\lambda_i-1+k$ and $B_{-i,\lambda}^k$ 
corresponds to a clockwise bubble with a dot of multiplicity $-\lambda_i-1+k$.
\end{remark}

\medskip

\subsection{A useful lemma}
Assume that $\alpha$ has the height $n-1$. Let $\lambda'=\lambda-\alpha_j$ and
$\lambda_i'=\langle\alpha_i^\vee,\lambda'\rangle=\lambda_i-a_{ij}$.
Consider the morphisms
$$\begin{aligned}
&\xymatrix{\bbX_{i,j,\lambda}:E'_iF'_i1_\lambda\ar[rr]^{E'_i\eta'_jF'_i}
&&E'_iE'_jF'_jF'_i1_\lambda\ar[r]^{\tau\tau}
&E'_jE'_iF'_iF'_j1_\lambda\ar[rr]^{E'_j{\hat\vep}'_{i,\lambda'}F'_j}
&&E'_jF'_j1_\lambda,}\\
&\xymatrix{\bbI_{i,j,\lambda}:E'_iF'_i1_\lambda\ar[r]^{\qquad{\hat\vep}'_{i,\lambda}}
&1_\lambda\ar[r]^{\eta_j'\quad}
&E'_jF'_j1_\lambda.}\\
\end{aligned}$$
The morphism $\bbX_{i,j,\lambda}$ is represented by the composition
\[\xymatrix{e(\alpha,i)\RL(\alpha+\alpha_i)e(\alpha,i)\ar[r]^{\iota_j\quad}
&e(\alpha,ij)\RL(\alpha+\alpha_i+\alpha_j)e(\alpha,ij)\ar[d]^{\tau_n(\bullet)\tau_n} \ar@{}[r] &\\
&e(\alpha,ji)\RL(\alpha+\alpha_i+\alpha_j)e(\alpha,ji)\ar[r]^{\quad{\hat\vep}'_{i,\lambda'}}
&e(\alpha,j)\RL(\alpha+\alpha_j)e(\alpha,j),}\]
and $\bbI_{i,j,\lambda}$ is represented by
\[\xymatrix{e(\alpha,i)\RL(\alpha+\alpha_i)e(\alpha,i)\ar[r]^{\qquad\qquad{\hat\vep}'_{i,\lambda}} 
&\RL(\alpha)\ar[r]^{\iota_j\qquad\qquad}
&e(\alpha,i)\RL(\alpha+\alpha_i)e(\alpha,i).
}\]
In other words, given $a\in e(\alpha,i)\RL(\alpha+\alpha_i)e(\alpha,i)$ we have
$$\bbX_{i,j,\lambda}(a)=\hat\vep'_{i,\lambda'}(\tau_n\iota_j(a)\tau_n),\quad 
\bbI_{i,j,\lambda}(a)=\iota_j\hat\vep'_{i,\lambda}(a).$$

Note that, since $\iota_j:\RL(\beta)\to\RL(\beta+\alpha_j)$ is the canonical embedding for any $\beta\in Q_+$, 
we write $\iota_j(b)=b\, e(\beta,j)$ or simply  $\iota_j(b)=b$ for any $b\in  R^\Lambda(\beta)$.
Note also that,since $B_{\pm i,\lambda}^k\in\Z(\RL(\alpha)),$ 
it can be viewed as an element in $\End(1_\lambda)$. 
Thus $x^rB_{\pm i,\lambda}^sx^t$ defines an endomorphism of $E_i'1_\lambda F_i'=E_i'F_i'$
for each $r$, $s$, $t\in\bbN$.

\smallskip

\begin{lemma} \label{lem:relationsI} The following hold

(a) if $i\neq j$ then $\bbX_{i,j,\lambda}=c_{i,j,-a_{ij},0}\,\bbI_{i,j,\lambda}$,

(b) $\bbX_{i,i,\lambda}=-\bbI_{i,i,\lambda}+\sum_{g_1+g_2+g_3=
-\lambda_i'-1}x^{g_1}B^{g_2}_{+i,\lambda'}x^{g_3}.$
 
\end{lemma}

Note that if $\lambda'_i\geqslant 0$ the sum over $g_1,g_2,g_3$ is empty, hence $\bbX_{i,i,\lambda}=-\bbI_{i,i,\lambda}$.

\begin{proof}
Let us prove part $(a)$.
First, assume $\lambda_i>0$. Then $$a=\mu_{\tau_{n-1}}(\pi(a))+\sum_{k=0}^{\lambda_i-1}p_k(a)x^k_n$$ 
with $\pi(a)\in\RL(\alpha)e(\alpha-\alpha_i,i)\otimes_{\RL(\alpha-\alpha_i)}e(\alpha-\alpha_i,i)\RL(\alpha)$ and 
$p_k(a)\in\RL(\alpha)$. We have
$$\tau_n\iota_j(a)\tau_n=
\mu_{\tau_n\tau_{n-1}\tau_ne(\alpha-\alpha_i, iji)}(\pi(a))+\sum_{k=0}^{\lambda_i-1}p_k(a)
\tau_nx_n^ke(\alpha,ij)\tau_n.$$
The relations $(d)$, $(e)$, $(f)$ in the definition of QHA yields
$$
\begin{aligned}
&\mu_{\tau_n\tau_{n-1}\tau_ne(\alpha-\alpha_i, iji)}(\pi(a))=\mu_{\big(\tau_{n-1}\tau_n\tau_{n-1}+
\frac{Q_{i,j}(x_{n-1},x_n)-Q_{i,j}(x_{n+1},x_{n})}{x_{n-1}-x_{n+1}}\big)e(\alpha-\alpha_i, iji)}(\pi(a))\ ,\\
&\tau_nx_n^ke(\alpha,ij)\tau_n=x^k_{n+1}\tau_n^2e(\alpha,ji)=
x_{n+1}^kQ_{j,i}(x_n,x_{n+1})e(\alpha,ji).
\end{aligned}
$$
Since $\lambda_i'=\lambda_i-a_{ij}\geqslant \lambda_i>0$, we have $\bbX_{i,j,\lambda}(a)=
p_{\lambda'_i-1}(\tau_n\iota_j(a)\tau_n)$, the coefficient of $x_{n+1}^{\lambda_i-a_{ij}-1}$.
Since the degree of $x_{n+1}$ in $\frac{Q_{i,j}(x_{n-1},x_n)-Q_{i,j}(x_{n+1},x_{n})}{x_{n-1}-x_{n+1}}$ is at 
most $-a_{ij}-1$, which is less than $\lambda'_i-1=\lambda_i-a_{ij}-1$, and since 
$p_{\lambda'_i-1}(\mu_{\tau_{n-1}\tau_n\tau_{n-1}}(\pi(a)))=0$, we deduce $p_{\lambda'_i-1}$ kills 
$\mu_{\tau_n\tau_{n-1}\tau_ne(\alpha-\alpha_i, iji)}(\pi(a))$.

Next, the degree of $x_{n+1}$ in $x^k_{n+1}Q_{j,i}(x_n,x_{n+1})=x^k_{n+1}Q_{i,j}(x_{n+1},x_n)$ is less or 
equal to $k-a_{ij}$ with the coefficient of $x_{n+1}^{k-a_{ij}}$ given by $c_{i,j,-a_{ij},0}$, therefore
$$p_{\lambda'_i-1}\big(\sum_{k=0}^{\lambda_i-1}p_k(a)\tau_nx_n^ke(\alpha,ij)\tau_n\big)=
\iota_j(p_{\lambda_i-1}(a))c_{i,j,-a_{ij},0}=c_{i,j,-a_{ij},0}(\iota_j\circ\hat\vep_i(a)).$$
It follows that $\bbX_{i,j,\lambda}=c_{i,j,-a_{ij},0}\bbI_{i,j,\lambda}$.

\smallskip

Now, we consider the case $\lambda_i\leqslant 0$. 
Let $\tilde a \in\RL(\alpha)e(\alpha-\alpha_i,i)\otimes_{\RL(\alpha-\alpha_i)}e(\alpha-\alpha_i,i)\RL(\alpha)$ 
such that $\mu_{\tau_{n-1}}(\tilde a)=a$, 
$\mu_{x_{n-1}^k}(\tilde a)=0$ for $k\in [0, -\lambda_i-1]$. We have
$$
\begin{aligned} 
\tau_n\iota_j(a)\tau_n&=\mu_{\tau_n\tau_{n-1}\tau_ne(\alpha-\alpha_i,iji)}(\tilde a)\\
&=\mu_{\big(\tau_{n-1}\tau_n\tau_{n-1}+
\frac{Q_{i,j}(x_{n-1},x_n)-Q_{i,j}(x_{n+1},x_{n})}{x_{n-1}-x_{n+1}}\big)e(\alpha-\alpha_i, iji)}(\tilde a).
\end{aligned}
$$

 Assume that $\lambda_i,\lambda'_i\leqslant 0$. 
Then $\bbX_{i,j,\lambda}(a)=\mu_{x_n^{-\lambda'_i}}(\widetilde{\tau_n\iota_j(a)\tau_n})$.
We claim that
\begin{equation}\label{claim1}
\widetilde{\tau_n\iota_j(a)\tau_n}=
(1\otimes\tau_{n-1}\otimes\tau_{n-1}\otimes 1)(\iota_j\otimes\iota_j)\tilde a.
\end{equation}
Denote the right hand side by $b$. Concretely write $\tilde a=\sum_r\tilde a'_r\otimes\tilde a''_r$ with 
$\tilde a'_r\in\RL(\alpha)e(\alpha-\alpha_i,i)$, $\tilde a''_r\in(\alpha-\alpha_i,i)\RL(\alpha)$, then
$b=\sum_r\tilde a'_re(\alpha-\alpha_i,ij)\tau_{n-1}\otimes\tau_{n-1}e(\alpha-\alpha_i,ij)\tilde a''_r$.

To prove \eqref{claim1}, note that the degree of $x_{n-1}$ in
$\frac{Q_{i,j}(x_{n-1},x_n)-Q_{i,j}(x_{n+1},x_{n})}{x_{n-1}-x_{n+1}}$ is less than or equal to 
$-a_{ij}-1\leqslant-\lambda_i-1$, hence
$$
\mu_{\frac{Q_{i,j}(x_{n-1},x_n)-Q_{i,j}(x_{n+1},x_{n})}{x_{n-1}-x_{n+1}}e(\alpha-\alpha_i, iji)}(\tilde a)=0.$$
We deduce
\begin{equation*}
\begin{aligned}
\mu_{\tau_n}(b)&=\mu_{\tau_{n-1}\tau_n\tau_{n-1}e(\alpha-\alpha_i,iji)}(\tilde a)=\tau_n\iota_j(a)\tau_n,\\
\mu_{x_n^k}(b)&=\mu_{\tau_{n-1}x_n^k\tau_{n-1}e(\alpha-\alpha_i,ij)}(\tilde a)
=\mu_{x_{n-1}^kQ_{ij}(x_{n-1},x_n)e(\alpha-\alpha_i,ij)}(\tilde a).
\end{aligned}
\end{equation*}
Next, using the fact $\mu_{x_{n-1}^k}(\tilde a)=0$ for $k\in [0,-\lambda_i-1]$, $\hat\vep'_{i,\lambda}(a)=\mu_{x_{n-1}^{-\lambda_i}}(\tilde a)$ and the degree of $x_{n-1}$ in $Q_{ij}(x_{n-1},x_n)$ is at most $-a_{ij}$ with the coefficient of $x_{n-1}^{-a_{ij}}$ equals $c_{i,j,-a_{ij},0}$, we get $\mu_{x_{n-1}^kQ_{ij}(x_{n-1},x_n)e(\alpha-\alpha_i,ij)}(\tilde a)=0$ if $k\in [0,-\lambda_i-1+a_{ij}]$ and
$$\mu_{x_{n-1}^{-\lambda'_i}Q_{ij}(x_{n-1},x_n)e(\alpha-\alpha_i,ij)}(\tilde a)=c_{i,j,-a_{ij},0}\,\iota_j(\mu_{x_{n-1}^{-\lambda_i}}(\tilde a))=c_{i,j,-a_{ij},0}\,\iota_j\circ\hat\vep'_{i,\lambda}(a).
$$
Formula \eqref{claim1}
follows and we have $\bbX_{i,j,\lambda}(a)=\mu_{x_n^{-\lambda_i'}}(b)=c_{i,j,-a_{ij},0}\bbI_{i,j,\lambda}(a)$.

Now, assume that $\lambda_i\leqslant 0$ and $\lambda'_i>0$. 
Then $\bbX_{i,j,\lambda}(a)=p_{\lambda'_i-1}(\tau_n\iota_j(a)\tau_n)$. We have $p_{\lambda'_i-1}(\mu_{\tau_{n-1}\tau_n\tau_{n-1}}(\tilde a))=0$. Recall $Q_{i,j}(u,v)=\sum_{p,q\geqslant 0}c_{i,j,p,q}u^pv^q$, with $p\leqslant -a_{ij}$,
$$\begin{aligned}
\frac{Q_{i,j}(x_{n-1},x_n)-Q_{i,j}(x_{n+1},x_{n})}{x_{n-1}-x_{n+1}}&=\sum_{p,q\geqslant 0}c_{i,j,p,q}\frac{x_{n-1}^p-x_{n+1}^p}{x_{n-1}-x_{n+1}}x^q_n\\
&=\sum_{p,q\geqslant 0}c_{i,j,p,q}\big(\sum_{r_1+r_2=p-1}x_{n-1}^{r_1}x_{n+1}^{r_2}\big)x_n^q.
\end{aligned}$$
The height of $\alpha$ is $n-1$, hence $x_n$, $x_{n+1}$ centralize $\RL(\alpha)$. 
We deduce
$$\begin{aligned}
\mu_{\frac{Q_{i,j}(x_{n-1},x_n)-Q_{i,j}(x_{n+1},x_{n})}{x_{n-1}-x_{n+1}}e(\alpha-\alpha_i, iji)}(\tilde a)&=&\sum_{p,q\geqslant 0}c_{i,j,p,q}\big(\sum_{r_1+r_2=p-1}\mu_{x_{n-1}^{r_1}}(\tilde a)x_{n+1}^{r_2}\big)x_n^qe(\alpha,ji).
\end{aligned}$$
Now $\mu_{x_{n-1}^{r_1}}(\tilde a)\neq 0$ only if $r_1\geqslant -\lambda_i$,and $p_{\lambda_i'-1}(\mu_{x_{n-1}^{r_1}}(\tilde a)x_{n+1}^{r_2}x_n^q)\neq 0$ only if $r_2\geqslant \lambda'_i-1=\lambda_i-a_{ij}-1$. Hence $r_1+r_2=p-1\leqslant -a_{ij}-1$ implies that 
$$\begin{aligned}
p_{\lambda'_i-1}\big(\mu_{\frac{Q_{i,j}(x_{n-1},x_n)-Q_{i,j}(x_{n+1},x_{n})}{x_{n-1}-x_{n+1}}e(\alpha-\alpha_i, iji)}(\tilde a)\big)&=c_{i,j,-a_{ij},0}\,\mu_{x_{n-1}^{-\lambda_i}}(\tilde a)e(\alpha,j)\\
&=c_{i,j,-a_{ij},0}\,\iota_j\circ\hat\vep'_{i,\lambda}(a).
\end{aligned}$$
We get $\bbX_{i,j,\lambda}(a)=c_{i,j,-a_{ij},0}\bbI_{i,j,\lambda}$.

\medskip

Now, we concentrate on part $(b)$ in the case $\lambda'_i\geqslant 0$.
Since $\lambda'_i=\lambda_i-2$, we have $\lambda_i\geqslant 2$. Thus 
\begin{equation}\label{eq:vep-pos}
a=\mu_{\tau_{n-1}}(\pi(a))+\sum_{k=0}^{\lambda_i-1}p_k(a)x^k_n
\end{equation}
with $\pi(a)\in\RL(\alpha)e(\alpha-\alpha_i,i)\otimes_{\RL(\alpha-\alpha_i)}e(\alpha-\alpha_i,i)\RL(\alpha)$ and $p_k(a)\in\RL(\alpha)$. We have
$$\begin{aligned}
&\tau_n\iota_i(a)\tau_n=\mu_{\tau_n\tau_{n-1}\tau_ne(\alpha-\alpha_i, i^3)}(\pi(a))+\sum_{k=0}^{\lambda_i-1}p_k(a)\tau_nx_n^k\tau_ne(\alpha,i^2)\,,\\
&\bbX_{i,i,\lambda}(a)=
p_{\lambda'_i-1}(\tau_n\iota_i(a)\tau_n).
\end{aligned}$$
Since $Q_{ii}=0$, the relation $(f)$ in the definition of QHA yields
$$\mu_{\tau_n\tau_{n-1}\tau_ne(\alpha-\alpha_i, i^3)}(\pi(a))=\mu_{\tau_{n-1}\tau_n\tau_{n-1}e(\alpha-\alpha_i, i^3)}(\pi(a)),\\
$$
Hence it is killed by $p_{\lambda_i'-1}$ when $\lambda'_i>0$. The relation (d), (e) for QHA implies
\begin{equation}\label{eq:tau-power-x}
\begin{split}
\begin{aligned}
\tau_nx_n^k\tau_ne(\alpha,i^2)&=x^k_{n+1}\tau_n^2e(\alpha,i^2)-\frac{x_n^k-x_{n+1}^k}{x_n-x_{n+1}}\tau_ne(\alpha,i^2)\\
&=-\sum_{r_1+r_2=k-1}x_n^{r_1}x_{n+1}^{r_2}\tau_ne(\alpha,i^2)\\
&=-\sum_{r_1+r_2=k-1}x_n^{r_1}(\tau_n x_n^{r_2}+\sum_{g_1+g_2=r_2-1}x_n^{g_1}x_{n+1}^{g_2})\\
&=-\sum_{r_1+r_2=k-1}x_n^{r_1}\tau_n x_n^{r_2}-\sum_{g_1+g_2=k-2}(g_1+1)x_n^{g_1}x_{n+1}^{g_2}.
\end{aligned}
\end{split}
\end{equation}

If $\lambda_i'>0$, we deduce that
$$\begin{aligned}
p_{\lambda'_i-1}\big(\tau_nx_n^k\tau_ne(\alpha,i^2)\big)&=-p_{\lambda_i-3}\big(\sum_{g_1+g_2=k-2}(g_1+1)x_n^{g_1}x_{n+1}^{g_2}\big)=\begin{cases}
0 &\text{ if }k<\lambda_i-1,\\
-1 &\text{ if }k=\lambda_i-1.
\end{cases}
\end{aligned}$$
By consequence $\bbX_{i,i,\lambda}(a)=-p_{\lambda_i-1}(a)e(\alpha,i)=-\iota_i\circ\hat\vep'_{i,\lambda}(a)=-\bbI_{i,i,\lambda}(a).$

If $\lambda'_i=0$, we deduce that
$$\begin{aligned}
\tau_n\iota_i(a)\tau_n&=\mu_{\tau_{n-1}\tau_n\tau_{n-1}e(\alpha-\alpha_i, i^3)}(\pi(a))-p_1(a)\tau_ne(\alpha,i^2)\\
&=\mu_{\tau_n}\big((1\otimes\tau_{n-1}\otimes\tau_{n-1}\otimes 1)(\iota_i\otimes\iota_i)\pi(a)-\iota_i (p_1(a))\otimes e(\alpha,i^2)\big).
\end{aligned}$$
Hence
$$\begin{aligned}\hat\vep'_{i,\lambda'}(\tau_n\iota_i(a)\tau_n)&=\mu_{1}\big((1\otimes\tau_{n-1}\otimes\tau_{n-1}\otimes 1)(\iota_i\otimes\iota_i)\pi(a)-\iota_i (p_1(a))\otimes e(\alpha,i^2)\big)\\
&=-\iota_i (p_1(a)).
\end{aligned}$$
Here in the second equality we have used the fact $\tau_{n-1}^2e(\alpha-\alpha_i,i^3)=0$. Since $\lambda_i=2$, we have $\hat\vep'_{i,\lambda}(a)=p_1(a)$. So
we get again $\bbX_{i,i,\lambda}(a)=-\bbI_{i,i,\lambda}(a)$.

\medskip

Finally, we prove on part $(b)$ for $\lambda_i'<0$.
By assumption we have $\lambda'_i=\lambda_i-2< 0$. 
First, if $\lambda_i=1$, then $a=\mu_{\tau_{n-1}}(\pi(a))+p_0(a)$ and $\hat\vep'_{i,\lambda}(a)=p_0(a)$ 
as in \eqref{eq:vep-pos}. The same computation as in the previous lemma yields
$$\tau_n\iota_i(a)\tau_n=\mu_{\tau_{n-1}\tau_n\tau_{n-1}e(\alpha-\alpha_i, i^3)}(\pi(a)).$$
Let $b=(1\otimes\tau_{n-1}\otimes\tau_{n-1}\otimes 1)(\iota_i\otimes\iota_i)\pi(a)$. 
Then $\tau_n\iota_i(a)\tau_n=\mu_{\tau_n}(b)$ and $\mu_{x_n^0}(b)=0$. Since $\lambda'_i=-1$, 
we deduce $b=\widetilde{\tau_n\iota_i(a)\tau_n}$ and
$$\begin{aligned}
\hat\vep'_{i,\lambda'}(\tau_n\iota_i(a)\tau_n)&=\mu_{x_n}(b)\\
&=\mu_{\tau_{n-1}x_n\tau_{n-1}e(\alpha-\alpha_i,i^2)}(\pi (a))\\
&=\mu_{\tau_{n-1}e(\alpha-\alpha_i,i^2)}(\pi(a))\\
&=a-\iota_i(p_0(a)).
\end{aligned}$$
We conclude $\bbX_{i,i,\lambda}(a)=-\bbI_{i,i,\lambda}(a)+a=-\bbI_{i,i,\lambda}(a)+B^0_{+i,\lambda'}a$.

It remains to consider the case $\lambda_i\leqslant 0$. Let $\tilde a \in\RL(\alpha)e(\alpha-\alpha_i,i)\otimes_{\RL(\alpha-\alpha_i)}e(\alpha-\alpha_i,i)\RL(\alpha)$ such that $\mu_{\tau_{n-1}}(\tilde a)=a$, $\mu_{x_{n-1}^k}(\tilde a)=0$ for $k\in [0, -\lambda_i-1]$. We have
$$
\begin{aligned} 
\tau_n\iota_i(a)\tau_n=\mu_{\tau_n\tau_{n-1}\tau_ne(\alpha-\alpha_i,i^3)}(\tilde a)=\mu_{\tau_{n-1}\tau_n\tau_{n-1}e(\alpha-\alpha_i, i^3)}(\tilde a).
\end{aligned}
$$
Let $b=(1\otimes\tau_{n-1}\otimes\tau_{n-1}\otimes 1)(\iota_i\otimes\iota_i)(\tilde a)\in\RL(\alpha+\alpha_i)e(\alpha,i)\otimes_{\RL(\alpha)}e(\alpha,i)\RL(\alpha+\alpha_i)$. Then $\mu_{\tau_n}(b)=\tau_n\iota_i(a)\tau_n$, and $\mu_{x_n^k}(b)=\mu_{\tau_{n-1}x_n^k\tau_{n-1}e(\alpha-\alpha_i,i^2)}(\tilde a)$. A computation similar to \eqref{eq:tau-power-x} yields
$$\tau_{n-1}x_n^k\tau_{n-1}e(\alpha-\alpha_i,i^2)=\big(\sum_{g_1+g_2=k-1}x_n^{g_1}\tau_{n-1}x_n^{g_2}-\sum_{g_1+g_2=k-2}(g_2+1)x_{n-1}^{g_1}x_n^{g_2}\big)e(\alpha-\alpha_i,i^2).$$
Therefore for $0\leqslant k\leqslant -\lambda'_i$, we have
$$\begin{aligned}
\mu_{x_n^k}(b)&=\sum_{g_1+g_2=k-1}x_n^{g_1}\mu_{\tau_{n-1}}(\tilde a)x_n^{g_2}-\sum_{g_1+g_2=k-2}(g_2+1)\mu_{x_{n-1}^{g_1}}(\tilde a)x_n^{g_2}\\
&=\sum_{g_1+g_2=k-1}x_n^{g_1}ax_n^{g_2}-\delta_{k=-\lambda'_i}\mu_{x_{n-1}^{-\lambda_i}}(\tilde a).
\end{aligned}
$$
Here we have used the fact $\mu_{\tau_{n-1}}(\tilde a)=a$ and $\mu_{x_{n-1}^k}(\tilde a)=0$ for 
$0\leqslant k\leqslant -\lambda_i-1$ in the second equality.
Finally, recall from \eqref{eq:adj-bim-neg} that there are elements 
$\tilde\pi_\ell\in\RL(\alpha+\alpha_i)e(\alpha,i)\otimes_{\RL(\alpha)}e(\alpha,i)\RL(\alpha+\alpha_i)$
 for $0\leqslant \ell\leqslant-\lambda_i'-1$ such that $\mu_{\tau_n}(\tilde\pi_\ell)=0$ and 
 $\mu_{x_n^k}(\tilde\pi_\ell)=\delta_{k,l}$.
Set $$c=b-\sum_{k=0}^{-\lambda'_i-1}\big(\sum_{g_1+g_2=k-1}x_n^{g_1}\tilde\pi_kx_n^{g_2}\big).$$
Then $\mu_{\tau_n}(c)=\mu_{\tau_n}(b)=\tau_n\iota_i(a)\tau_n$ and $\mu_{x_n^k}(c)=0$ for 
$0\leqslant k\leqslant -\lambda'_i-1$.
Hence $c=\widetilde{\tau_n\iota_i(a)\tau_n}$, and $\bbX_{i,i,\lambda}(a)$ is equal to
$$\begin{aligned}
\mu_{x_n^{-\lambda'_i}}(c)
&=\mu_{x_n^{-\lambda'_i}}(b)-\sum_{k=0}^{-\lambda'_i-1}\big(\sum_{g_1+g_2=k-1}x_n^{g_1}a\mu_{x_n^{-\lambda'_i}}(\tilde\pi_k)x_n^{g_2}\big)\\
&=-\mu_{x_{n-1}^{-\lambda_i}}(\tilde a)e(\alpha,i)+\sum_{g_1+g_2=-\lambda'_i-1}x_n^{g_1}ax_n^{g_2}-\sum_{g_3=1}^{-\lambda'_i}\sum_{g_1+g_2=-\lambda'_i-1-g_3}x_n^{g_1}a\mu_{x_n^{-\lambda'_i}}(\tilde\pi_{-\lambda'_i-g_3})x_n^{g_2}\\
&=-\bbI_{i,i,\lambda}(a)+\sum_{g_1+g_2+g_3=-\lambda'_i-1}x_n^{g_1}B^{g_3}_{+i,\lambda'}x_n^{g_2}(a).
\end{aligned}$$
In the second equality, we have substituted $g_3=-\lambda'_i-k$. 
In the third equality, we have used the definition of $B^{g_3}_{+i,\lambda'}$ for 
$0\leqslant g_3\leqslant -\lambda'_i$ in Definition \ref{df:bubbles}.
\end{proof}

\smallskip

\begin{remark}\label{rk:KL action}
Assume that the \emph{$Q$-cyclicity} condition in \cite[(2.4),(2.5)]{CL} holds for $\calV^\Lambda$, i.e., the
endomorphisms $x_i\in \End(E_i')$, $\tau_{ij}\in\End(E_i'E_j')$ are such that
$$x_i^\vee={}^\vee x_i,\qquad\tau_{ij}^\vee={}^\vee\tau_{ij}.$$ 
See the notation in Section \ref{sss:adjunction}. 
Then, under the adjunction isomorphism $\Hom(E_i'F_i', E_j'F_j')=\Hom(F_jE_i, F_jE_i)$, part $(a)$ of the previous 
lemma gives the second equality in the mixed relation \cite[(2.16)]{CL}.
For $i=j$, under the same adjunction, part $(b)$ gives the second equalities in \cite[(2.22), (2.24), (2.26)]{CL}. The other 
relations in \cite[sec.~2.6.3]{CL} can be checked similarly. Since the computations are quite lengthy and will not be 
needed, we omit the details here. Finally, the \emph{fake bubbles} relations \cite[(2.20)]{CL} is proved in Lemma 
\ref{lem:A10}$(a)$ below. Therefore, assuming the $Q$-cyclicity condition, we have proved that $\calV^\Lambda$ carries 
a representation of the Khovanov-Lauda's $2$-Kac-Moody algebra. 
The $Q$-cyclicity condition can probably be proved by similar computations as in \cite{K}. We have not checked this.
\end{remark}

\medskip

\subsection{Proof of Proposition \ref{prop:symmetrizing-form}}\label{sec:A1}
Assume that $\height(\alpha)=n$.
For $\nu=(\nu_1,\dots ,\nu_n)\in I^\alpha$ and $k\in[1,n]$ we set $\nu^{(k)}=(\nu_1,\cdots,\nu_k)$.
Consider the map 
$$\hat\vep_\nu=\hat\vep_{\nu_1}\circ\cdots\circ\hat\vep_{\nu_n}: e(\nu)\,\RL(\alpha)\,e(\nu)\to\bfk.$$
Recall that $\hat\vep_{\nu_k}$ is a map
$$\hat\vep_{\nu_k}: e(\nu^{(k)})R^\Lambda(\alpha-\sum_{j=k+1}^n\alpha_{\nu_j})e(\nu^{(k)})\to 
e(\nu^{(k-1)})R^\Lambda(\alpha-\sum_{j=k}^n\alpha_{\nu_j})e(\nu^{(k-1)}).$$
Next, we define an invertible element $r_\nu\in\bfk^\times_0$ by
\begin{equation}\label{rij}
r_\nu=\prod_{k<l}r_{\nu_k,\nu_l} \ \text{where }r_{ij}=\begin{cases} c_{i,j,-a_{ij},0} &\text{ if }j\neq i,\\
1 &\text{ if }j=i.
\end{cases}
\end{equation}

\begin{definition}\label{def:symmetrizing} We
define a $\bfk$-linear map $t_{\Lambda,\alpha}: \RL(\alpha)\to\bfk$ by setting
for all $\nu,\nu'\in I^\alpha$
$$t_{\Lambda,\alpha}\big(e(\nu)\,\bullet\,e(\nu')\big)=\begin{cases} 0 &\text{ if }  \nu\neq \nu',\\
r_\nu\,\hat\vep_\nu\big(e(\nu)\,\bullet\,e(\nu')\big) &\text{ if } \nu=\nu'.\end{cases}$$
\end{definition}

\smallskip

For $\nu\in I^\alpha$ we'll abbreviate 
\begin{equation}\label{tr} 
r(\alpha,\nu_n)=\prod_{k=1}^{n-1}r_{\nu_k,\nu_n}.
\end{equation}
By Corollary \ref{cor:vepdeg} the map $t_\alpha$ is homogenous of degree $-d_{\Lambda,\alpha}$.
Note that $r_\nu=r(\alpha,\nu_n)\,r_{\nu^{(n-1)}}$.
Therefore we have
\begin{equation}\label{eq:trec}
t_\alpha(a)=r(\alpha-\alpha_{\nu_n},\nu_n)\,t_{\alpha-\alpha_{\nu_n}}\big(\hat\vep_{\nu_n}(a)\big),\quad\forall a\in e(\nu)\RL(\alpha)e(\nu).
\end{equation}

We will prove that $t_\alpha$ is a symmetric form. 
By Theorem \ref{thm:Kbiadjoint}, the form $t_\alpha$ is nondegenerate. 
We must prove that for each $w,z\in \RL(\alpha)$ we have
\begin{equation}\label{eq:symmetric}
t_\alpha(zw)=t_\alpha(wz).
\end{equation} 
Without loss of generality, we may assume 
\begin{equation}\label{eq:zwnunu}
z\in e(\nu)\RL(\alpha)e(\mu),\quad w\in e(\mu)\RL(\alpha)e(\nu),
\end{equation}
the other cases being trivial.
We will prove \eqref{eq:symmetric} by induction on the height of $\alpha$. Assuming it holds for all $\alpha$ 
of height $n-1$, let us prove it for $\alpha$ of height $n$. We will write $\nu_n=i$, $\mu_n=j$, 
$\beta=\alpha-\alpha_i-\alpha_j$ and $\lambda=\Lambda-(\alpha-\alpha_i)$.

First, consider the case when $z$ belongs to the image of the map
$$\sigma_{ij}=\mu_{\tau_{n-1}}: \RL(\alpha-\alpha_i)e(\beta,j)\otimes_{\RL(\beta)} e(\beta,i)\RL(\alpha-\alpha_j)\to e(\alpha-\alpha_i,i)\RL(\alpha)e(\alpha-\alpha_j,j)$$
for $\beta=\alpha-\alpha_i-\alpha_j$ and $w$ belongs to the image of $\sigma_{ji}$.
Note that this is always the case if $i\neq j$ or if $i=j$ and $\lambda_i\leqslant 0$. 
In this situation, up to taking a linear combination, we may write
\begin{equation}\label{eq:zwtt}
z=\iota_i(z')\tau_{n-1}\iota_j(z''),\quad w=\iota_j(w')\tau_{n-1}\iota_i(w''),
\end{equation}
where $\iota_s$ is the canonical embedding $\RL(\alpha-\alpha_s)\to e(\alpha-\alpha_s,s)\RL(\alpha)e(\alpha-\alpha_s,s)$ for $s=i,j$ and 
$$\begin{aligned}
z'\in e(\nu^{(n-1)})\RL(\alpha-\alpha_i)e(\xi,j),\quad &z''\in e(\xi,i)\RL(\alpha-\alpha_j)e(\mu^{(n-1)})\\
w'\in e(\mu^{(n-1)})\RL(\alpha-\alpha_j)e(\eta,i),\quad &w''\in e(\eta,j)\RL(\alpha-\alpha_i)e(\nu^{(n-1)})
\end{aligned}$$
for some $\xi, \eta\in I^{\beta}$. 
By \eqref{eq:trec} and $\RL(\alpha-\alpha_i)$-bilinearity of $\hat\vep_i$ we have
$$\begin{aligned}
t_\alpha(zw)&=r_\nu\hat\vep_\nu(\iota_i(z')\tau_{n-1}\iota_j(z''w')\tau_{n-1}\iota_i(w''))\\
&=r(\alpha-\alpha_i,i)\,t_{\alpha-\alpha_i}\big(z'\hat\vep_i(\tau_{n-1}\iota_j(z''w')\tau_{n-1})w''\big).
\end{aligned}$$
Next, Lemma \ref{lem:relationsI} yields
$$\hat\vep_i(\tau_{n-1}\iota_j(z''w')\tau_{n-1})=r_{ij}\iota_j\hat\vep_i(z''w')+\delta_{ij}\sum_{g_1+g_2+g_3=
-\lambda_i-1}x_{n-1}^{g_1}B^{g_2}_{+i,\lambda}z''w'x_{n-1}^{g_3}.$$ Therefore
$t_\alpha(zw)=A(z,w)+B(z,w),$
where
$$\begin{aligned}
A(z,w)&=r(\alpha-\alpha_i,i)r_{ij}t_{\alpha-\alpha_i}\big(z'\iota_j\circ\hat\vep_i(z''w')w''\big)\\
B(z,w)&=\delta_{ij}r(\alpha-\alpha_i,i)t_{\alpha-\alpha_i}\big(\sum_{g_1+g_2+g_3=
-\lambda_i-1}z'x_{n-1}^{g_1}B^{g_2}_{+i,\lambda}z''w'x_{n-1}^{g_3}w''\big).
\end{aligned}$$
Thus, the formula \eqref{eq:symmetric} follows from the identities 
\begin{equation}\label{AB}
A(z,w)=A(w,z),\qquad B(z,w)=B(w,z). 
\end{equation}

Let us first prove \eqref{AB} for $A(z,w)$. We have
$$\begin{aligned}
t_{\alpha-\alpha_i}\big(z'\iota_j\circ\hat\vep_i(z''w')w''\big)
&=t_{\alpha-\alpha_i}(\iota_j\circ\hat\vep_i(z''w')w''z')\\
&=(\prod_{p=1}^{n-2}r_{\xi_p,j})t_{\beta}\big(\hat\vep_j(\iota_j\circ\hat\vep_i(z''w')w''z')\big)\\
&=(\prod_{p=1}^{n-2}r_{\xi_p,j})t_{\beta}\big(\hat\vep_i(z''w')\hat\vep_j(w''z')\big).
\end{aligned}$$
Here, the first equality is because $t_{\alpha-\alpha_i}$ symmetric by induction, the second one is given by \eqref{eq:trec}
since $\hat\vep_i(z''w')w''z'\in e(\xi,j)\RL(\alpha-\alpha_i)e(\xi,j)$, the third one
is the $\RL(\beta)$-bilinearity of $\hat\vep_j$.
Now, observe that
$r_{ij}(\prod_{p=1}^{n-2}r_{\xi_p,j})=\prod_{p=1}^{n-1}r_{\nu'_p,j}$.
We deduce
$$\begin{aligned}
A(z,w)=r(\alpha-\alpha_i,i)r(\alpha-\alpha_j,j)t_\beta\big(\hat\vep_i(z''w')\hat\vep_j(w''z')\big)
\end{aligned}$$
Exchanging $z$ and $w$ means also exchanging $i$ and $j$, $\nu$ and $\mu$. So the right hand side is symmetric with respect to $z$ and $w$. We deduce $A(z,w)=A(w,z)$.

Next, since $t_{\alpha-\alpha_i}$ is symmetric by the inductive hypothesis and $since B^{g_2}_{+i,\lambda}$ belongs to the center of $\RL(\alpha-\alpha_i)$, we have
$$B(z,w)=\delta_{ij}r(\alpha-\alpha_i,i)t_{\alpha-\alpha_i}\big(\sum_{g_1+g_2+g_3=-\lambda_i-1}B^{g_2}_{+i,\lambda}x_{n-1}^{g_1}z''w'x_{n-1}^{g_3}w''z'\big).$$
Exchanging $z$ and $w$ means also exchanging $i$ and $j$, $\nu$ and $\mu$. But here $B(z,w)\neq 0$ only when $i=j$. In this situation $r(\alpha-\alpha_i,i)=r(\alpha-\alpha_j,j)$.
We conclude that $B(z,w)=B(w,z)$. Hence, we have proven \eqref{eq:symmetric} when $z$, $w$ are both of the form \eqref{eq:zwtt}.

If $z$ and $w$ are not of the form \eqref{eq:zwtt}, then we must have $i=j$ and $\lambda_i>0$ as discussed 
in the paragraph before \eqref{eq:zwtt}. In this situation 
$z\in e(\alpha-\alpha_i,i)\RL(\alpha)e(\alpha-\alpha_i,i)$ can be uniquely written as 
$z=\mu_{\tau_{n-1}}(\pi(z))+\sum_{k=0}^{\lambda_i-1}p_k(z)x_n^k$, see Theorem \ref{thm:KK}. 
Similar for $w$. 
By linearity of $t_\alpha$ the remaining cases to be considered are 
\begin{itemize}
\item[(1)] $z=z_kx_n^k,\quad w=w'\tau_{n-1}w''$,
\item[(2)] $z=z_kx_n^k,\quad w=w_lx_n^l$,
\end{itemize}
for $z_k, w_l\in \RL(\alpha-\alpha_i)$, $k,l\in [0,\lambda_i-1]$, 
and $w'\in \RL(\alpha-\alpha_i)e(\beta,i), w''\in e(\beta,i)\RL(\alpha-\alpha_i)$.
Note that in both cases we have
$$\begin{aligned}
t_\alpha(zw-wz)&=r_\nu\hat\vep_\nu(zw)-r_\mu\hat\vep_\mu(wz)\\
&=r(\alpha,i)t_{\alpha-\alpha_i}(\hat\vep_i(zw-wz)).
\end{aligned}$$
Here, in the last equality, we used \eqref{eq:trec} and the fact that $i=j$. By induction, to prove $t_\alpha$ 
symmetric, it is enough to prove $\hat\vep_i(zw-wz)$ belongs to the commutator of $\RL(\alpha-\alpha_i)$. xx
We do this case by case

\begin{itemize}
\item We have $zw=z_kw'(x^k_n\tau_{n-1})e(\beta,i^2)w''$. Now
$$x^k_n\tau_{n-1}e(\beta,i^2)=\big(\tau_{n-1}x_{n-1}+\sum_{p+q=k-1}x_{n-1}^px_n^q\big)e(\beta,i^2).$$
Hence $\hat\vep_i(zw)=p_{\lambda_i-1}(zw)=0$. Similarly $\hat\vep_i(wz)=0$. 
\item We have $zw-wz=(z_kw_l-w_lz_k)x^{k+l}_n$. 
Hence $\hat\vep_i(zw-wz)=[z_k,w_l]\hat\vep_i(x_n^{k+l})$. Note that $\hat\vep_i(x_n^{k+l})$ belongs to the 
center of $\RL(\alpha-\alpha_i)$. So $\hat\vep_i(zw-wz)$ belongs to the commutator of 
$\RL(\alpha-\alpha_i)$.
\end{itemize}
The proof of Proposition \ref{prop:symmetrizing-form} is now complete.

\qed

\medskip

\section{Relations}\label{app:relations}

In this section we prove Theorem \ref{thm:action}.

\subsection{A useful lemma}
Let $z\in e(\alpha,i)\RL(\alpha+\alpha_i)e(\alpha,i)$.
Recall that
\begin{itemize}
\item
if $\lambda_i\geqslant 0$,  then $z$   can 
be uniquely written as
$$z=\mu_{\tau_n}(\pi(z))+\sum_{k=0}^{\lambda_i-1}p_k(z)x_{n+1}^k,$$
\item
if $\lambda_i\leqslant 0$, there are $\tilde z$, $\tilde \pi_k\in\RL(\alpha)e(\alpha-\alpha_i,i)\otimes_{\RL(\alpha-\alpha_i)}e(\alpha-\alpha_i,i)\RL(\alpha)$ 
such that $\mu_{\tau_n}(\tilde z)=z$, $\mu_{x_n^p}(\tilde z)=\mu_{\tau_n}(\tilde\pi_k)=0$ and $\mu_{x_n^p}(\tilde\pi_k)=\delta_{k,p}$ for $k,p\in[0,-\lambda_i-1]$.
\end{itemize}

\smallskip

To prove Theorem \ref{thm:action} we'll need the following technical result.

\smallskip

\begin{lemma}\label{lem:recrec}
For each $r\in\bbN,$ we have

(a) if $\lambda_i>0$ then 
 \begin{align*}
 &\pi\big(x^{r}_{n+1}e(\alpha,i)\big)=\sum_{a=0}^{r-\lambda_i}(B^{r-\lambda_i-a}_{+i,\lambda}\otimes x_{n}^a\otimes 1\otimes 1)(-\hat\eta'_i(1)),\\
 &p_{k}(x_{n+1}^{r}e(\alpha,i))=\sum_{a=0}^{\lambda_i-k-1}B^{r-k-a}_{+i,\lambda}B^{a}_{-i,\lambda},\quad \forall k\in [0,\lambda_i-1],
  \end{align*}

 (b) if $\lambda_i\leqslant 0$ then
  \begin{align*}
 &\mu_{x_n^r}(\tilde z)=\sum_{p=0}^{r+\lambda_i}\hat\vep'_{i,\lambda}(zx_{n+1}^p)B_{-i,\lambda}^{r+\lambda_i-p},\quad \forall z\in e(\alpha,i)\RL(\alpha+\alpha_i)e(\alpha,i),\\
 &\mu_{x_n^r}(\tilde\pi_k)=\sum_{a=k}^{-\lambda_i-1}B^{a-k}_{+i,\lambda}B^{r-a}_{-i,\lambda},\quad \forall k\in [0,-\lambda_i-1],
 \end{align*}
 
 (c) $\sum_{a=0}^{r}B_{+i,\lambda}^{r-a}B_{-i,\lambda}^{a}=\delta_{r,0}.$

\end{lemma}

\begin{proof}
First, assume $\lambda_i>0$.
To simplify notation, we write $x^r_{n+1}=x^r_{n+1}e(\alpha,i)$. We have
\begin{align*}
x^{r+1}_{n+1}&=x_{n+1}(\mu_{\tau_n}\big(\pi(x^r_{n+1}))+\sum_{k=0}^{\lambda_i-1}p_k(x_{n+1}^r)x_{n+1}^k\big)\\
&=\mu_{x_{n+1}\tau_n}(\pi(x^r_{n+1}))+\sum_{k=0}^{\lambda_i-1}p_k(x_{n+1}^r)x_{n+1}^{k+1}\\
&=\mu_{\tau_nx_n}(\pi(x^r_{n+1}))+\vep'_i(\pi(x^r_{n+1}))+\sum_{k=0}^{\lambda_i-1}p_k(x_{n+1}^r)x_{n+1}^{k+1}.
\end{align*}

Here in the last equality we have used $(x_{n+1}\tau-\tau_nx_n-1)e(\alpha-\alpha_i,i^2)=0$ and $\mu_1=\vep'_i$. It follows that
\begin{align}
\pi(x^{r+1}_{n+1})&=(1\otimes x_n\otimes 1\otimes 1)\pi(x_{n+1}^r)+p_{\lambda_i-1}(x^r_{n+1})\pi(x_{n+1}^{\lambda_i}),\label{eq:pi}\\
p_0(x^{r+1}_{n+1})&=\vep'_i(\pi(x^r_{n+1}))+p_{\lambda_i-1}(x^r_{n+1})p_0(x_{n+1}^{\lambda_i}),\label{eq:p0}\\
p_k(x^{r+1}_{n+1})&=p_{k-1}(x^r_{n+1})+p_{\lambda_i-1}(x^r_{n+1})p_k(x_{n+1}^{\lambda_i}),\quad \forall\,k\in [1,\lambda_i-1]\label{eq:pk}.
\end{align}
Now, recall that $p_{\lambda_i-1}(x_{n+1}^{r+\lambda_i-1})=\hat\vep'_{i,\lambda}(x_{n+1}^{r+\lambda_i-1})=B_{+i,\lambda}^r$, $p_k(x^{\lambda_i}_{n+1})=B_{-i,\lambda}^{\lambda_i-k}$ and $\hat\eta'_{i,\lambda}=-\pi(x^{\lambda_i}_{n+1})$. 

\smallskip

If $r<\lambda_i$ the first equality in part $(a)$  is trivial with both sides being zero. If $r\geqslant\lambda_i$, it follows recursively from \eqref{eq:pi}.
\smallskip

Next, by applying recursively \eqref{eq:pk} we obtain
\begin{align*}
p_k(x^{r+\lambda_i}_{n+1})=p_0(x^{r+\lambda_i-k}_{n+1})-\sum_{a=0}^{k-1}B_{+i,\lambda}^{r-a}B_{-i,\lambda}^{\lambda_i-k+a}.
\end{align*}
Substitute $p_0(x^{r+\lambda_i-k}_{n+1})$ using \eqref{eq:p0} gives 
$$p_{k}(x_{n+1}^{r+\lambda_i})=\vep'_i(\pi(x^{r+\lambda_i-k-1}_{n+1}))-\sum_{a=0}^{k}B_{+i,\lambda}^{r-a}B_{-i,\lambda}^{\lambda_i-k+a}.$$
Apply this to the special case $k=\lambda_i-1$ we get
\begin{align}\label{eq:vep'}
\vep'_i(\pi(x^{r}_{n+1}))=\sum_{a=0}^{\lambda_i}B_{+i,\lambda}^{r+1-a}B_{-i,\lambda}^{a}.
\end{align}
Therefore we deduce
\begin{align*}
p_{k}(x_{n+1}^{r+\lambda_i})&=\sum_{a=0}^{\lambda_i}B_{+i,\lambda}^{r+\lambda_i-k-a}B_{-i,\lambda}^{a}-\sum_{a=\lambda_i-k}^{\lambda_i}B_{+i,\lambda}^{r+\lambda_i-k-a}B_{-i,\lambda}^{a}\\
&=\sum_{a=0}^{\lambda_i-1-k}B_{+i,\lambda}^{r+\lambda_i-k-a}B_{-i,\lambda}^{a}.
\end{align*}
This proves the second equality in part $(a)$ for $r\geqslant \lambda_i$. 

\smallskip

On the other hand, we have
\begin{align*}
\vep'_i(\pi(x^{r}_{n+1}))=\sum_{a=0}^{r-\lambda_i}B^{r-\lambda_i-a}_{+i,\lambda}\mu_{x_n^a}(-\hat\eta'_i(1))
=-\sum_{a=\lambda_i+1}^{r+1}B^{r+1-a}_{+i,\lambda}B^{a}_{-i,\lambda}.
\end{align*}
by the first equality in part $(a)$. Combined with \eqref{eq:vep'} it gives part (c) for $r>0$. The case $r=0$ is obvious.

Finally, for $r\leqslant\lambda_i-1$ we have $p_{k}(x_{n+1}^{r}e(\alpha,i))=\delta_{k,r}$ and
$$\sum_{a=0}^{\lambda_i-k-1}B^{r-k-a}_{+i,\lambda}B^{a}_{-i,\lambda}=\sum_{a=0}^{r-k}B^{r-k-a}_{+i,\lambda}B^{a}_{-i,\lambda}=\delta_{r,k}$$
by part $(c)$. We deduce the second equality in part (a) for $r\leqslant\lambda_i-1$.

\smallskip

The case $\lambda_i\leqslant 0$ is proved by a computation of similar style. We only indicate some key steps. First, one checks by a direct computation that
$$\widetilde{zx_{n+1}}=(1\otimes x_n\otimes 1\otimes 1)\tilde z-\hat\vep'_{i,\lambda}(z)\hat\eta'_i(1).$$ Applying it recursively we get
$$\widetilde{zx^r_{n+1}}=(1\otimes x^r_n\otimes 1\otimes 1)\tilde z-\sum_{p=0}^{r-1}\hat\vep'_{i,\lambda}(zx_{n+1}^p)(1\otimes x^{r-1-p}_n\otimes 1\otimes 1)\hat\eta'_i(1).$$
The first equality in $(b)$ is obtained by applying $\mu_{x_n^{-\lambda_i}}$ to both sides of the above equality with $r$ replaced by $r+\lambda_i$.
To prove the second equality, observe that for $k,p\in [0,-\lambda_i-2]$, write $A=(1\otimes x_n\otimes 1\otimes 1)\tilde\pi_{k+1}$ we have $\mu_{\tau_n}(A)=0$,
$\mu_{x^p_n}(A)=\delta_{k,p}$, and 
$\mu_{x^{-\lambda_i-1}_n}(A)=-B_{-i,\lambda}^{-\lambda_i-k-1}$.
We deduce that 
$$\tilde\pi_k=\sum_{p=0}^{-\lambda_i-1-k}B_{-i,\lambda}^{-\lambda_i-k-1-p}(1\otimes x_n^{p}\otimes 1\otimes 1)\hat\eta'_{i,\lambda}(1).$$
Now, apply $\mu_{x_n^r}$ to both sides we get the second equality in $(b)$.
Finally, to get $(c)$, observe that the first equality applied to $z=1$ yields
$$\mu_{x_n^r}(\tilde 1)=\sum_{p=-\lambda_i+1}^{r+1}B^p_{+i,\lambda}B^{r+1-p}_{-i,\lambda}.$$ On the other hand, it is easy to check that
$-\tilde 1=(1\otimes x_n\otimes 1\otimes 1)\tilde\pi_0+B^{-\lambda_i}_{+i,\lambda}\hat\eta'_{i,\lambda}(1)$.
Hence 
$$-\mu_{x^r_n}(\tilde 1)=\mu_{x_n^{r+1}}(\tilde\pi_0)+B^{-\lambda_i}_{+i,\lambda}B^{r+1+\lambda_i}_{-i,\lambda}=\sum_{p=0}^{-\lambda_i}B^p_{+i,\lambda}B^{r+1-p}_{-i,\lambda}$$
by the second equality in $(b)$. Combining the two equalities gives $(c)$.
\end{proof}

\medskip

\subsection{The Cartan loop operators}\label{sec:CLO}

Consider the following formal power series
$$B_{\pm i,\lambda}(z)=\sum_{k\geqslant 0}B^k_{\pm i,\lambda}\,z^{k}\in \Z(\RL(\alpha))[[z]].$$
The aim of this section is to express the coefficients of the formal series $B_{\pm i,\lambda}(z)$ as elements in the image of the canonical map 
$Z(R(\alpha))\to Z(R^\Lambda(\alpha))$, see Proposition \ref{prop:A10} for details.
To do that, fix formal variables $y_{i,1},\dots,y_{i,\Lambda_i}$ of degree 2 such that 
\begin{equation}\label{a2}a^\Lambda_i(u)=\prod_{p=1}^{\Lambda_i}(u+y_{ip})
\quad \text{with}\quad
c_{ip}=e_p(y_{i,1},\dots,y_{i,\Lambda_i}).
\end{equation}
Consider the following formal series in $\bfk[[u,v]]$
\begin{equation}q_{ij}(u, v)=\begin{cases} u^{-a_{ij}}\,Q_{ij}(u^{-1}, v)/c_{i,j,-a_{ij},0} &\text{ if } i\neq j,\\
(1-uv)^{-2}&\text{ else.}
\end{cases}
\end{equation}

\smallskip

\begin{lemma}\label{lem:A10} 
For each $\alpha\in Q_+$ of height $n$, we have

(a) $B_{+i,\lambda}(z)\,B_{-i,\lambda}(z)=1$,

(b) $\big(
B_{\pm i,\lambda}(z)
-q_{ij}(z, x_n)^{\mp 1}\,B_{\pm i,\lambda+\alpha_j}(z)\big)\,
e(\alpha-\alpha_j,j)=0$ for each $i,j$.
\end{lemma}

\begin{proof}
Part $(a)$ will be proved in Lemma \ref{lem:recrec}$(c)$ below.
Now, we concentrate on $(b)$.
By $(a)$ it is enough to prove it
for $B_{+i,\lambda}(z)$. Write $\beta=\alpha-\alpha_j$ and $\lambda'=\Lambda-\beta=\lambda+\alpha_j$.

\smallskip

First, assume that $i=j$. 
Recall that
$$\varphi_ne(\beta,i^2)=(x_n\tau_n-\tau_nx_n)e(\beta,i^2)\in \RL(\alpha+\alpha_i).$$
We have 
$$\begin{aligned}
x^k_{n+1}e(\beta,i^2)&=\varphi_nx_n^k\varphi_ne(\beta,i^2)\\
&=(x_n\tau_n-\tau_nx_n)x_n^k(x_n\tau_n-\tau_nx_n)e(\beta,i^2)\\
&=\big(x_n(\tau_nx_n^{k+1}\tau_n)-\tau_nx_n^{k+2}\tau_n-x_n(\tau_nx_n^k\tau_n)x_n+
(\tau_nx_n^{k+1}\tau_n)x_n\big)e(\beta,i^2).
\end{aligned}$$
By Lemma \ref{lem:relationsI}$(b)$, for all $\ell\in\bbN$ we have
$$\begin{aligned}
\hat\vep'_{i,\lambda}(\tau_nx_n^\ell\tau_n\,e(\beta,i^2))
&=\bbX_{i,i,\lambda'}(x^\ell_n\,e(\beta,i))\\
&=\big(-\bbI_{i,i,\lambda'}+\sum_{g_1+g_2+g_3=
-\lambda_i-1}x^{g_1}B^{g_2}_{+i,\lambda}x^{g_3}\big)(x^\ell_n\,e(\beta,i))\\
&=-\hat\vep'_{i,\lambda'}(x^\ell_n\,e(\beta,i))+\sum_{g_1+g_2+g_3=
-\lambda_i-1}x_n^{\ell+g_1}B^{g_2}_{+i,\lambda}x_n^{g_3}.
\end{aligned}$$
If $k\geqslant -\lambda_i+1$ then $k-2\geqslant -\lambda_i'+1$, we have 
$$B^k_{+i,\lambda}=\hat\vep'_{i,\lambda}(x_{n+1}^{\lambda_i-1+k}\,e(\alpha,i)),\quad
B^{l}_{+i,\lambda'}=\hat\vep'_{i,\lambda'}(x_{n}^{\lambda_i-3+l}\,e(\beta,i)),\quad\forall\ l\geqslant k-2.$$
In this situation
\begin{align}
B^k_{+i,\lambda}e(\beta,i)&
=\hat\vep'_{i,\lambda}(x_{n+1}^{\lambda_i-1+k}\,e(\alpha,i))\,e(\beta,i)\nonumber\\
&=\hat\vep'_{i,\lambda}(x_{n+1}^{\lambda_i-1+k}\,e(\beta,i^2))\nonumber\\
&=-2x_n\hat\vep'_{i,\lambda'}(x^{\lambda_i+k}_n\,e(\beta,i))
+\hat\vep'_{i,\lambda'}(x^{\lambda_i+k+1}_n\,e(\beta,i))+\nonumber\\
&\quad+x_n^2\hat\vep'_{i,\lambda'}(x^{\lambda_i+k-1}_n\,e(\beta,i))\nonumber\\
&=\big(B^k_{+i,\lambda'}-2x_nB^{k-1}_{+i,\lambda'}+x^2_nB^{k-2}_{+i,\lambda'}\big)e(\beta,i)\label{eq:bete}.
\end{align}
In particular, this yields
$$B_{+i,\lambda}(z)\,e(\beta,i)=
(1-x_nz)^{2}B_{+i,\lambda'}(z)\,e(\beta,i)\ \text{if}\ \lambda_i\geqslant 0.$$
If $\lambda_i<0$ we must also check \eqref{eq:bete} for $k\leqslant -\lambda_i$.
Consider the element
$$A_k=(1\otimes\tau_{n-1}\otimes\tau_{n-1}\otimes 1)(\iota_i\otimes\iota_i)\big(\tilde\pi_{-\lambda'_i-k}-2x_n\tilde\pi_{-\lambda_i'-(k-1)}+x^2_n\tilde\pi_{-\lambda_i'-(k-2)}\big)$$
in $\RL(\alpha)e(\beta,i)\otimes_{\RL(\beta)}e(\beta,i)\RL(\alpha)$.
Here $\tilde\pi_{-\lambda'_i-l}\in\RL(\beta)e(\beta-\alpha_i,i)\otimes_{\RL(\beta-\alpha_i)}e(\beta-\alpha_i,i)\RL(\beta)$ is the element defined in \eqref{rk:centralpi} 
for $l\in [1,-\lambda'_i]$, and we have set $\tilde\pi_{-1}=-\widetilde{e(\beta,i)}$ and $\tilde\pi_{-2}=-\widetilde {x_{n}e(\beta,i)}$, where $e(\beta,i)$ 
and $x_{n}e(\beta,i)$ are viewed as elements in $e(\beta,i)R(\alpha)e(\beta,i).$
One can check by direct computation that $\mu_{\tau_n}(A_k)=0$, $\mu_{x_n^a}(A_k)=\delta_{a,-\lambda_i-k}e(\beta,i)$ for $a\in [0,-\lambda_i-1]$,
 and that $$\mu_{x_n^{-\lambda_i}}(A_k)=\big(B^k_{+i,\lambda'}-2x_nB^{k-1}_{+i,\lambda'}+x^2_nB^{k-2}_{+i,\lambda'}\big)e(\beta,i).$$
It follows that
 $A_k=\tilde\pi_{-\lambda_i-k}e(\beta,i)$ and 
$$B^k_{+i,\lambda}e(\beta,i)=\mu_{x_n^{-\lambda_i}}(A_k)=\big(B^k_{+i,\lambda'}-2x_nB^{k-1}_{+i,\lambda'}+x^2_nB^{k-2}_{+i,\lambda'}\big)e(\beta,i).$$
The proof for part $(b)$ in the case $i=j$ is complete.

\smallskip

Finally, assume that $i\neq j$. By relation $(d)$ in QHA, we have
\begin{align*}
\tau_nx_n^k\tau_ne(\beta,ji)=x_{n+1}^k\tau_n^2e(\beta,ji)
=\sum_{p,q}c_{i,j;p,q}x_{n+1}^{p+k}x_n^qe(\beta,ji).
\end{align*}
Applying $\hat\vep'_{i,\lambda}$ to both sides of the equality, we get 
$$c_{i,j,-a_{ij},0}\ \hat\vep'_{i,\lambda'}(x^k_n)e(\beta,j)=\sum_{p,q}c_{i,j,p,q}\hat\vep'_{i,\lambda}(x^{p+k}_{n+1}e(\alpha,i))x^q_ne(\beta,j)$$
by Lemma \ref{lem:relationsI}(a).
Hence if $k\geqslant -\lambda_i+1$ we get
$$c_{i,j,-a_{ij},0}\ B^{k-a_{ij}}_{+i,\lambda'}e(\beta,j)=\sum_{p,q}c_{i,j,p,q}B^{k+p}_{+i,\lambda}x^q_ne(\beta,j).$$
Now, assume $k\leqslant -\lambda_i$. Then $k-a_{ij}\leqslant -\lambda_i'$. In this case 
$B_{+i,\lambda'}^{k-a_{ij}}=-\mu_{x_{n-1}^{-\lambda_i'}}(\tilde\pi_{-\lambda_i-k})$ with $\tilde\pi_{-\lambda_i-k}\in\RL(\beta)e(\beta-\alpha_i,i)\otimes_{\RL(\beta-\alpha_i)}e(\beta-\alpha_i,i)\RL(\beta)$. Set $$A_k=(1\otimes\tau_{n-1}\otimes\tau_{n-1}\otimes 1)(\iota_j\otimes\iota_j)(\tilde\pi_{-\lambda_i-k}).$$
Then we have the following equalities in $e(\alpha,i)\RL(\alpha+\alpha_i)e(\alpha,i)$,
\begin{align*}
\mu_{\tau_n}(A_k)&=\mu_{\tau_{n-1}\tau_n\tau_{n-1}e(\beta-\alpha_i,iji)}(\tilde\pi_{-\lambda_i-k})\\
&=\mu_{(\tau_{n}\tau_{n-1}\tau_{n}-\sum_{p,q}c_{i,j,p,q}(\sum_{a=0}^{p-1}x_{n-1}^{a}x_{n+1}^{p-1-a})x_n^q)e(\beta-\alpha_i,iji)}(\tilde\pi_{-\lambda_i-k})\\
&=-\sum_{p,q}c_{i,j,p,q}\big(\sum_{a=0}^{p-1}\mu_{x_{n-1}^a}(\tilde\pi_{-\lambda_i-k})x_{n+1}^{p-1-a}\big)x_n^qe(\beta,ji),
\end{align*}
because $\mu_{\tau_{n}\tau_{n-1}\tau_{n}}(\tilde\pi_{-\lambda_i-k})=\tau_n\mu_{\tau_{n-1}}(\tilde\pi_{-\lambda_i-k})\tau_n=0$.
Since $\mu_{x_{n-1}^a}(\tilde\pi_{-\lambda_i-k})=\delta_{a,-\lambda_i-k}$ for $a\in [0,-\lambda_i'-1]$, for any $p\in [0,-a_{ij}]$ we have
$$\sum_{a=0}^{p-1}\mu_{x_{n-1}^a}(\tilde\pi_{-\lambda_i-k})x_{n+1}^{p-1-a}=\begin{cases}
-\sum_qc_{i,j,p,q}x_n^qx_{n+1}^{p-1+\lambda_i+k}e(\beta,ji),&\text{ if } p>-\lambda_i-k+1\\
0,&\text{ otherwise. }
\end{cases}$$
Next, for any positive integer $l$ we have
\begin{align*}
\mu_{x^l_n}(A_k)&=\mu_{\tau_{n-1}x^l_n\tau_{n-1}e(\beta-\alpha_i,ij)}(\tilde\pi_{-\lambda_i-k})\\
&=\mu_{(x_{n-1}^l\sum_{p,q}c_{i,j,p,q}x_{n-1}^px_n^q)e(\beta-\alpha_i,ij)}(\tilde\pi_{-\lambda_i-k})\\
&=\sum_{p,q}c_{i,j,p,q}\,\mu_{x_{n-1}^{l+p}}(\tilde\pi_{-\lambda_i-k})x_n^qe(\beta,j).
\end{align*}
In particular, since $\mu_{x_{n-1}^a}(\tilde\pi_{-\lambda_i-k})=\delta_{a,-\lambda_i-k}$ for $a\in [0,-\lambda_i'-1]$, we get 
\begin{equation}\label{eq:eqeqeq}
\mu_{x_n^{-\lambda_i}}(A_k)=c_{i,j,-a_{ij},0}\,\mu_{x_{n-1}^{-\lambda'_i}}(\tilde\pi_{-\lambda_i-k})e(\beta,j)
=-c_{i,j,-a_{ij},0}\ B^{k-a_{ij}}_{+i,\lambda'}e(\beta,j)
\end{equation}
and for $l\in [0,-\lambda_i-1]$ we have
\begin{align*}
\mu_{x_n^l}(A_k)=\begin{cases} \sum_qc_{i,j,(-\lambda_i-k-l),q}x_n^qe(\beta,j), &\text {if } l\in [\min\{0,-\lambda_i-k+a_{ij}\},-\lambda_i-k],\\
0, &\text{ otherwise. }
\end{cases}
\end{align*} 
It follows that 
$$A_k=-\sum_{p=-\lambda_i-k+1}^{-a_{ij}}\sum_qc_{i,j,p,q}x_n^q(\widetilde{x_{n+1}^{p-1+\lambda_i+k}})e(\beta,ji)+\sum_{p=0}^{-\lambda_i-k}\sum_qc_{i,j,p,q}x_n^qe(\beta,j)\tilde\pi_{-\lambda_i-p-k}$$
in $\RL(\alpha)e(\alpha-\alpha_i,i)\otimes_{\RL(\alpha-\alpha_i)}e(\alpha-\alpha_i,i)\RL(\alpha)$.
Apply $\mu_{x_n^{-\lambda_i}}$ to both sides of the equation, by \eqref{eq:eqeqeq} and Definition \ref{df:bubbles}
we get 
$$c_{i,j,-a_{ij},0}\ B^{k-a_{ij}}_{+i,\lambda'}e(\beta,j)=\sum_{p,q}c_{i,j,p,q}B^{k+p}_{+i,\lambda}x^q_ne(\beta,j).$$
The proof for part (b) is now complete.
\end{proof}

\medskip

We can now prove the main result of this section.

\smallskip

\begin{proposition}\label{prop:A10} 
For each $\alpha\in Q_+$ of height $n$ we have
$$B_{\pm i,\lambda}(z)=z^{\mp\Lambda_i} a_i^\Lambda(z^{-1})^{\mp 1}
\sum_{\nu\in I^\alpha}\prod_{k=1}^nq_{i\nu_k}(z,x_k)^{\mp 1}\,e(\nu).$$
\end{proposition}

\begin{proof}
The relation $a^\Lambda_i(x_1)=0$ yields 
$x^{\Lambda_i}_1=-\sum_{k=0}^{\Lambda_i-1}e_k(y_{i,1},\dots ,y_{i,\Lambda_i})x^{\Lambda_i-k}_1$.
Therefore, for $k\in [1,\Lambda_i]$ we have 
$B^k_{-i,\Lambda}=-p_{\Lambda_i-k}(x^{\Lambda_i}_1)=e_k(y_{i,1},\dots ,y_{i,\Lambda_i})$,
and for $k>\Lambda_i$, since $\eta'_{i,\Lambda}=-\pi(x_1^{\Lambda_i})=0,$ 
we deduce $B^k_{-i,\Lambda}=0$.
So we have 
$$\begin{aligned}
B_{-i,\Lambda}(z)=\sum_{p=0}^{\Lambda_i}B_{-i,\Lambda}^pz^{p}
=\sum_{p=0}^{\Lambda_i}e_p(y_{i,1},\dots ,y_{i,\Lambda_i})z^{p}
=\prod_{p=1}^{\Lambda_i}(1+y_{ip}\,z).
\end{aligned}$$
We deduce that
$B_{\pm i,\Lambda}(z)=\prod_{p=1}^{\Lambda_i}(1+y_{ip}\,z)^{\mp 1}.$
Now, by Lemma \ref{lem:A10}, we have
$$\begin{aligned}
B_{\pm i,\lambda}(z)
&=B_{\pm i,\lambda'}(z)\,\sum_jq_{ij}(z, x_n)^{\mp 1}\,
e(\alpha-\alpha_j,j),\\
&=B_{\pm i,\Lambda}(z)\,
\sum_{\nu\in I^\alpha}\prod_{k=1}^nq_{i\nu_k}(z,x_k)^{\mp 1}\,e(\nu).
\end{aligned}$$

\end{proof}

\medskip

\subsection{Proof of Theorem \ref{thm:action}}
Consider the operator $h_{ir}\in\End_\bfk(\Tr(\calC/\bbZ))$ which acts on 
$\Tr(R^\Lambda(\alpha))$ by multiplication by the central element
$$h_{ir,\lambda}=\sum_{k=0}^r(\lambda_i-k)\, B_{+i,\lambda}^{r-k}\,B_{-i,\lambda}^k$$
in $\RL(\alpha)$. Then, define the following formal series
\begin{equation}\label{formal series}
\Psi_i(z)=\sum_{r\geqslant 0}\psi_{ir}\,z^r=\exp\big(-\sum_{r\geqslant 1}h_{ir}\,z^r/r\big),
\qquad
H_i(z)=\sum_{r\geqslant 0}h_{ir}\,z^r.
\end{equation}
The following holds.

\smallskip

\begin{lemma}\label{lem:psi-B}
The operator $\psi_{ir}$ acts on $\Tr(R^\Lambda(\alpha))$ 
by multiplication by $B^r_{-i,\lambda}$.
\end{lemma}

\begin{proof}
By definition, the operator $h_{ir}$ acts by multiplication by the element 
$$\begin{aligned}
h_{ir,\lambda}&=\sum_{k=0}^r(\lambda_i-k)\, B_{+i,\lambda}^{r-k}\,B_{-i,\lambda}^k.
\end{aligned}$$
Since $B_{\pm i,\lambda}(z)=\sum_{k\geqslant 0}B^k_{\pm i,\lambda}\,z^{k}$ and since
$B_{+i,\lambda}(z)\,B_{-i,\lambda}(z)=1$ by Lemma \ref{lem:A10},
we deduce that 
$$\begin{aligned}
H_{i,\lambda}(z)
&=\lambda_i-z\,\text{d}/\text{d}z\log B_{-i,\lambda}(z).\\
\end{aligned}$$
Now, from \eqref{formal series} we get $$H_{i}(z)=h_{i0}-z\,\text{d}/\text{d}z\log \Psi_{i}(z).$$
Hence, the formal series $\Psi_i(z)$ acts on $\Tr(R^\Lambda(\alpha))$ by multiplication
by $B_{-i,\lambda}(z)$.

\end{proof}

\smallskip

Finally, let $a_{i,j,p,q}\in\bfk$ be such that
$$q_{ij}(u,v)=\sum_{p,q\geqslant 0} a_{i,j,p,q}\,u^pv^q.$$
We can now prove the following.

\smallskip

\begin{proposition}\label{prop:B5}
For each $i,j\in I$, $r,s\in\bbN$, we have

\begin{itemize}
\item[$(a)$] $[h_{ir},h_{js}]=0$,

\item[$(b)$] $[x_{ir}^+,x_{js}^-]=\delta_{ij}\,h_{i,r+s}$,

\item[$(c)$] $\psi_{ir}\,x^-_{js}=\sum_{p,q\geqslant 0}a_{i,j,p,q}\,x^-_{j,s+q}\,\psi_{i,r-p}$ and
$x^+_{js}\,\psi_{ir}=\sum_{p,q\geqslant 0}a_{i,j,p,q}\,\psi_{i,r-p}\,x^+_{j,s+q}$,

\item[$(d)$] $\sum_{p,q\geqslant 0}c_{i,j,p,q}\,[x^\pm_{i,r+p},x^\pm_{j, s+q}]=0$ if $i\neq j$,

\item[$(e)$] $[x^\pm_{i,r},x^\pm_{i,s}]=0,$

\item[$(f)$] $[x^\pm_{i,r_{1}},[x^\pm_{i,r_{2}},\dots [x^\pm_{i,r_{m}},x^\pm_{j,s}]\dots]]=0$\ with $i\neq j$, $r_p\in\bbN$, $m=1-a_{ij}$.
\end{itemize}

\end{proposition}

\begin{proof} The first relation is obvious. Let us concentrate on part $(b)$.
If $i\neq j$ we have an isomorphism $\sigma_{ji,\lambda}: F_jE_i1_\lambda\simeq E_iF_j1_\lambda$ such that $\sigma_{ji,\lambda}(x_j^sx_i^r)\sigma^{-1}_{ji,\lambda}=x_i^rx_j^s$. Hence by Lemma \ref{lem:traceop} we have $\Tr_{F_jE_i1_\lambda}(x_j^sx_i^r)=\Tr_{E_iF_j1_\lambda}(x_i^rx_j^s)$, which is $[x^+_{ir},x^-_{js}]=0$.

Now consider the case $i=j$.
First, assume $\lambda_i> 0$. Let $G=F_iE_i1_\lambda\oplus 1_\lambda^{\oplus \lambda_i}$. Recall the isomorphism of functors $\rho_{i,\lambda}:G\to E_iF_i1_\lambda$. By Lemma \ref{lem:traceop} we have 
$$x^+_{ir}x^-_{is}=\Tr_{E_iF_i}(x_i^rx_i^s)=\Tr_G(\rho_{i,\lambda}^{-1}(x_i^rx_i^s)\rho_{i,\lambda})$$
and it is equal to the sum of the trace of $\rho_{i,\lambda}^{-1}(x_i^rx_i^s)\rho_{i,\lambda}$ restricted to each direct factor of $G$. The restriction of $\rho_{i,\lambda}^{-1}(x_i^rx_i^s)\rho_{i,\lambda}$ to $F_iE_i1_\lambda$ is represented by
\begin{align*}
\RL(\alpha)e(\alpha-\alpha_i,i)\otimes_{\RL(\alpha-\alpha_i)}e(\alpha-\alpha_i,i)\RL(\alpha)&\to\RL(\alpha)e(\alpha-\alpha_i,i)\otimes_{\RL(\alpha-\alpha_i)}e(\alpha-\alpha_i,i)\RL(\alpha)\\
z&\mapsto \pi(x_{n+1}^r\mu_{\tau_n}(z)x_{n+1}^s).
\end{align*}
Now, we have
 \begin{align*}
 \pi(x_{n+1}^r\mu_{\tau_n}(z)x_{n+1}^s)&=\pi(\mu_{x_{n+1}^r\tau_nx_{n+1}^s}(z))\\
 &=\pi\big(\mu_{x_n^s\tau_nx_n^r}(z)+\sum_{p=0}^{r+s-1}\mu_{x_n^{r+s-1}}(z)x_{n+1}^p\big)\\
 &=(1\otimes x_n^s\otimes x_n^r\otimes 1)(z)+\sum_{p=\lambda_i}^{r+s-1}\mu_{x_n^{r+s-1-p}}(z)\pi(x_{n+1}^p)\\
 &=(1\otimes x_n^s\otimes x_n^r\otimes 1)(z)-\sum_{p=\lambda_i}^{r+s-1}\sum_{a=0}^{\ p-\lambda_i}\mu_{x_n^{r+s-1-p}}(z)(B^{p-\lambda_i-a}_{+i,\lambda}\otimes x_n^a\otimes 1\otimes 1)\hat\eta'_i(1).
\end{align*}
Here we used the relation $(e)$ of QHA to get the second equality, and Lemma \ref{lem:recrec} for the last equality. It yields that the restriction of $\rho_{i,\lambda}^{-1}(x_i^rx_i^s)\rho_{i,\lambda}$ to $F_iE_i1_\lambda$ is the endomorphism $$x_i^sx_i^r-\sum_{p=\lambda_i}^{r+s-1}\sum_{a=0}^{\ p-\lambda_i}(B^{p-\lambda_i-a}_{+i,\lambda}F_ix_i^a)\circ\hat\eta'_i\circ\vep'_i\circ(F_ix_i^{r+s-1-p}),$$
and its trace is equal to
\begin{align}\label{eq:resfe}
&\Tr_{F_iE_i1_\lambda}(x_i^sx_i^r)-\Tr_{1_\lambda}\big(\sum_{p=\lambda_i}^{r+s-1}\sum_{a=0}^{\ p-\lambda_i}B^{p-\lambda_i-a}_{+i,\lambda}\vep'_i\circ(F_ix_i^{r+s-1-p+a})\circ\hat\eta'_i\big)=\nonumber\\
=&x^-_{is}x^+_{ir}-\sum_{p=\lambda_i}^{r+s-1}\sum_{a=0}^{\ p-\lambda_i}B^{p-\lambda_i-a}_{+i,\lambda}B^{r+s-p+a+\lambda_i}_{-i,\lambda}\nonumber\\
=&x^-_{is}x^+_{ir}+\sum_{a=\lambda_i+1}^{r+s}(\lambda_i-a)B^{r+s-a}_{+i,\lambda}B^{a}_{-i,\lambda}.
\end{align}
The restriction of $\rho_{i,\lambda}^{-1}(x_i^rx_i^s)\rho_{i,\lambda}$ to the $k$-th copy of $1_\lambda$ is represented by the map
$\RL(\alpha)\to\RL(\alpha),$ $z\mapsto p_k(x_{n+1}^rzx_{n+1}^{s+k})=zp_k(x_{n+1}^{r+s+k}).$
By  the second equality in Lemma \ref{lem:recrec}$(a)$ it is equal to
$\sum_{a=0}^{\lambda_i-k-1}B^{r+s-a}_{+i,\lambda}B^{a}_{-i,\lambda}.
$
Combined with \eqref{eq:resfe} we obtain
\begin{align*}
x^+_{ir}x^-_{is}&=\Tr_G(\rho_{i,\lambda}^{-1}(x_i^rx_i^s)\rho_{i,\lambda})\\
&=x^-_{is}x^+_{ir}+\sum_{a=\lambda_i+1}^{r+s}(\lambda_i-a)B^{r+s-a}_{+i,\lambda}B^{a}_{-i,\lambda}+\sum_{k=0}^{\lambda_i-1}\sum_{a=0}^{\lambda_i-k-1}B^{r+s-a}_{+i,\lambda}B^{a}_{-i,\lambda}\\
&=x^-_{is}x^+_{ir}+\sum_{a=\lambda_i+1}^{r+s}(\lambda_i-a)B^{r+s-a}_{+i,\lambda}B^{a}_{-i,\lambda}
+\sum_{a=0}^{\lambda_i-1}(\lambda_i-a)B^{r+s-a}_{+i,\lambda}B^{a}_{-i,\lambda}\\
&=x^-_{is}x^+_{ir}+h_{i,r+s}.
\end{align*}

Now, assume $\lambda_i< 0$. Let $G=E_iF_i1_\lambda\oplus 1_\lambda^{\oplus (-\lambda_i)}$ and consider
the isomorphism $\rho_{i,\lambda}:F_iE_i1_\lambda\to G$.
By Lemma \ref{lem:traceop} we have 
$$x^-_{is}x^+_{ir}=\Tr_{F_iE_i}(x_i^sx_i^r)=\Tr_G(\rho_{i,\lambda}(x_i^sx_i^r)\rho_{i,\lambda}^{-1})$$
and it is equal to the sum of the trace of $\rho_{i,\lambda}(x_i^sx_i^r)\rho_{i,\lambda}^{-1}$ restricted to each direct factor of $G$. The restriction of $\rho_{i,\lambda}(x_i^sx_i^r)\rho_{i,\lambda}^{-1}$ to $E_iF_i1_\lambda$ is represented by
\begin{align*}
e(\alpha,i)\RL(\alpha+\alpha_i)e(\alpha,i)&\to e(\alpha,i)\RL(\alpha+\alpha_i)e(\alpha,i)\\
z&\mapsto \mu_{x_n^s\tau_nx_n^r}(\tilde z).
\end{align*}
By the relation $(e)$ of QHA and the definition of $\tilde z$ we have
\begin{align*}
\mu_{x_n^s\tau_nx_n^re(\alpha-\alpha_i,i^2)}(\tilde z)
&=x_{n+1}^rzx_{n+1}^s-\sum_{p=-\lambda_i}^{r+s-1}\mu_{x^p_n}(\tilde z)x^{r+s-1-p}_{n+1}.
\end{align*}
The restriction of $\rho_{i,\lambda}(x_i^sx_i^r)\rho_{i,\lambda}^{-1}$ to the $k$-th copy of $1_\lambda$ is represented by
\begin{align*}
\RL(\alpha)&\to \RL(\alpha),\quad a\mapsto a\mu_{x_n^{r+k+s}}(\tilde\pi_k).
\end{align*}
Now, relation $(b)$ follows from Lemma \ref{lem:recrec}$(b)$ and the fact that 
$$\Tr_{E_iF_i1_\lambda}\big((E_ix_i^{a})\circ\eta_{i,\lambda}'\circ\hat\vep'_{i,\lambda}\circ(E_ix_i^b)\big)=
\Tr_{1_\lambda}\big(\hat\vep'_{i,\lambda}(E_ix_i^{a+b})\circ\eta_{i,\lambda}'\big)$$
using similar computation as in the previous case. We leave the details to the reader.

Next, we prove $(c)$. Consider the formal series $X_j^-(w)=\sum_{s\geqslant 0}x_{js}^-w^s$.
From Proposition \ref{prop:A10} and Lemma \ref{lem:psi-B} we deduce that
for each $f\in R^\Lambda(\alpha)$ we have
\begin{align*}
&X_j^-(w)(\Tr(f))=
\Tr\Big(\sum_{s\geqslant 0}x_{n+1}^sw^sf\,e(\alpha,j)\Big),\\
&\Psi_i(z) X_j^-(w)\Psi_i(z)^{-1}(\Tr(f))=
\Tr\Big(q_{ij}(z,x_{n+1})\sum_{s\geqslant 0}x_{n+1}^sw^sf\,e(\alpha,j)\Big).
\end{align*}
This yields the first equation of $(c)$.
The second one is obtained in a similar way.

Let us now prove the relation $(d)$. Consider the endomorphism $x^r_ix^s_j\tau_{ji}\tau_{ij}$ on $E_iE_j$. We have
$x^r_ix^s_j\tau_{ji}\tau_{ij}=\tau_{ji}x^s_jx^r_i\tau_{ij}.$
Therefore 
$$\Tr_{E_iE_j}(x^r_ix^s_j\tau_{ji}\tau_{ij})=\Tr_{E_iE_j}(\tau_{ji}x^s_jx^r_i\tau_{ij})=\Tr_{E_jE_i}(x^s_jx^r_i\tau_{ij}\tau_{ji}).$$
Next, by relation $(d)$ in QHA we have
$$\tau_{ji}\tau_{ij}=Q_{ij}(x_i,x_j)=\sum_{p,q}c_{i,j,p,q}x_i^px_j^q=Q_{ji}(x_j,x_i)=\tau_{ij}\tau_{ji}.$$
Put this back into the equation above we get
$$\sum_{p,q}c_{i,j,p,q}\big(\Tr_{E_iE_j}(x^{r+p}_ix^{s+q}_j)-\Tr_{E_jE_i}(x^{s+q}_jx^{r+s}_i)\big)=0,$$
which is the relation $(d)$ .

Next, let us prove $(e)$. The functor $E_i^2$ acting on $\RL(\alpha)$ is represented by the bimodule $e(\alpha-2\alpha_i,i^2)\RL(\alpha)$. 
A morphism $x_i^rx_i^s$ on $E_i^2$ is represented by the left multiplication by $x_{n-1}^rx_n^s$ if $\alpha$ has height $n$. It is enough to consider the case $n=2$. 
The intertwiner $\varphi_{1}=(x_1-x_2)\tau_1+1$ is such that
$\varphi_1(x_1^rx_2^s)e(i^2)=(x_2^rx_1^s)\varphi_1e(i^2)$ and $\varphi_1^2e(i^2)=e(i^2)$.
It follows that $\Tr_{E_i^2}(x_i^rx_i^s)=\Tr_{E_i^2}(x_i^sx_i^r)$. Hence $x^+_{ir}x^+_{is}=x^+_{is}x^+_{ir}$.
The proof for $x^-$ is similar.

Finally let us prove the relation $(f)$. By $(e),$ it is equivalent to prove the following relation
$$\sum_{w\in S_m}[x^\pm_{i,r_{w(1)}},[x^\pm_{i,r_{w(2)}},\dots [x^\pm_{i,r_{w(m)}},x^\pm_{j,r_0}]\dots]]=0\ \text{with}\ i\neq j,\,r_p\in\bbN\ \text{and}\  m=1-a_{ij}.$$
We will prove it for $x^+$, the proof for $x^-$ is similar. First, recall that the functor $E^{(a)}_i$ is the image of the divided power operator 
$$e_{i,a}=x_1^{a-1}\cdots x_{a-1}\tau_{w_0} \in R(a\alpha_i)$$ by the canonical morphism $R(a\alpha_i)\to\End(E_i^a)$. 
The element $e_{i,a}$ is an idempotent, and we have $E_i^a\simeq (E_i^{(a)})^{\oplus a!}$. See, e.g., \cite{KL2} for details.

Given $i,j\in I$ with $i\neq j$, we write $m=1-a_{ij}$. 
In \cite[prop.~6]{KL2} an isomorphism of functors 
$$\alpha':\  \bigoplus_{a=0}^{[\frac{m}{2}]}E_i^{(2a)}E_jE_i^{(m-2a)}\simto  \bigoplus_{a=0}^{[\frac{m-1}{2}]}E_i^{(2a+1)}E_jE_i^{(m-2a-1)}$$
and a quasi-inverse $\alpha''$ are constructed.

If $a_{ij}=0$, we have $\alpha'=\tau_{ji}: E_jE_i\to E_iE_j$ and $\alpha''=c_{ij,0,0}^{-1}\tau_{ij}$. Since $i\neq j$, we have $\alpha'(x_j^sx^r_i) =(x_i^rx_j^s)\alpha'$.
Hence
$$x^+_{j,s}x^+_{i,r}=\Tr_{E_jE_i}(x_j^sx_i^s)=\Tr_{E_iE_j}(\alpha''(x^r_ix_j^s)\alpha')=\Tr_{E_iE_j}(x^r_ix_j^s)=x^+_{i,r}x^+_{j,s},$$
yielding relation $(f)$ in this case.

Assume now $a_{ij}<0$. Then the maps $\alpha'$, $\alpha''$ are given by
\begin{align*}
\alpha'&=\sum_{a=0}^{[\frac{m-1}{2}]}\alpha^+_{(2a,m-2a)}+\sum_{a=0}^{[\frac{m}{2}]}\alpha^-_{(2a,m-2a)},\\
\alpha''&=\sum_{a=0}^{[\frac{m}{2}]}\alpha^-_{(2a+1,m-1-2a)}-\sum_{a=0}^{[\frac{m-1}{2}]}\alpha^-_{(2a+1,m-1-2a)},
\end{align*}
where for $a\in [0,m]$ and $b=m-a$, we have
\begin{align*}
\alpha^+_{a,b}&=(e_{i,a+1}E_je_{i,b-1})\circ\tau_{a+1}\circ\tau_{a+2}\cdots\circ\tau_{a+b}\circ\iota_a:\  E_i^{(a)}E_jE_i^{(b)}\to E_i^{(a+1)}E_jE_i^{(b-1)},\\
\alpha^-_{a,b}&=(e_{i,a-1}E_je_{i,b+1})\circ\tau_{a}\circ\tau_{a-1}\cdots\circ\tau_{1}\circ\iota_a:\quad\quad  E_i^{(a)}E_jE_i^{(b)}\to E_i^{(a-1)}E_jE_i^{(b+1)}.
\end{align*}
Here $\iota_a: E_i^{(a)}E_jE_i^{(b)}\to E_i^{a}E_jE_i^{b}$ is the canonical embedding, and $\tau_{k}$ acts by $\tau$ on the $k$-th and $k+1$-th copy of $E$ in the sequence $E_i^{a}E_jE_i^{b}$.

\smallskip

Given any integers $r_1,\dots,r_m,s\in\bbN$, we define for each $a\in [0,m]$ a morphism
\begin{equation}
\Xi_a=\sum_{w\in\frakS_m}x_{i}^{r_{w(1)}}\cdots x_{i}^{r_{w(a)}}x_{j}^{s}x_{i}^{r_{w(a+1)}}\cdots x_{i}^{r_{w(m)}}\in\End(E_i^aE_jE_i^b)
\end{equation}
Note that $\Xi_a$ is symmetric in the first $a$-tuple of $x_i$'s and also in the last $b$-tuple of $x_i$'s, hence it commutes with the divided power operator $e_{i,a}E_je_{i,b}$. 
By consequence $\Xi_a$ restricts to a well defined endomorphism of $E_i^{(a)}E_jE_i^{(b)}$, which we denote again by $\Xi_a$. 
Further, since $E_i^aE_jE_i^b\simeq (E_i^{(a)}E_jE_i^{(b)})^{\oplus a!b!}$, we have
$\Tr_{E_i^aE_jE_i^b}(\Xi_a)=(a!b!)\Tr_{E_i^{(a)}E_jE_i^{(b)}}(\Xi_a).$

\begin{claim} We have
\begin{itemize}
\item[$(a)$] $\sum_{w\in\frakS_m}[x^+_{i,r_{w(1)}},[\dots,[x^+_{i,r_{w(m)}},x^+_{j,s}]\dots]]=m!\sum_{a+b=m}(-1)^b\Tr_{E_i^{(a)}E_jE_i^{(b)}}(\Xi_a),$

\item[$(b)$] $\alpha^+_{a,b}\Xi_a=\Xi_{a+1}\alpha^+_{a,b}$,  $\alpha^-_{a,b}\Xi_a=\Xi_{a-1}\alpha^-_{a,b}$.
\end{itemize}
\end{claim}

\smallskip

Let us show how to deduce the relation $(f)$ from this claim. Part $(b)$ implies that the following diagram commute
\[\xymatrix{\bigoplus_{a=0}^{[\frac{m}{2}]}E_i^{(2a)}E_jE_i^{(m-2a)}\ar[r]^{\alpha'\quad}_{\sim\quad} \ar[d]_{\bigoplus_a\Xi_{2a}} &\bigoplus_{a=0}^{[\frac{m-1}{2}]}E_i^{(2a+1)}E_jE_i^{(m-2a-1)}\ar[d]^{\bigoplus_a\Xi_{2a+1}}\\
\bigoplus_{a=0}^{[\frac{m}{2}]}E_i^{(2a)}E_jE_i^{(m-2a)}\ar[r]^{\alpha'\quad}_{\sim\quad}  &\bigoplus_{a=0}^{[\frac{m-1}{2}]}E_i^{(2a+1)}E_jE_i^{(m-2a-1)}.}\]
Therefore
$$\sum_{a=0}^{[\frac{m}{2}]}\Tr_{E_i^{(2a)}E_jE_i^{(m-2a)}}(\Xi_{2a})=
\sum_{a=0}^{[\frac{m-1}{2}]}\Tr_{E_i^{(2a+1)}E_jE_i^{(m-2a-1)}}(\Xi_{2a+1}).$$
Hence by part $(a)$ of the claim we get $\sum_{w\in\frakS_m}[x^+_{i,r_{w(1)}},[\dots,[x^+_{i,r_{w(m)}},x^+_{j,s}]\dots]]=0$ as desired.

\smallskip

It remains to prove the claim. Let us introduce some more notation. For $a\in [0,m]$ let 
$\Gamma^a=\{\underline{k}=(k_1,\dots,k_a)\in\bbN^a\,|\,1\leqslant k_1<k_2<\dots <k_a\leqslant m\}$, and for $\underline{k}\in\Gamma^a$ we set $\underline{k}^\circ=(l_1,\dots,l_b)\in\Gamma^b$ such that $\{k_1,\dots,k_a\}\cup\{l_1,\dots,l_b\}=\{r_1,\dots,r_m\}$.
Let $\frakS^{(a,b)}$ be the set of minimal representatives of the left cosets $\frakS_m/\frakS_a\times\frakS_b$. Then the map $w\mapsto w(1,2,\dots,m)$ yields a bijection 
$\frakS^{(a,b)}\simeq \Gamma^a$. Let $w_b$ be the longest element in $\frakS_b$. Given any sequence $r_1,\dots,r_m$ and $s\in\bbN$ we have
\begin{align*}
[x^+_{i,r_1},[x^+_{i,r_2},\dots,[x^+_{i,r_m}, x^+_{j,s}]\dots]]&=\sum_{a=0}^m(-1)^b\sum_{\underline{k}\in\Gamma^a,\,\underline{l}=
\underline{k}^\circ}x^+_{i,r_{k_1}}\dots x^+_{i,r_{k_a}}x^+_{j,s}x^+_{i,r_{l_b}}\dots x^+_{i,r_{l_1}}\\
&=\sum_{a=0}^{m}(-1)^b\sum_{y\in\frakS^{(a,b)}w_b}x^+_{i,r_{y(1)}}\cdots x^+_{i,r_{y(a)}}x^+_{j,s}x^+_{i,r_{y(a+1)}}\cdots x^+_{i,r_{y(m)}}.
\end{align*}
It follows that
\begin{align*}
&\sum_{w\in\frakS_m}[x^+_{i,r_{w(1)}},[x^+_{i,r_{w(2)}},\dots,[x^+_{i,r_{w(m)}}, x^+_{j,s}]\dots]]=\\
=&
\sum_{a=0}^{m}(-1)^b\sum_{y\in\frakS^{(a,b)}w_b}\sum_{w\in\frakS_m}x^+_{i,r_{wy(1)}}\cdots x^+_{i,r_{wy(a)}}x^+_{j,s}x^+_{i,r_{wy(a+1)}}\cdots x^+_{i,r_{wy(m)}}\\
=&
\sum_{a=0}^{m}(-1)^b\frac{m!}{a!b!}\sum_{w\in\frakS_m}x^+_{i,r_{w(1)}}\cdots x^+_{i,r_{w(a)}}x^+_{j,s}x^+_{i,r_{w(a+1)}}\cdots x^+_{i,r_{w(m)}}\\
=&
\sum_{a=0}^{m}(-1)^b\frac{m!}{a!b!}\Tr_{E_i^aE_jE_i^b}(\Xi_a)\\
=&
m!\sum_{a=0}^{m}(-1)^b\Tr_{E_i^{(a)}E_jE_i^{(b)}}(\Xi_a).\end{align*}
This proves part $(a)$ of the claim.

Part $(b)$ is a direct computation. On each $\RL(\alpha)$ the functor $E^{(a)}_iE_jE^{(b)}_i$ is represented by the $(\RL(\alpha-m\alpha_i-\alpha_j), \RL(\alpha))$-bimodule
$${}_{i^{(a)}ji^{(m-a)}}P=(1\otimes e_{i,a}\otimes 1\otimes e_{i,m-a})e(\alpha-m\alpha_i-\alpha_j,i^a,j,i^{m-a})\RL(\alpha).$$
To check the relation in $(b)$, without loss of generality we may assume $\alpha=m\alpha_i+\alpha_j$. Then $\Xi_a$ is represented by the left multiplication on ${}_{i^{(a)}ji^{(b)}}P$ by
$$\sum_{w\in\frakS_m}x_{1}^{r_{w(1)}}\cdots x_{a}^{r_{w(a)}}x_{a+1}^{s}x_{a+2}^{r_{w(a+1)}}\cdots x_{m+1}^{r_{w(m)}},$$
and $\alpha^+_{a,b}$ represented by $(1\otimes e_{i,a+1}\otimes 1\otimes e_{i,b-1})e(i^{a+1}ji^{b-1})\tau_{a+1}\tau_{a+2}\cdots \tau_{m}$.
Since $\Xi_a$ is symmetric in $\{x_{a+2},\dots ,x_{m+1}\}$, we have $e(i^aji^b)\tau_k\Xi_a=e(i^aji^b)\Xi_a\tau_k$ for any $k\in [a+2, m]$. 
Next,
\begin{align*}
e(i^{a+1}ji^{b-1})\tau_{a+1}\Xi_a&=\tau_{a+1}e(i^{a}ji^{b})\Xi_a\\
&=\sum_{w\in\frakS_m}x_{1}^{r_{w(1)}}\cdots x_{a}^{r_{w(a)}}\tau_{a+1}(x_{a+1}^{s}x_{a+2}^{r_{w(a+1)}})\cdots x_{m+1}^{r_{w(m)}}e(i^{a}ji^{b})\\
&=\sum_{w\in\frakS_m}x_{1}^{r_{w(1)}}\cdots x_{a}^{r_{w(a)}}(x_{a+2}^{s}x_{a+1}^{r_{w(a+1)}})\tau_{a+1}\cdots x_{m+1}^{r_{w(m)}}e(i^{a}ji^{b})\\
&=\Xi_{a+1}\tau_{a+1}e(i^{a}ji^{b})\\
&=\Xi_{a+1}e(i^{a+1}ji^{b-1})\tau_{a+1}.
\end{align*}
Finally, since $\Xi_{a+1}$ is symmetric in $\{x_1,\dots ,x_{a+1}\}$ and in $\{x_{a+3},\dots ,x_{a+b+1}\}$, 
the divided power operator $1\otimes e_{i,a+1}\otimes 1\otimes e_{i,b-1}$ commute with $\Xi_{a+1}$. To summarize, we have
\begin{align*}
\alpha^+_{a,b}\Xi_a&=(1\otimes e_{i,a+1}\otimes 1\otimes e_{i,b-1})e(i^{a+1}ji^{b-1})\tau_{a+1}\tau_{a+2}\cdots \tau_{a+b}\Xi_a\\
&=(1\otimes e_{i,a+1}\otimes 1\otimes e_{i,b-1})e(i^{a+1}ji^{b-1})\tau_{a+1}\Xi_a\tau_{a+2}\cdots \tau_{a+b}\\
&=(1\otimes e_{i,a+1}\otimes 1\otimes e_{i,b-1})\Xi_{a+1}e(i^{a+1}ji^{b-1})\tau_{a+1}\tau_{a+2}\cdots \tau_{a+b}\\
&=\Xi_{a+1}\alpha^+_{a,b},
\end{align*}
which is the first equality in part $(b)$, the proof for the second one is similar.

\end{proof}

\smallskip

Finally, we can deduce Theorem \ref{thm:action} for $\ell=1$ as a special case of Proposition \ref{prop:B5}.
Indeed, assume $\frakg$ is symmetric and \eqref{Q} holds.
Then, we have 
$$q_{ij}(u,v)=(1-uv)^{-a_{ij}},\qquad\forall i,j.$$
We must check that the relations $(c)$, $(d)$ in Proposition \ref{prop:B5} can be re-written as
\begin{itemize}
\item[$(c)$] $[h_{ir},x_{js}^\pm]=\pm a_{ij}\,x^\pm_{j,r+s}$,
\item[$(d)$] $\sum_{p=0}^{m}(-1)^p(\begin{smallmatrix}m\cr p
\end{smallmatrix})[x_{i,r+p}^\pm,x_{j,s+m-p}^\pm]=0$ with $i\neq j$ and $m=-a_{ij}$.
\end{itemize}
The relation $(d)$ is obvious, because by \eqref{Q} we have $c_{i,j,p,q}=\delta_{q,-a_{ij}-p}r_{ij}(-1)^{p}(\begin{smallmatrix}-a_{ij}\cr p
\end{smallmatrix})$ for all $p$, $q$ and $r_{ij}$ is invertible.
Let us prove $(c)$. Given $\alpha$ of height $n$ and $\lambda=\Lambda-\alpha$,
by Proposition \ref{prop:A10}  we have
$$B_{-i,\lambda}(z)=
\sum_{\nu\in I^\alpha}\prod_{p=1}^{\Lambda_i}(1+y_{ip}\,z)
\prod_{k=1}^n(1-zx_k)^{-a_{i\nu_k}}\,e(\nu).$$
By Lemma \ref{lem:psi-B} we have
\begin{align*}
\exp\Big(-\sum_{r\geqslant 1}h_{ir,\lambda}z^r/r\Big)=B_{-i,\lambda}(z).
\end{align*}
Since the $e(\nu)$'s are orthogonal idempotents in $\RL(\alpha)$, we have 
$$\exp\Big(-\sum_{r\geqslant 1}h_{ir,\lambda}z^r/r\Big)=\sum_{\nu\in I^\alpha}\exp\Big(-\sum_{r\geqslant 1}h_{ir,\lambda}e(\nu)z^r/r\Big)e(\nu).$$
Therefore, we have
\begin{align*}
\sum_{r\geqslant 1}h_{ir,\lambda}\,e(\nu)\,z^r&=z(\d/\d z)\log\Big(\prod_{p=1}^{\Lambda_i}(1+y_{ip}\,z)
\prod_{k=1}^n(1-zx_k)^{-a_{i\nu_k}}\Big)\,e(\nu)\\
h_{ir,\lambda}&=\Big(\sum_{p=1}^{\Lambda_i}(-y_{ip})^r-\sum_{k=1}^na_{i,\nu_k}x_k^r\Big)e(\nu),\quad\forall r\geqslant 1.
\end{align*}
In particular, we deduce that
$$h_{ir,\lambda-\alpha_j}e(\alpha,j)=(h_{ir,\lambda}-a_{ij}x_{n+1}^r)e(\alpha,j).$$
Hence for any $f\in\RL(\alpha)$ we have
\begin{align*}
h_{ir}x^-_{js}\Tr(f)&=\Tr(h_{ir,\lambda-\alpha_j}x_{n+1}^sfe(\alpha,j))\\
&=\Tr(h_{ir,\lambda}x_{n+1}^sfe(\alpha,j))-a_{ij}\Tr(x_{n+1}^{r+s}fe(\alpha,j)\\
&=x^-_{js}h_{ir}-a_{ij}x^-_{j,r+s}.
\end{align*}
The proof is complete.

\qed

\end{document}